\documentclass[12pt]{amsart}
\usepackage{amssymb}
\usepackage{amsmath,amscd}
\usepackage{amsfonts,latexsym,verbatim,amscd,mathrsfs,color,array}
\usepackage[all,cmtip]{xy}
\usepackage{mathrsfs}
\usepackage{multirow}
\usepackage{graphicx}
\usepackage{amsfonts, amsthm}
\usepackage{bbm}
\usepackage[hmargin=1in,vmargin=1in]{geometry}
\usepackage{mathdots}
\usepackage{stmaryrd}
\usepackage{MnSymbol}
\usepackage{bbm}
\usepackage[hidelinks]{hyperref}
\usepackage{enumitem}
\usepackage[all]{xy}
\usepackage{tikz-cd}
\usepackage{cancel}
\usepackage{tensor}

\allowdisplaybreaks

\def\XXint#1#2#3{{\setbox0=\hbox{$#1{#2#3}{\int}$ }
		\vcenter{\hbox{$#2#3$ }}\kern-.6\wd0}}

\renewcommand{\Im}{\operatorname{Im}}

\newcommand{\calM}{\mathcal{M}}

\newcommand{\transint}{\cap\kern-0.63em|\kern0.7em}

\DeclareMathSymbol{\intprod}{\mathbin}{MnSyC}{'270}

\newcommand{\LB}{\left[}
\newcommand{\RB}{\right]}

\newcommand{\LA}{\left\langle}
\newcommand{\RA}{\right\rangle}

\newcommand{\p}{{ \partial}}

\newcommand{\N}{{\mathbb N}}
\newcommand{\C}{{\mathbb C}}
\newcommand{\calL}{{\mathcal L}}

\newcommand{\R}{{\mathbb R}}

\newcommand{\Rm}{\mathrm{Rm}}
\newcommand{\Ric}{\mathrm{Ric}}

\renewcommand{\div}{{\operatorname{div} }}

\newcommand{\scrL}{{\mathscr L}}

\newcommand{\dist}{{\operatorname{dist}}}

\newcommand{\tr}{{\operatorname{tr}}}

\renewcommand{\p}{{\partial}}
\newcommand{\eps}{{\varepsilon}}

\newtheorem{thm}{Theorem}[section]
\newtheorem{lemma}[thm]{Lemma}
\newtheorem*{lemma*}{Lemma}
\newtheorem{prop}[thm]{Proposition}

\newtheorem{cor}[thm]{Corollary}

\newtheorem*{conj*}{Conjecture}

\newenvironment{claim}{\par\medskip\noindent\textit{Claim.}\space}{\par\medskip}
\newenvironment{claimproof}{\par\noindent\textit{Proof of claim.}\space}{\hfill$\diamond$\medskip\par}

   \newtheoremstyle{others}
     {3pt}
     {2pt}
     {}
     {}
     {\bf}
     {.}
     {.5em}
     {}

\theoremstyle{others}
\newtheorem{rmk}[thm]{Remark}
\newtheorem*{rmk*}{Remark}
\newtheorem{defn}[thm]{Definition}

\newtheorem*{question*}{Question}

\numberwithin{equation}{section}

\title{Integrable deformations and stability of the Ricci flow}
\author{Maxwell Stolarski and Alex Waldron}
\date{\today}

\begin{document}

\begin{abstract}
We provide a comparatively simple proof of the dynamical stability of Ricci flow near a linearly stable Ricci-flat ALE metric with integrable deformations. 
Our proof relies on the equivalence between integrability and an ``almost-orthogonality'' property of the Ricci-DeTurck tensor, allowing us to analyze the latter directly. 
We obtain our main results in weighted H\"older spaces and then show how to recover the $L^p$-stability theorems of Deruelle-Kr\"oncke \cite{DeruelleKroncke20} and Kr\"oncke-Petersen \cite{KronckePetersen20}.
\end{abstract}

\maketitle

\tableofcontents

\thispagestyle{empty}

\section{Introduction}

A collection of Riemannian metrics $(g(t))_{t \in [0, T)}$ on a manifold $M$ is said to evolve by \emph{Ricci flow} if
    $$\frac{\partial g}{\partial t} = - 2 \Ric(g).$$
Ricci-flat metrics correspond to fixed points of the Ricci flow.
As with any dynamical system, it is natural to ask which Ricci-flat metrics are \emph{stable} fixed points for the Ricci flow.
In other words, do Ricci flows starting near a given Ricci-flat metric remain close to that 
metric for all time?

There is a collection of important and beautiful results on the stability of Ricci flows.
{\v S}e{\v s}um proved that compact, linearly stable, Ricci-flat manifolds with integrable deformations are dynamically stable under Ricci flow \cite{Sesum06} (see also earlier results for flat metrics \cite{GIK02}).
For such compact manifolds $(M, g_0)$, \cite{Sesum06} uses the integrability assumption to remove kernel directions (with respect to a Lichnerowicz Laplacian) for $g(t) - g_0$ along the flow.
After removing kernel directions, the dynamics are governed by strictly stable directions, which enables \cite{Sesum06} to deduce stability.
Cheeger-Tian previously applied a similar argument to prove the uniqueness of tangent cones for Ricci-flat manifolds \cite{CheegerTian94}, which was in turn inspired by related results in minimal surfaces \cite{AllardAlmgren81, Simon83}.
Later, in \cite{HM14}, Haslhofer-Buzano developed a new approach that removed the integrability assumption from \cite{Sesum06}.
The argument in \cite{HM14} is based on a \L ojasiewicz–Simon inequality for Perelman's $\lambda$-functional, and therefore relies heavily on the fact that $M$ is compact.
\cite{SunWang15} uses a similar approach with Perelman's $\mu$-functional to deduce a related stability result for Fano K{\"a}hler-Einstein metrics.

In the non-compact case, Perelman's functionals are no longer directly available (see however \cite{DeruelleOzuch20, DeruelleOzuch21}).
Despite this difficulty, Deruelle-Kr{\" o}ncke proved that linearly stable Ricci-flat ALE manifolds $(M,g_0)$ with integrable deformations are dynamically stable in $L^2 \cap L^\infty$ \cite{DeruelleKroncke20}.
Their argument relies on a choice of dynamically changing reference Ricci-flat metrics $h(t)$ for flows $g(t)$ near $g_0$, and $L^2$ estimates on the difference $g(t) - h(t)$ from these moving reference metrics.
Kr{\" o}ncke-Petersen improved this result to show that 
Ricci-flat ALE manifolds with a parallel spinor and integrable deformations are dynamically stable in $L^p \cap L^\infty$ for $1 < p < n$ \cite{KronckePetersen20}.
Their argument uses a dynamic decomposition of metrics near $g_0$ and heat-kernel estimates from \cite{KronckePetersen22} to obtain stability through a sophisticated fixed-point argument.
Note that both \cite{DeruelleKroncke20, KronckePetersen20} rely on constructing a family of time-dependent reference Ricci-flat metrics, which broadly resembles the method in \cite{Sesum06} of subtracting kernel directions.
Prior stability results were also obtained for Euclidean space and noncompact symmetric spaces in \cite{SSS08, Bamler15}.


In this paper, we develop an alternative approach that yields a new and relatively simple proof of stability in the ALE case. 

\begin{thm} \label{main thm weighted Schauder spaces}
    Let $(M^n, g_0)$ be a smooth, complete, Ricci-flat ALE manifold of dimension $n \geq 4.$
    Let $\gamma \in (0,1)$, $\mu \in \{n-1,n\},$ $x_* \in M$, and $ n- \mu + 1 < \ell < n-2.$ 
    For weighted H{\"o}lder spaces as in Definition \ref{defn weighted Holder}, assume
    \begin{enumerate}
        \item  \label{main thm weighted Schauder spaces, hypothesis stable}
        $g_0$ is linearly stable (Definition \ref{defn linearly stable}),
        \item \label{main thm weighted Schauder spaces, hypothesis 2}
        the Lichnerowicz Laplacian $L = L_{g_0}$ associated to $g_0$ satisfies 
            $$\ker_{L^2} L \subseteq C^{2, \gamma}_{-\mu}(S^2 M, g_0),$$
        and
        \item \label{main thm weighted Schauder spaces, hypothesis integrable}
        $g_0$ has integrable deformations in $C^{2, \gamma}_{-\ell} ( M, g_0)$ (Definition \ref{defn integrable ALE}).
    \end{enumerate}
    
    If $\|g(0) - g_0 \|_{C^{2, \gamma}_{-\ell}(M, g_0, x_*)} \leq \eps$ is sufficiently small,  
    then the (complete, bounded-curvature) Ricci-DeTurck flow $g(t)$ (with respect to $g_0$) starting from $g(0)$ is defined for all $t \in [0, \infty),$ with
    $$\sup_{t \in [0, \infty)} \|g(t) - g_0 \|_{C^{2, \gamma}_{-\ell}(M, g_0, x_*)} \lesssim \eps$$
    and 
        $$g(t) \xrightarrow[t\to \infty]{C^\infty(M, g_0)} g_\infty$$
    for some smooth, gauged Ricci-flat ALE metric $g_\infty$ on $M$.

    There further exists a smoothly convergent family $\theta_t$ of diffeomorphisms 
    such that $\theta_t^* g(t)$ solves the Ricci flow and $\theta_t^*g(t) \to g_{\infty}'$ in $C^\infty_{loc}$
    for a Ricci-flat metric $g'_\infty$ diffeomorphic to $g_\infty,$ with $\|g'_\infty - g_0 \|_{C^{0,\gamma}_{-\ell}} \lesssim \eps.$
\end{thm}

\begin{rmk}
    The hypotheses of Theorem \ref{main thm weighted Schauder spaces} hold in several important cases (see Subsections \ref{subsect Lich Laplacian and stability}--\ref{subsect integrability} for further details):

    \begin{itemize}


        \item The decay hypothesis (\ref{main thm weighted Schauder spaces, hypothesis 2}) is true with $\mu = n-1$  for general Ricci-flat ALE spaces \cite[Theorem 2.7, Remark 2.10]{DeruelleKroncke20} and with $\mu = n$  in the case that $n=4$ \cite{BKN89} or $(M^n, g_0)$ is spin with a nonzero parallel spinor \cite{KronckePetersen22}.

        \vspace{2mm}

        \vspace{2mm}

        \item If $(M^n, g_0)$ is spin and carries a nonzero parallel spinor, then the stability hypothesis (\ref{main thm weighted Schauder spaces, hypothesis stable}) is satisfied \cite{DaiWangWei05} and the decay hypothesis (\ref{main thm weighted Schauder spaces, hypothesis 2}) is satisfied with $\mu = n$ \cite{KronckePetersen22}.  
        The integrability hypothesis (\ref{main thm weighted Schauder spaces, hypothesis integrable}) also holds, unless perhaps the holonomy is $\text{Hol}(M, g_0) = \text{Spin}(7)$ \cite[Theorem 2.18]{DeruelleKroncke20}.
        In the case of $\text{Spin}(7)$ holonomy, integrability is nonetheless widely expected to be true.

        \begin{itemize}
            \item All known Ricci-flat ALE manifolds, aside from Euclidean space, have special holonomy 
            ($\text{Hol}(M, g_0) = \text{SU}(n/2)$, $\text{Sp}(n/4)$, or $\text{Spin}(7)$) and are therefore spin with a nonzero parallel spinor \cite[Remark 1.4]{KronckePetersen20}.
        \end{itemize}


        \vspace{2mm}

        \item If $(M^n, g_0)$ is K{\"a}hler of general dimension, i.e. a Calabi-Yau ALE space, then hypotheses (\ref{main thm weighted Schauder spaces, hypothesis stable})--(\ref{main thm weighted Schauder spaces, hypothesis integrable}) are satisfied with $\mu = n$ and any $\ell \in (1, n-2)$ \cite{DaiWangWei05, DeruelleKroncke20, KronckePetersen22}.
        Thus, for Calabi-Yau ALE spaces, Theorem \ref{main thm weighted Schauder spaces} applies for each $\ell \in (1, n-2).$
        \begin{itemize}
            \item In particular, if $(M^4, g_0)$ is a 4-dimensional Calabi-Yau ALE space, then Theorem \ref{main thm weighted Schauder spaces} applies for each $\ell \in (1, 2)$. 
            Note that all currently-known 4D Ricci-flat ALE spaces are Calabi-Yau.
            It is an open question whether there are 4D Ricci-flat ALE spaces which are not K{\"a}hler.
        \end{itemize}




    \end{itemize}

    \vspace{2mm}

\end{rmk}

Theorem \ref{main thm weighted Schauder spaces} is essentially a H{\"o}lder-space version of the $L^p$-stability theorems due to Deruelle-Kr{\"o}ncke \cite{DeruelleKroncke20} and Kr{\"o}ncke-Petersen \cite{KronckePetersen20}. Our proof, however, is original, and rests on the following observation. Assuming that $g_0$ has integrable deformations, all nearby metrics enjoy an \emph{almost-orthogonality} property: as $g$ tends to $g_0$, the kernel component of the Ricci-DeTurck tensor $\Ric(g) - \frac12 \mathcal L_{V(g, g_0)} (g)$ tends to zero relative to its overall norm---see Lemma \ref{lemma:almostorthog} and Theorem \ref{thm:ALEalmostorthog}.
Almost-orthogonality allows us to directly analyze the Ricci-DeTurck tensor along a Ricci-DeTurck flow near $g_0;$ 
it avoids the construction of a suitable gauge along the flow, Perelman's functionals, and the choice of dynamically changing Ricci-flat reference metrics. 
Moreover, this approach easily recovers {\v S}e{\v s}um's stability theorem \cite{Sesum06} in the compact case (see Theorem \ref{thm stability compact case}).

While Theorem \ref{main thm weighted Schauder spaces} is most naturally proved in weighted H{\"o}lder spaces, some additional integral estimates allow us to fully recover the $L^p \cap L^\infty$-stability results of \cite{DeruelleKroncke20, KronckePetersen20}. We assume for simplicity that the decay order is $\mu = n,$ although the proof also works for $\mu = n - 1$ and $p$ in an appropriate range. 
\begin{thm} \label{main thm L^p spaces}
    Let $(M^n, g_0)$ be a smooth, complete, Ricci-flat ALE manifold and fix $1 < p < n.$
    Assume that $g_0$ {satisfies the hypotheses of Theorem \ref{main thm weighted Schauder spaces} 
    with $\mu = n$ and for some $1 < \ell \leq \frac{n}{p}$ with $\ell < \frac{n}{2}.$}
     Given $\eps > 0,$ there exists $\delta > 0$ such that if $\|g(0) - g_0 \|_{L^p\cap L^\infty(M, g_0)} < \delta,$ 
    then the (complete, bounded curvature) Ricci-DeTurck flow $g(t)$ starting from $g(0)$ is defined for all $t \in [0, \infty)$, with $\|g(t) - g_0 \|_{L^p \cap L^\infty(M, g_0) } < \eps$ and 
        $$g(t) \xrightarrow[t\to \infty]{C^\infty(M, g_0)} g_\infty$$
    for some smooth, gauged Ricci-flat ALE metric $g_\infty$ on $M$. Moreover, $g(t) \to g_\infty$ in $L^p(M)$ as $t \to \infty.$

There further exists a smoothly convergent family $\theta_t$ of diffeomorphisms such that $\theta_t^* g(t)$ solves the Ricci flow and $\theta_t^*g(t) \to g_{\infty}'$ in $C^\infty_{loc}$ for a Ricci-flat metric $g'_\infty$ diffeomorphic to $g_\infty.$
\end{thm}

We begin in Section \ref{sec: almost-orthogonality and exponential convergence} by proving an abstract almost-orthogonality lemma and showing that it implies an exponential convergence result for associated evolution equations.
After reviewing some preliminaries in Section \ref{sec: preliminaries}, Section \ref{s:exponentialconvergence} applies the general almost-orthogonality result to prove the stability of linearly stable Ricci-flat metrics with integrable deformations on closed manifolds (Theorem \ref{thm stability compact case}).
In Section \ref{sec: almost-orthogonality ALE case}, we begin our treatment of ALE metrics and establish a \emph{quantitative} almost-orthogonality property (Theorem \ref{thm:ALEalmostorthog}) for these non-compact spaces.
Sections \ref{sec: parabolic estimates}--\ref{sec: weighted Holder stability} exploit this quantitative almost-orthogonality to prove stability in weighted H{\"o}lder spaces (Theorem \ref{main thm weighted Schauder spaces}).
Finally, 
the appendices provide the additional (scalar) $L^p$-estimates required to prove the $L^p \cap L^\infty$-stability result (Theorem \ref{main thm L^p spaces}).

\subsection{Acknowledgments} The authors thank Anuk Dayaprema for reading over the appendices.
The authors also thank Tristan Ozuch, Alix Deruelle, and Qi Zhang for helpful conversations.

This work was supported by the first-listed author's Leverhulme Trust Early Career Fellowship (ECF 2023-182).


\vspace{10mm}

\section{Almost-orthogonality and exponential convergence
} \label{sec: almost-orthogonality and exponential convergence}

\subsection{Abstract almost-orthogonality}

Let $V$ and $W$ be Banach spaces and $\calM$ a $C^1$ map between open sets in $V$ and $W.$ Write
$$S = \{ \calM(u) = 0 \} \subset V.$$
Fix $u_0 \in S$ and let
$$L = d\calM_{u_0}: V \to W.$$ 
Define $V_0 = \ker L,$ which is closed, and suppose that there exists a 
complementary closed subspace $V_1 \subset V:$ 
$$V = V_0 \oplus V_1.$$
Let $\pi_0 : V= V_0 \oplus V_1 \to V_0$ denote the associated projection map.

\begin{defn}
The solution 
$u_0 \in S$ is said to have \emph{integrable deformations} in $V$ if there exists a $C^1$ map $\sigma$ from a neighborhood of the origin $0 \in V_0$ onto a neighborhood of $u_0$ in $S,$ such that $\sigma(0) = u_0$ and
$$ \left( \pi_0 \circ d\sigma \right)_{0} = \mathbf{1} : V_0 \to V_0.$$
\end{defn}

\noindent Suppose also that the image $W_1 : = \Im L$ is closed and complemented by a closed subspace $W_0 \subset W:$ 
$$W = W_0 \oplus W_1.$$
We decompose $\calM$ along the direct sum as follows:
\begin{equation}\label{Msplitting}
\calM(u) = \calM_0(u) \oplus \calM_1(u) \in W_0 \oplus W_1.
\end{equation}

\begin{lemma}[Almost-orthogonality]\label{lemma:almostorthog} The solution $u_0 \in S$ has integrable deformations if and only if, for each $\eps > 0,$ there exists $\delta > 0$ such that
\begin{equation}\label{almostorthog:est}
    | \calM_0(u) | \leq \eps | \calM(u) |
\end{equation}
for all $u \in B_{\delta}(u_0).$
\end{lemma}

The proof will be based on:
\begin{lemma}\label{lemma:fvanident}
    If $f(x,y)$ is $C^1$ with $f(x,0) \equiv 0$ and $df_{(0,0)} = 0$ then for each $\eps > 0$ there exists $\delta > 0$ such that
    \begin{equation}
        |f(x,y)| \leq \eps |y|
    \end{equation}
    for $|(x,y)| < \delta.$
\end{lemma}
\begin{proof}
    By the mean-value theorem, we have
    \begin{equation*}
    \begin{split}
    |f(x,y) - f(x,0)|\leq |y| \sup |df(x,t)| \leq \eps |y|
    \end{split}
    \end{equation*}
    for $(x,y)$ sufficiently close to $(0,0).$
\end{proof}

\begin{proof}[Proof of Lemma \ref{lemma:almostorthog}] Assume without loss of generality that $u_0 = 0.$ 
Let 
$$S_0 = \{ \calM_1(u) = 0 \}$$
and observe that $S \subset S_0$ by definition. The integrability assumption is that $S = S_0$ after restricting to a sufficiently small neighborhood.

We may further assume that $S_0$ is a neighborhood of the origin in $V_0,$ by the following argument. Write
$$u = (x,y) \in V_0 \oplus V_1.$$
Since $L : V_1 \stackrel{\sim}{\to} W_1$ is an isomorphism, the implicit function theorem provides $y_1(x),$ defined in a neighborhood of the origin in $V_0,$ such that $\calM_1(x,y_1(x)) = 0.$ After applying the coordinate change
\begin{equation}\label{almostorthog:coordchange}
(x,y) \mapsto \left( x, y + y_1(x) \right)
\end{equation}
on $V,$ we have $\calM_1(x,0) \equiv 0.$ Note that (\ref{almostorthog:coordchange}) does not change the splitting (\ref{Msplitting}), which is on $W$,
hence the conclusion is unaffected. Also, since $V_0 = \ker L,$ we have $\left(d y_1 \right)_{(0,0)} = 0,$ so $L$ is unchanged. 


\vspace{2mm}

\noindent ($\Rightarrow$) Suppose $S = S_0.$ Define $\delta_1(x,y) \in W_1$ by 
    $$\calM(u) = \calM(x,y) = \calM_0(x,y) + L(y) + \delta_1(x,y).$$
    By definition of the derivative, we have
    $$\left(d \calM_0 \right)_{(0,0)} + \left( d \delta_1 \right)_{(0,0)} = 0.$$
    Since the images of $\left(d \calM_0 \right)_{(0,0)}$ and $\left( d \delta_1 \right)_{(0,0)}$ lie in complementary subspaces, they must vanish individually.
    
    Also notice that $\calM(x, 0)$ and $\calM_0(x,0)$ both vanish identically, since $S = S_0$, hence the same is true of $\delta_1(x,0).$
    Applying Lemma \ref{lemma:fvanident} both to $\delta_1$ and to $\calM_0,$ we may conclude that
    \begin{equation}\label{almostorthog:deltaM0eps}
    |\delta_1(x,y)| + |\calM_0(x,y)| \leq \eps |y|
    \end{equation}
    for $|(x,y)| < \delta.$
    
    On the other hand, since $L : V_1 \stackrel{\sim}{\to} W_1$ is an isomorphism, the Bounded Inverse Theorem gives  
\begin{equation}\label{almostorthog:L(y)est}
    |L(y)| \geq c |y|
\end{equation}
for a constant $c > 0.$
    Applying (\ref{almostorthog:L(y)est}) and (\ref{almostorthog:deltaM0eps}) twice, we have
    \begin{equation}
        \begin{split}
           \eps |\calM(x,y)| & \geq \eps \left( |L(y)| - |\delta_1(x,y) + \calM_0(x,y)| \right) \\
            & \geq (c - \eps) \eps |y| \\
            & \geq (c - \eps)|\calM_0(x,y)|,
        \end{split}
    \end{equation}
    which implies the desired conclusion.

\vspace{2mm}

    \noindent ($\Leftarrow$) Suppose that the integrability assumption fails, i.e. $S \cap B_\delta(0) \subsetneq V_0 \cap B_\delta(0)$ for every $\delta > 0.$ Take $\eps = \frac12.$ For any $\delta > 0,$ there exists $x \in B_{\delta}(0) \cap V_0$ with $\calM(x,0) = \calM_0(x,0) \neq 0,$ so that
    $$|\calM_0(x,0) | > \frac12 |\calM_0(x,0)| = \frac12 |\calM(x,0)|.$$
    Hence (\ref{almostorthog:est}) fails.
\end{proof}

\vspace{10mm}

\subsection{Exponential convergence of 2nd-order parabolic evolution equations on compact manifolds}

Suppose now that 
$\calM$ is a 2nd-order quasilinear elliptic operator on sections of a vector bundle $E$ (with metric) over a closed Riemannian manifold, and consider the parabolic evolution equation
\begin{equation}\label{parabolicevolution}
\frac{ \p u}{\p t} = \mathcal{M}(u).
\end{equation}
Let $V \hookrightarrow W$ be Banach spaces of sections on which $\mathcal{M}$ defines a $C^1$ map between open subsets. Assume that $C^k \hookrightarrow W \hookrightarrow L^2$ for some $k \geq 0.$
Let $u_0 \in V$ be a solution of the equation $\calM(u) = 0,$ and define $L = d\mathcal{M}_{u_0}.$ Take
$$V_0 = \ker_V L, \qquad V_1 = V_0^\perp \cap V, \qquad W_1 = \Im L, \qquad W_0 = W_1^\perp \cap W,$$
where $V_0^\perp, W_1^\perp$ refer to $L^2$-orthogonal complements.

Assume further that
\begin{itemize}
\item (\ref{parabolicevolution}) is well-posed and smoothing on $V$, 



\item $L$ is self-adjoint (i.e. $\left( L u , v \right)_{L^2} = \left( u , Lv \right)_{L^2}$ for all smooth sections $u,v$),

\item 
$V = V_0 \oplus V_1$, and

\item 
$W = W_0 \oplus W_1.$
\end{itemize}
For example, if $V = C^{2,\alpha}$ and $W = C^\alpha,$ these are standard facts. 

\begin{thm}[Two-Interval Theorem\footnote{Cf. Leon Simon's Three-Interval Theorem \cite[(4.11-4.12), p. 546]{Simon83}.}]\label{thm:twointerval} Let $u_0$ be a solution of the equation $\calM(u) = 0$ as above, and assume that $u_0$ has integrable deformations in $V$ and is linearly stable.
Given $\lambda < \lambda_1,$ the first positive eigenvalue of $L,$ there exists $\delta > 0$ such that for any $u(0) \in B_\delta (u_0) \subset V,$ the corresponding solution $u(t)$ of (\ref{parabolicevolution}) satisfies 
\begin{equation}\label{twointerval:est}
\sup_{1 \leq t \leq 2} ||\calM(u(t))||_{L^2} \leq e^{ -\lambda} \sup_{0 \leq t \leq 1} ||\calM(u(t))||_{L^2}.
\end{equation}
\end{thm}
\begin{proof}

Suppose for contradiction that 
(\ref{twointerval:est}) fails for solutions $u_i(t)$ with $u_i(0) \in B_{1/i}(u_0).$ By well-posedness in $V$ and smoothing, we may assume that $u_i(t)$ are defined for $t \in \LB 0, 2\RB$ and converge in $C^\infty_{loc}\left( \left( 0 , 2 \RB \right)$ to $u_0,$ with $u_i(t) \to u_0$ in $V$ for each $t$ as $i \to \infty.$ 

Let $c_i \to \infty $ be such that
$$f_i(t) := c_i \calM(u_i(t))$$
satisfies
\begin{equation}\label{fiL2equalsone}
\sup_{1 \leq t \leq 2} ||f_i(t)||_{L^2} = 1.
\end{equation}
Since the estimate (\ref{twointerval:est}) fails, we in fact have
\begin{equation}\label{fiboundedinL2}
\sup_{0 \leq t \leq 2} || f_i(t) ||_{L^2} \leq e^{\lambda}.
\end{equation}
Let $k, \tau > 0.$
Since the $f_i$'s solve a linear parabolic equation
\begin{equation}\label{fievolution}
    \left( \frac{\p}{\p t} - d\calM_{u_i(t)} \right) f_i(t) = 0
\end{equation}
with uniformly smooth coefficients at positive times, we have
\begin{equation*}
\sup_{\tau \leq t \leq 2} \| f_i(t) \|_{C^k} \leq C_{k,\tau} \sup_{0 \leq t \leq 2} \| f_i(t) \|_{L^2}.
\end{equation*}
In particular,
\begin{equation}\label{L2controlsV}
\sup_{1 \leq t \leq 2} \| f_i(t) \|_{W} \leq C.
\end{equation}

By Arzela-Ascoli, we may pass to a subsequence such that $f_i(t) \to f_\infty(t)$ in $C^\infty_{loc}\left( \left( 0 , 2 \RB \right)$ and weakly in $L^2$ at $t = 0,$
so that the assumptions (\ref{fiL2equalsone}-\ref{fiboundedinL2}) are preserved in the limit. Since $u_i \to u_0$ as $i \to \infty,$ (\ref{fievolution}) converges to the fixed equation
$$\left( \frac{\partial}{\partial t} - L \right) f_\infty(t) = 0.$$
We now apply Lemma \ref{lemma:almostorthog}. Given $\eps > 0,$ for $i$ sufficiently large, at any $t \in \LB 1,2 \RB$ we have 
$$ \| \pi_0 f_i(t) \|_{L^2} \lesssim \| \pi_0 f_i(t) \|_W = c_i \| \calM_0(u_i(t)) \|_W \leq c_i \eps \|\calM(u_i(t)) \|_W = \eps \| f_i(t) \|_{W} \leq C \eps,$$
where we have applied (\ref{L2controlsV}). 
Since $\eps > 0$ was arbitrary, we conclude that
$$f_\infty(t) \perp \ker L $$
for $t \in \LB 1, 2 \RB.$ However, by the explicit formula for such a solution 
in terms of eigenfunctions of $L$ \cite[(1.20-21)]{Simon83}, for each $t \in \LB 1,2 \RB,$ we must have
$$||f_\infty(t)||_{L^2} \leq e^{-\lambda_1} ||f_\infty(t-1)||_{L^2} \leq e^{\lambda - \lambda_1} < 1$$
by (\ref{fiboundedinL2}). This is a contradiction.
\end{proof}

\begin{cor}\label{cor:exponentialconvergence}
    For any initial datum in a sufficiently small $V$-neighborhood of $u_0,$ the corresponding solution of (\ref{parabolicevolution}) exists for all time and converges exponentially.
\end{cor}

\begin{rmk}
    Below we will check the above hypotheses in the case of Ricci flow. The theorem can equally be used to show exponential convergence in the compact case for harmonic map flow near a linearly stable harmonic map with integrable deformations, for Yang-Mills flow near a linearly stable Yang-Mills connection with integrable deformations, \it{et cetera}.
\end{rmk}

\vspace{10mm}

\section{Geometric preliminaries} \label{sec: preliminaries}

\subsection{The Ricci-DeTurck operator and its linearization}\label{ss:riccideturckoperator}

Fix a reference metric $g_0$ on a manifold $M$ and define
\begin{equation}\label{Mdefn}
\begin{split}
     \calM(g) & := -2 \Ric(g) + \scrL_{V(g, g_0)} (g).
    \end{split}
    \end{equation}
Here $V(g, g_0)$ is the DeTurck vector field defined locally by
$$V^k = g^{ij} \left( \Gamma(g)^k_{ij} - \Gamma(g_0)^k_{ij} \right)$$
or globally by
$$g_0(V(g, g_0), \cdot) = - \div_{g} g_0 + \frac12 \nabla^{g} \tr_{g} g_0.$$
Let
$$\Delta_{g,g_0} = g^{ab} \nabla^{g_0,2}_{ab}.$$
By \cite[Lemma 2.1]{Shi89} (see also \cite[p.\ 5]{DeruelleKroncke20}),
we have
\begin{equation}\label{RHSofRdT}
\begin{split}
    \calM(g)_{ij} & = \Delta_{g,g_0} g_{ij} - g^{k\ell} g_{ip} (g_0)^{pq} \Rm(g_0)_{jk\ell q} - g^{k \ell} g_{jp} (g_0)^{pq} \Rm(g_0)_{ik\ell q} \\
   & + g^{ab} g^{pq} \left( \frac12 \nabla^{g_0}_i g_{pa} \nabla^{g_0}_j g_{qb} + \nabla^{g_0}_a g_{jp} \nabla^{g_0}_q g_{ib} \right) \\
   & - g^{ab} g^{pq} \left( \nabla^{g_0}_a g_{jp} \nabla^{g_0}_b g_{iq} - \nabla^{g_0}_j g_{pa} \nabla^{g_0}_b g_{iq} - \nabla^{g_0}_i g_{pa} \nabla^{g_0}_b g_{jq} \right).
    \end{split}
\end{equation}
Suppose that $g_0$ is Ricci-flat, and write $h = g - g_0.$ Observing that 
\begin{equation}\label{g0inversegrelation}
    g^{ab} = g_0^{ab} - g_0^{am} h_{mn} g^{nb},
    \end{equation}
    and using Ricci-flatness of $g_0,$ we have 
\begin{equation}\label{calMfullexpression}
\begin{split}
    \calM(g)_{ij} & = \Delta_{g,g_0} h_{ij} - h_{ab} g^{ka} (g_0)^{\ell b} g_{ip} (g_0)^{pq} \Rm(g_0)_{jk\ell q} - h_{ab} g^{k a} (g_0)^{\ell b} g_{jp} (g_0)^{pq} \Rm(g_0)_{ik\ell q} \\
   & + g^{ab} g^{pq} \left( \frac12 \nabla^{g_0}_i h_{pa} \nabla^{g_0}_j h_{qb} + \nabla^{g_0}_a h_{jp} \nabla^{g_0}_q h_{ib} \right) \\
   & - g^{ab} g^{pq} \left( \nabla^{g_0}_a h_{jp} \nabla^{g_0}_b h_{iq} - \nabla^{g_0}_j h_{pa} \nabla^{g_0}_b h_{iq} - \nabla^{g_0}_i h_{pa} \nabla^{g_0}_b h_{jq} \right).
    \end{split}
\end{equation}
\noindent The linearization of $\calM$ at $g$ comes out to
\begin{equation}\label{dMcompactexpression}
\begin{split}
d\calM_{g}(\hat{h})_{ij} & = \Delta_{g,g_0} \hat{h}_{ij} -2 g_0^{ka}g_0^{\ell b} \Rm(g_0)_{ik\ell j} \hat{h}_{ab} + F_{ij}{}^{cab} \nabla^{g_0}_c \hat{h}_{ab} + G_{ij}{}^{ab} \hat{h}_{ab},
\end{split}
\end{equation}
where
\begin{equation}\label{Fexpression}
\begin{split}
F_{ij}{}^{cmn} 
    & = g^{mq} g^{nb} \left( \delta_i{}^c  \left( \frac12 \nabla^{g_0}_j h_{qb} - \nabla^{g_0}_b h_{jq} \right) + \delta_j{}^c \left( \frac12 \nabla^{g_0}_i h_{bq} - \nabla^{g_0}_b h_{iq} \right) \right) \\
   & + g^{bc} g^{nq} \left( \delta_j{}^m \left( \nabla^{g_0}_q h_{ib} - \nabla^{g_0}_b h_{iq} - \nabla^{g_0}_i h_{qb} \right) + \delta_i{}^m \left( \nabla^{g_0}_q h_{jb} - \nabla^{g_0}_b h_{jq} - \nabla^{g_0}_j h_{qb} \right) \right) 
\end{split}
\end{equation}
and
\begin{equation}\label{Gexpression}
\begin{split}
G_{ij}{}^{mn} & = - g^{am} g^{nb} \nabla^{g_0, 2}_{a,b} h_{ij} \\
& - g_0^{km} g_0^{\ell n} g_0^{pq} \left( h_{ip}   \Rm(g_0)_{jk\ell q} + h_{jp}  \Rm(g_0)_{ik\ell q} \right) \\
& + h_{ab} \left( \begin{split} & \left( g_0^{ka}  g^{bm} g_0^{\ell n} + g^{km} g^{na} g_0^{\ell b} \right) g_0^{pq} \left( g_{ip} \Rm(g_0)_{jk\ell q} + g_{jp} \Rm(g_0)_{ik\ell q} \right) \\
& - g^{ka} g_0^{\ell b} g_0^{nq} \left( \delta_i{}^m \Rm(g_0)_{jk\ell q} + \delta_j{}^m \Rm(g_0)_{ik\ell q} \right) \end{split} \right) \\
   & - 2 g^{am} g^{nb} g^{pq} \left( \begin{split} \frac12 \nabla^{g_0}_i h_{pa} \nabla^{g_0}_j h_{qb} + \nabla^{g_0}_a h_{jp} \nabla^{g_0}_q h_{ib} - \nabla^{g_0}_a h_{jp} \nabla^{g_0}_b h_{iq} \\ - \nabla^{g_0}_j h_{pa} \nabla^{g_0}_b h_{iq} - \nabla^{g_0}_i h_{pa} \nabla^{g_0}_b h_{jq} \end{split} \right).
\end{split}
\end{equation}
These exact expressions are not crucial for the present paper but have been included for the sake of concreteness.
When $g = g_0,$ we have $F = G = 0$ and (\ref{dMcompactexpression}) reduces to the Lichnerowicz operator $L_{g_0}(\hat{h}).$

The Ricci-DeTurck flow (with respect to reference metric $g_0$) is given by
\begin{equation} \label{eqn Ricci-DeTurck flow}
\frac{\partial g}{\partial t} = \calM(g).
\end{equation}
Notice that along Ricci-DeTurck flow, we have
\begin{equation}\label{Mevolution}
\begin{split}
    \frac{\p}{\p t } \mathcal{M}(g(t))_{ij} & = d \mathcal{M}_{g(t)} \left( \frac{\p  g}{\p t} \right)_{ij} \\
    & = \Delta_{g,g_0} \mathcal{M}_{ij} -2 g_0^{ka}g_0^{\ell b} \Rm(g_0)_{ik\ell j} \mathcal{M}_{ab} + F_{ij}{}^{cab} \nabla^{g_0}_c \mathcal{M}_{ab} + G_{ij}{}^{ab} \mathcal{M}_{ab}.
    \end{split}
\end{equation}

\vspace{2mm}

\subsection{Preliminaries on ALE manifolds}

\begin{defn} \label{defn: ALE}
    A complete, connected Riemannian manifold $(M^n, g_0)$ is \emph{asymptotically locally Euclidean (ALE) of order $\tau > 0$} if there exists $K \Subset M$, a finite subgroup $\Gamma < SO(n)$ acting freely on $\R^n \setminus \{ 0\}$, and a diffeomorphism 
        $$ \Psi : ( \R^n \setminus \overline B_1) / \Gamma  \to M  \setminus K $$
    such that
        $$| \nabla^k_{g_{Euc}} ( \Psi^* g_0 - g_{Euc} ) |_{g_{Euc}} = O ( r^{-\tau - k} ) \qquad \forall k \in \mathbb N.$$
    Here $r = |x|$ is the radial coordinate on $\R^n$.
        
    Such a diffeomorphism $\Psi : (\R^n \setminus \overline B_1)/ \Gamma \to M \setminus K $ will be called a \emph{coordinate system at infinity}.

    A triple $(M^n, g_0, x_*)$ consisting of an ALE manifold $(M^n, g_0)$ and a point $x_* \in M$ will be called a \emph{pointed ALE manifold}.
    We use $\rho : M \to [0, \infty)$ to denote the regularized distance function $\rho(x) := \sqrt { 1 + \dist_{g_0}^2 ( x, x_*)} $.
\end{defn}

Throughout the remainder of this subsection unless otherwise noted, all raising/lowering of indices, norms $| \cdot | = |\cdot |_{g_0}$, and all covariant derivatives $\nabla = \nabla^{g_0}$ are with respect to $g_0$, where $(M^n, g_0)$ is an ALE manifold.

\begin{lemma} \label{lemma: Rm est for ALE}
    Let $(M^n, g_0, x_*)$ be a pointed ALE manifold of order $\tau > 0$ and $\rho(x) := \sqrt{1 + \dist_{g_0}^2(x, x_*)}$.
    For all $k \in \mathbb N$, there exists $C_k = C_k(k, M^n, g_0, x_*, \tau) >0 $ such that 
        $$|\nabla^k Rm | \le C_k \rho^{-\tau -2-k} \qquad \text{on } M.$$
\end{lemma}

\begin{defn} \label{defn weighted Holder}
    Let $(M^n, g_0, x_*)$ be a pointed ALE manifold and let $E$ be a tensor bundle over $M$.
    Let $\rho(x) := \sqrt{1 + \dist_{g_0}^2(x, x_*)}.$ 
    For $0 < \gamma < 1$ and $\delta \in \R$, define a weighted H{\"o}lder norm $C^{0, \gamma}_{-\delta}$ on sections $h$ of $E$ by 
\begin{equation} \label{eqn defn weighted Holder norm}
\begin{split}
    \| h \|_{C^{0 , \gamma}_{-\delta} } := \sup_M \rho(x)^{\delta} | h | + \sup_{x\in M} \rho(x)^{\delta + \gamma} [h ]_{C^{0, \gamma}_* (B(x, \rho(x)/2))}, \text{ where }\\
 [h ]_{C^{0, \gamma}_*(B(x, \frac{\rho(x)}2))} :=  \sup_{\sigma:[0, l] \to B(x, \rho(x)/2)} \frac{|P_l h(\sigma(0)) - h(\sigma(l)) |}{l^\gamma},
    \end{split}
\end{equation}
the supremum is taken over all unit-speed $g_0$-geodesics $\sigma :[0, l] \to B(x, \rho(x)/2)$, and $P_l$ denotes parallel transport along $\sigma$ from $\sigma(0)$ to $\sigma(l)$.

Define weighted H{\"o}lder spaces $C^{k, \gamma}_{-\delta}(E, M, g_0, x_*)$ as the set of sections $h$ of $E$ for which the norm
\begin{gather} \label{eqn defn weighted Holder norm, higher order}
    \| h \|_{C^{k , \gamma}_{-\delta} } := \sum_{i=0}^k \sup_M \rho(x)^{\delta +i} |\nabla^i h | + \sup_{x\in M} \rho(x)^{\delta + k + \gamma} [ \nabla^{k} h ]_{C^{0, \gamma}_* (B(x, \rho(x)/2))}
\end{gather}
is finite.

When the tensor bundle $E$, base manifold $M$, or base point $x_* \in M$ is implicit from context, we shall simplify the notation and write simply $C^{k, \gamma}_{-\delta} (E, g_0)$ or $C^{k, \gamma}_{-\delta} (M, g_0)$ for example.
Note that different choices of base point $x_* \in M$ define equivalent $C^{k,\gamma}_{-\delta}$ norms.
\end{defn}

Observe that if $h \in C^{k,\gamma}_{-\delta}$ then for $i \leq k,$ $\nabla^{(i)} h \in C^{k - i,\gamma}_{-\delta - i}.$
We have the following simple interpolation inequality:


\begin{lemma}\label{lemma:interpolation} We have
\begin{equation}
    \|\nabla^{i}f \nabla^{j} g \|_{C^{0,\gamma}_{-\alpha -\beta - i - j}} \leq C \|f \|_{C^{i,\gamma}_{-\alpha }} \|g \|_{C^{j,\gamma}_{-\beta}}.
    \end{equation}
\end{lemma}
\begin{proof}
    This follows by writing
    $$\rho^{\alpha + \beta + i + j} \nabla^{i}f \nabla^{j} g = \left( \rho^{\alpha + i } \nabla^{i} f \right) \left( \rho^{\beta + j } \nabla^{j} g \right).$$
\end{proof}

\begin{lemma} \label{lemma: difference of Laplacians}
    Let $(M^n, g_0, x_*)$ be a pointed ALE manifold of order $\tau > 0$.
    Let $\Psi : ( \R^n \setminus \overline B_1) / \Gamma  \to M  \setminus K$ be a coordinate system at infinity and denote the inverse $\Phi  = \Psi^{-1}$.
    Let $S^2 M$ be the bundle of symmetric 2-tensors over $M $.

    Then there exists $C = C(M^n, g_0, x_*, \tau, \Psi) > 0$ such that for all sections $h: M \setminus K \to S^2 M,$
    \begin{equation} \label{eqn difference of Laplacians, pointwise est}
        | ( \Delta_{g_0} - \Delta_{\Phi^* g_{Euc}}) h |_{g_0} 
        \le C \left( \rho^{- \tau} |\nabla^2 _{g_0} h |_{g_0} + \rho^{- \tau-1} |\nabla _{g_0} h |_{g_0} + \rho^{- \tau-2} | h |_{g_0}  \right) 
        \quad \text{on }  M \setminus K.
    \end{equation}

    Moreover, for any $\gamma \in (0,1)$ and $\delta \ge 0$,
    there exists $C = C(M^n, g_0,x_*, \gamma, \delta, \tau, \Psi) > 0$ such that 
    \begin{equation} \label{eqn difference of Laplacians, Holder est}
        \| ( \Delta_{g_0} - \Delta_{\Phi^* g_{Euc}} ) h \|_{C^{0, \gamma}_{-\tau - \delta - 2} (S^2(M\setminus K ), g_0) } \le C \| h \|_{C^{2, \gamma}_{-\delta} (S^2M, g_0)}
        \quad \forall h \in C^{2, \gamma}_{-\delta}(S^2M , g_0). 
    \end{equation}

    Identical statements hold for $\R$-valued functions $u : M  \to \R$.
\end{lemma}
\begin{proof} 
    To simplify notation, denote the Riemannian metric $\overline g = \Phi^* g_{Euc}$ on $M \setminus K $ throughout the proof.
    Denote also
        $$\tilde g \doteqdot \overline g - g_0 \qquad \text{and} \qquad \tilde \Gamma \doteqdot \Gamma(\overline g) - \Gamma(g_0)$$
    where $\Gamma(\overline g), \Gamma(g_0)$ are the Christoffel symbols for $\overline g$ and $g_0$, respectively.
    Since $g_0$ is ALE of order $\tau > 0$, it follows that
    \begin{equation} \label{proof difference of Laplacians, eqn 1}
        |\tilde g|_{g_0} = O ( \rho^{-\tau} ) , \quad 
        |\tilde \Gamma|_{g_0} = O ( \rho^{- \tau-1}), \quad \text{and} \quad 
        |\nabla_{g_0} \tilde \Gamma |_{g_0} = O ( \rho^{-\tau - 2} ). 
    \end{equation}
    Observe that, for $C^2$-sections $h: M \to S^2 M$,
    \begin{align*}
        (\Delta_{g_0} - \Delta_{\overline g} )h 
        ={}& g_{0}^{-1} \nabla_{g_0} \nabla_{g_0} h - \overline g^{-1} \nabla_{\overline g} \nabla_{\overline g} h \\
        ={}&  g_0^{-1} \nabla_{g_0} \nabla_{g_0}h
        - ( g_0 + \tilde g)^{-1} \left( \nabla_{g_0} + \tilde \Gamma * \right) \left( \nabla_{g_0} + \tilde \Gamma * \right) h \\
        ={}& \left[ g_0^{-1} - (g_0 + \tilde g)^{-1} \right] \nabla_{g_0} \nabla_{g_0} h 
        + ( g_0 + \tilde g)^{-1} \tilde \Gamma * \nabla_{g_0} h 
        + ( g_0 + \tilde g)^{-1} (\nabla_{g_0} \tilde \Gamma) *  h\\
        &+ ( g_0 + \tilde g)^{-1} \tilde \Gamma * \tilde \Gamma*  h.
    \end{align*}
    By taking norms with respect to $g_0$ and using \eqref{proof difference of Laplacians, eqn 1}, it follows that 
    \begin{align*}
        | (\Delta_{g_0} - \Delta_{\overline g} )h |_{g_0} 
        \le{}& O( \rho^{-\tau} ) |\nabla^2_{g_0} h |_{g_0}
        + O (\rho^{-\tau-1} ) |\nabla_{g_0} h |_{g_0} 
        + O ( \rho^{- \tau-2}) | h |_{g_0} 
    \end{align*}
    Estimate \eqref{eqn difference of Laplacians, pointwise est} now follows.
    Multiplying both sides by $\rho^{\tau + \delta + 2}$ and taking a supremum over $M\setminus K$ then implies
        $$\| ( \Delta_{g_0} - \Delta_{\overline g} ) h \|_{C^{0}_{-\tau-\delta-2} ( S^2(M\setminus K) , g_0) } \le C \| h \|_{C^2_{-\delta}(S^2M, g_0)} $$
    for some $C= C(M^n, g_0, \delta, \tau, \Psi) > 0$.
    The proof of the $C^{0, \gamma}_{-\tau-\delta-2}$-estimates \eqref{eqn difference of Laplacians, Holder est} and for $\R$-valued functions $u$ follow similarly.
\end{proof}

Throughout the paper, we also make use of the Bishop-Gromov volume comparison estimate to estimate certain integrals on $M$.
\begin{prop}[Bishop-Gromov Volume Comparison]
    Let $\omega_n$ denote the volume of the unit ball in Euclidean space $\R^n$.
    If $(M^n, g_0)$ is a smooth, complete, connected Riemannian manifold with $Rc \ge 0 $, then, for any $x_* \in M$,
        $$r \mapsto \frac{Vol_{g_0} ( B_{g_0}(x_*, r) ) }{\omega_n r^n} \text{ is nonincreasing and } \lim_{r \searrow 0} \frac{Vol_{g_0} ( B_{g_0}(x_*, r) ) }{\omega_n r^n} = 1.$$
    Moreover, there exists a dimensional constant $C_n > 0$ such that for any $f : [0, \infty) \to \R$ we have
        $$\int_M f( d(x_*, x) ) dV_{g_0}(x) \le C_n \int_0^\infty f(r) r^{n-1}\, dr.$$
\end{prop}

\vspace{2mm}

\subsection{The Lichnerowicz Laplacian and stability} \label{subsect Lich Laplacian and stability}

\begin{defn}
    Let $(M, g_0)$ be a Riemannian manifold.
    The \emph{Lichnerowicz Laplacian} $L = L_{g_0}$ is the operator on symmetric 2-tensors given in local coordinates by
        $$Lh_{ij} = \Delta h_{ij} + 2 R_{ipjq} h^{pq} - R_{ip} h^p{}_j - R_{jp} h^p{}_i $$
    where the Laplacian $\Delta$, curvatures $Rm, Rc$, and raising and lowering of indices are all with respect to $g_0$.
    Note that, in the case $(M, g_0)$ is Ricci-flat, $Lh_{ij} = \Delta h_{ij} + 2 R_{ipjq} h^{pq}$.
\end{defn}

\begin{defn} \label{defn linearly stable}
    A Ricci-flat manifold $(M, g_0)$ is said to be \emph{linearly stable} if 
        $$\int_M \langle h, Lh \rangle_{g_0} dV_{g_0} \le 0 \qquad \forall  h \in C^\infty_c(S^2M, g_0) .$$
\end{defn}

\begin{defn}
    Let $(M^n, g_0)$ be a Ricci-flat ALE manifold.
    Observe that the Lichnerowicz Laplacian $L= L_{g_0}$ defines a bounded linear map
        $$L : C^{2, \gamma}_{-\delta}(S^2M, g_0) \to C^{0, \gamma}_{-\delta-2}(S^2 M, g_0)$$
    for any $\gamma \in (0,1)$ and $\delta \in \R$.
    We let $\ker_{C^{2, \gamma}_{-\delta}} L $ denote its kernel
        $$\ker_{C^{2, \gamma}_{-\delta}} L := \{ h \in C^{2, \gamma}_{-\delta}(S^2 M, g_0) \ | \ Lh = 0 \} .$$    

    We also define 
        $$\ker_{L^2} L := \{ h \in L^2(S^2M, g_0) \ | \ h \text{ is smooth and } Lh=0 \}.$$
\end{defn}

\begin{rmk}
    As a consequence of elliptic regularity estimates, any $h \in \ker_{C^{2, \gamma}_{-\delta}} L $ is also smooth 
    and the kernels $\ker_{C^{2, \gamma}_{-\delta}} L$ and $\ker_{L^2} L $ are always finite-dimensional \cite[Theorem 4.21, Corollary 4.23]{Marshall02}.
\end{rmk}

We collect a few elementary properties about the kernel(s) of the Lichnerowicz Laplacian.
\begin{lemma} \label{lem kernel properties}
    Let $(M^n, g_0)$ be a Ricci-flat ALE manifold, let $L= L_{g_0}$ denote its Lichnerowicz Laplacian, and let $\gamma \in (0,1)$.

    \begin{enumerate}

        \item \label{lem kernel properties, item 1}
        For all $\delta, \delta' \in \R$,
            $$\delta \le \delta' \quad \implies\quad  \ker_{C^{2, \gamma}_{-\delta'}} L \subseteq \ker_{C^{2, \gamma}_{-\delta}} L .$$

        \item \label{lem kernel properties, item 2}
        If $\ker_{L^2}L \subseteq C^{2, \gamma}_{-\mu} ( S^2 M, g_0) $ for some $\mu > \frac n2$, then
            $$\ker_{C^{2, \gamma}_{-\delta}} L = \ker_{L^2} L \qquad \forall \delta \in \left(\frac n2 , \mu \right].$$

        \item \label{lem kernel properties, item 3}
        We have
            $$\ker_{C^{2, \gamma}_{-\delta}} L = \ker_{C^{2, \gamma}_{-\delta'}} L \qquad \forall \delta, \delta' \in (0, n-2) .$$
    \end{enumerate}
\end{lemma}
\begin{proof}
    \eqref{lem kernel properties, item 1}:
    Clearly, $C^{2, \gamma}_{-\delta'} \subseteq C^{2, \gamma}_{-\delta}$ when $\delta \le \delta'$.
    Hence, $\delta \le \delta'$ implies
        $$\ker_{C^{2, \gamma}_{-\delta'} } L 
        = \{ h \in C^{2, \gamma}_{-\delta'} \ | \ L h = 0 \} \subseteq \{ h \in C^{2, \gamma}_{-\delta} \ | \ L h = 0 \} = \ker_{C^{2, \gamma}_{-\delta}} L.$$

    \eqref{lem kernel properties, item 2}:
    Note that $C^{2, \gamma}_{-\delta} (S^2 M, g_0) \subseteq L^2(S^2M, g_0) $ when $\delta > \frac n2$.
    Indeed, $\delta > \frac n2$ implies
        $$\int_M | h |_{g_0}^2 dV_{g_0} \le C \| h \|^2_{C^{2, \gamma}_{-\delta} } \int_1^\infty \rho^{-2 \delta + n-1} d \rho
        \le C \| h \|^2_{C^{2, \gamma}_{-\delta}}
        \qquad ( C = C(M^n, g_0, \delta) ).$$
    Thus, $C^{2, \gamma}_{-\delta} \subseteq L^2$ and elliptic regularity imply
    \begin{multline} \label{pf kernel properties, item 2, eqn 1}
        \ker_{C^{2, \gamma}_{-\delta} } L = \{ h \in C^{2, \gamma}_{-\delta} \ | \ L h = 0 \} \subseteq \{ h \in L^2 \ | \ h \text{ is smooth and } Lh = 0 \} = \ker_{L^2 } L  \\ \forall \delta > \frac n2.
    \end{multline}

    On the other hand, $\ker_{L^2} L \subseteq C^{2, \gamma}_{-\mu} $ implies
    \begin{equation} \label{pf kernel properties, item 2, eqn 2}
        \ker_{L^2} L = \{ h \in L^2 \ | \ h \text{ is smooth and } L h = 0 \} 
        \subseteq \{ h \in C^{2, \gamma}_{-\mu} \  | \ Lh = 0 \} = \ker_{C^{2, \gamma}_{-\mu}} L.  
    \end{equation}

    Combining item \eqref{lem kernel properties, item 1} with \eqref{pf kernel properties, item 2, eqn 1}--\eqref{pf kernel properties, item 2, eqn 2} then implies that, for all $\delta \in \left( \frac n2 ,\mu \right]$,
    $$\ker_{C^{2, \gamma}_{-\mu}} L \subseteq \ker_{C^{2, \gamma}_{-\delta}} L \subseteq \ker_{L^2 }L \subseteq \ker_{C^{2,\gamma}_{-\mu}} L .$$
    Thus, we have equality throughout and \eqref{lem kernel properties, item 2} is proven.

    \eqref{lem kernel properties, item 3}:
    $-\delta \in (2-n, 0)$ consists of non-exceptional values $-\delta$, 
    and so the Fredholm index of $L : C^{2, \gamma}_{-\delta} \to C^{0, \gamma}_{-\delta -2}$ is independent of $-\delta \in (2-n, 0)$ (see e.g. \cite[Corollary 4.22, Theorem 4.24, Corollary 4.27, and Chapter 5]{Marshall02}).
    One can then argue as in \cite[Lemma 7.1]{LM85} to deduce that, for any $\gamma \in (0,1)$, 
        $$\ker_{C^{2, \gamma}_{-\delta}} L = \ker_{C^{2, \gamma}_{-\delta'}} L \qquad \forall \delta, \delta' \in (0, n-2).$$
\end{proof}

\begin{prop} \label{prop decay rate of L2 kernel elements}
    Let $(M^n, g_0)$ be a Ricci-flat ALE manifold and $L = L_{g_0}$ its Lichnerowicz Laplacian.
    Then, for any $\gamma \in (0,1)$,
    \begin{equation} \label{eqn decay of L^2 kernel}
        \ker_{L^2} L \subseteq C^{2, \gamma}_{-n+1} ( S^2M,g_0).
    \end{equation}

    If additionally
        \begin{enumerate}
            \item $n=4$,
            \item $(M^n, g_0)$ is K{\"a}hler, or
            \item $(M^n, g_0)$ is spin and admits a nonzero parallel spinor,
        \end{enumerate}
    then, for any $\gamma \in (0,1)$,
    \begin{equation} \label{eqn decay of L^2 kernel+}
        \ker_{L^2} L \subseteq C^{2, \gamma}_{-n}(S^2M, g_0) .
    \end{equation}
\end{prop}
\begin{proof}
    The $-n+1$ decay rate \eqref{eqn decay of L^2 kernel} was proven in \cite[Theorem 2.7, Remark 2.10]{DeruelleKroncke20}.

    The improved decay rate when $n=4$ follows from \cite{BKN89}.
    For spin manifolds with parallel spinor, the improved decay rate was obtained in \cite{KronckePetersen22}.
    If $(M^n, g_0)$ is K{\"a}hler, then it has $SU\left(\frac n2\right)$ holonomy. It is therefore spin with a parallel spinor, and so the result follows from above.
\end{proof}

\begin{prop} \label{prop L2 kernel equals Holder kernel}
    Let $(M^n, g_0)$ be a Ricci-flat ALE manifold and $L = L_{g_0}$ its Lichnerowicz Laplacian.

    Then the following hold:
    \begin{enumerate}
        \item \label{prop L2 kernel equals Holder kernel, item 1}
        If $n \ge 5$, then
        $$\ker_{C^{2, \gamma}_{-\delta}} L = \ker_{L^2} L \qquad \forall \delta \in (0, n-1], \gamma \in (0,1).$$

        \item \label{prop L2 kernel equals Holder kernel, item 2}
        If $n \ge 5$ and $(M^n, g_0)$ is spin and admits a nonzero parallel spinor, then 
            $$\ker_{C^{2, \gamma}_{-\delta}} L = \ker_{L^2} L \qquad \forall \delta \in (0, n], \gamma \in (0,1).$$

        \item \label{prop L2 kernel equals Holder kernel, item 3}
        If $n = 4$, then
        $$\ker_{C^{2, \gamma}_{-\delta}} L = \ker_{L^2} L \qquad \forall \delta \in (0, 4], \gamma \in (0,1).$$
        
    \end{enumerate}
\end{prop}
\begin{proof}
    By Proposition \ref{prop decay rate of L2 kernel elements}, $\ker_{L^2} L \subseteq C^{2, \gamma}_{-n+1}$.
    If $n\ge 5$, then this fact can be combined with Lemma \ref{lem kernel properties} (\ref{lem kernel properties, item 2})--(\ref{lem kernel properties, item 3}) to deduce (\ref{prop L2 kernel equals Holder kernel, item 1}).

    When $n \ge 5$ and $(M^n, g_0)$ is spin and admits a nonzero parallel spinor, then Proposition \ref{prop decay rate of L2 kernel elements} yields $\ker_{L^2} L \subseteq C^{2, \gamma}_{-n}$ and a similar argument as above then proves (\ref{prop L2 kernel equals Holder kernel, item 2}).

    Finally, consider the case that $n=4$.
    Then Proposition \ref{prop decay rate of L2 kernel elements} gives $\ker_{L^2}L \subseteq C^{2, \gamma}_{-4}$, and Lemma \ref{lem kernel properties} (\ref{lem kernel properties, item 2}) then implies
    \begin{equation} \label{proof L2 kernel equals Holder kernel, eqn 1}
        \ker_{C^{2, \gamma}_{-\delta}} L = \ker_{L^2} L \qquad \forall \delta \in (2, 4].
    \end{equation}
    \cite[Lemma 4.2]{Ozuch22} furthermore implies 
        $$\ker_{C^{2, \gamma}_{-\delta}} L \subseteq \ker_{C^{2, \gamma}_{-4}} L \qquad \forall \delta \in (0, 4],$$
    which, combined with Lemma \ref{lem kernel properties} (\ref{lem kernel properties, item 1}) and \eqref{proof L2 kernel equals Holder kernel, eqn 1}, thereby proves (\ref{prop L2 kernel equals Holder kernel, item 3}).
\end{proof}

\subsection{Integrable deformations} \label{subsect integrability}

\begin{defn} \label{defn integrable closed mfld}
    Let $(M^n, g_0)$ be a closed Ricci-flat manifold, let $\gamma \in (0,1)$, and let
        $$\mathcal F := \{ g \in C^{2,\gamma}(S^2 M, g_0) \, | \, \mathcal M(g) = 0 \} .$$
    Consider the Lichnerowicz Laplacian $L = L_{g_0}$ as a bounded linear map $C^{2, \gamma}(S^2 M, g_0) \to C^{0, \gamma}(S^2 M, g_0)$, and denote its kernel $\ker_{C^{2, \gamma}} L$ in $C^{2, \gamma}(S^2 M, g_0)$.
    Denote the $L^2(S^2 M, g_0)$-orthogonal complement 
        $$V_1 := \left\{ h \in C^{2, \gamma}(M, g_0) \, \left| \, \int_M \langle h, e \rangle_{g_0} dV_{g_0} = 0 \quad \forall e \in \ker_{C^{2, \gamma}} L \right.  \right\} ,$$
    and let $\pi_0 : C^{2, \gamma} = \ker_{C^{2, \gamma} } L \oplus V_1 \to \ker_{C^{2, \gamma}}L$ be the associated projection map.

    We say $g_0$ has \emph{integrable deformations in $C^{2, \gamma}(M, g_0)$} if there exists a $C^1$-map $\sigma$ from a neighborhood of $0 \in \ker_{C^{2, \gamma}} L$ onto a neighborhood of $g_0 \in \mathcal F$ such that 
        $$\sigma(0) = g_0 \qquad \text{and} \qquad \left( \pi_0 \circ d \sigma \right)_0 = \mathbf 1 : \ker_{C^{2, \gamma}} L \to \ker_{C^{2, \gamma}}L.$$
\end{defn}

\begin{defn} \label{defn integrable ALE}
    Let $(M^n, g_0, x_*)$ be a pointed Ricci-flat ALE manifold. 
    Let $\gamma \in (0,1)$, $\delta \ge 0$, and
        $$\mathcal F := \{ g \in C^{2, \gamma}_{-\delta}(S^2M, g_0, x_*) \, | \, \mathcal M(g) = 0 \} \qquad (\text{with the } C^{2, \gamma}_{-\delta}  \text{-topology}).$$
    Consider the Lichnerowicz Laplacian $L = L_{g_0}$ as a map $C^{2, \gamma}_{-\delta} (M, g_0, x_*) \to C^{0, \gamma}_{-\delta-2} (M, g_0, x_*)$ and its kernel $\ker_{C^{2, \gamma}_{-\delta}} L $ in $C^{2, \gamma}_{-\delta}(M, g_0, x_*)$.
    Assume additionally that 
    \begin{equation} \label{eqn assumption for integrability defn}
        \ker_{C^{2, \gamma}_{-\delta}} L = \ker_{L^2} L \qquad \text{and} \qquad \ker_{L^2} L  \subseteq C^{2, \gamma}_{-\mu}(M, g_0, x_*) \qquad \text{for some } \mu > n-\delta,
    \end{equation}
    in which case $\langle h, e \rangle_{g_0}$ is in $L^1(M, g_0)$ for any $h \in C^{2, \gamma}_{-\delta}$ and any $e \in \ker_{C^{2, \gamma}_{-\delta}} L $ (see \eqref{orthogonal elliptic est proof, eqn 1} below).
    Thus, $C^{2, \gamma}_{-\delta}(S^2M , g_0, x_*) $ decomposes as
        $$C^{2, \gamma}_{-\delta}(S^2M , g_0, x_*) = \ker_{C^{2, \gamma}_{-\delta}} L \oplus \{ h \in C^{2, \gamma}_{-\delta}(S^2M, g_0, x_*) \, | \, ( h , e )_{L^2(M, g_0)} = 0 \quad \forall e \in \ker_{C^{2, \gamma}_{-\delta}} L \},$$
    and we can let $\pi_0: C^{2, \gamma}_{-\delta} \to \ker_{C^{2, \gamma}_{-\delta}} L$ denote the associated projection map.

    We say $g_0$ has \emph{integrable deformations in $C^{2, \gamma}_{-\delta}(M, g_0, x_*)$} if there exists a $C^1$-map $\sigma$ from a neighborhood of $0 \in \ker_{C^{2, \gamma}_{-\delta}} L$ onto a neighborhood of $g_0$ in $\mathcal F$ such that 
        $$\sigma(0) = g_0 \qquad \text{and} \qquad \left( \pi_0 \circ d \sigma \right)_0 = \mathbf 1 : \ker_{C^{2, \gamma}_{-\delta}} L \to \ker_{C^{2, \gamma}_{-\delta}} L.$$
\end{defn}

\bigskip

\begin{lemma} \label{lem integrability and changing decay rate}
    Let $(M^n, g_0, x_*)$ be a pointed Ricci-flat ALE manifold, $\gamma \in (0,1)$, $\mu \in \{ n-1, n\}$, and $n - \mu < \delta, \delta' \le n-1$.
    Assume
        $$\ker_{L^2} L \subseteq C^{2,\gamma}_{-\mu}.$$
    Then
        $$\ker_{C^{2, \gamma}_{-\delta}} L = \ker_{C^{2, \gamma}_{-\delta'}} L = \ker_{L^2} L,$$
    and
    \begin{equation*}
        g_0 \text{ has integrable deformations in } C^{2, \gamma}_{-\delta} 
        \iff g_0 \text{ has integrable deformations in } C^{2, \gamma}_{-\delta'} .
    \end{equation*}
\end{lemma}
\begin{proof}
    Since $n - \mu \ge 0$, Proposition \ref{prop L2 kernel equals Holder kernel} implies 
    $$\ker_{C^{2, \gamma}_{-\delta}} L = \ker_{C^{2, \gamma}_{-\delta'}} L = \ker_{L^2} L .$$
    In particular, condition \eqref{eqn assumption for integrability defn} in Definition \ref{defn integrable ALE} holds for both $\delta$ and $\delta'$.
    Since $\ker_{C^{2, \gamma}_{-\delta}} L = \ker_{C^{2, \gamma}_{-\delta'} } L $, it follows directly from Definition \ref{defn integrable ALE} that $g_0$ has integrable deformations in $C^{2, \gamma}_{-\delta}$ if and only if it has integrable deformations in $C^{2, \gamma}_{-\delta'}$.
\end{proof}

It's worth comparing the definition of integrable (Definition \ref{defn integrable ALE}) to the definition of integrable in \cite[Definition 2.4]{DeruelleKroncke20}.

\begin{lemma} \label{lem comparing integrability defns}
    Let $(M^n, g_0)$ be a Ricci-flat ALE manifold and $\gamma \in (0,1)$.
    Assume $(M^n, g_0)$ is integrable in the sense of \cite[Definition 2.4]{DeruelleKroncke20}.
    Then $(M^n, g_0)$ is integrable in $C^{2, \gamma}_{-\delta}$ for all $\delta \in (1, n-1]$.

    If additionally $n=4$ or $(M^n, g_0)$ is spin with a nonzero parallel spinor,
    then $(M^n, g_0)$ is integrable in $C^{2, \gamma}_{-\delta}$ for all $\delta \in (0, n]$.
\end{lemma}
\begin{proof}   
    In \cite[Definition 2.4]{DeruelleKroncke20}, a Ricci-flat ALE space $(M^n, g_0)$ is said to be integrable if, for the set
        $$\mathcal U_{g_0} := \{ g \in C^\infty \cap L^2 \cap L^\infty \, | \, \text{$g$ is an ALE metric on $M$ with } Rc_g = 0 \text{ and } \mathcal L_{V(g, g_0)} g = 0 \}$$ 
    with the $L^2 \cap L^\infty(M, g_0)$ topology,
    the $L^2(M,g_0)$-projection map to the kernel of the Lichnerowicz Laplacian $L = L_{g_0}$, that is,
        $$\Psi_{g_0} : \mathcal U_{g_0} \to \ker_{L^2} L , \qquad \Psi_{g_0} (g) := \pi_0 ( g - g_0)$$
    is a local diffeomorphism at $g_0$.
    In this case, the inverse function theorem implies there exists a locally defined smooth inverse $\sigma = \Psi_{g_0}^{-1}$ onto a neighborhood of $g_0 \in \mathcal U_{g_0}$ such that $\sigma(0) = g_0$ and 
        $$\pi_0 ( \sigma(h) - g_0) = h$$
    for all $h $ in a neighborhood of $0\in \ker_{L^2} L$.
    Differentiating at $0$ then gives $(\pi_0 \circ d\sigma)_0 = \mathbf 1$.
    
    Proposition \ref{prop L2 kernel equals Holder kernel} implies $\ker_{C^{2, \gamma}_{-\delta} } L = \ker_{L^2} L $ for all $\delta \in (0, n-1]$, and a neighborhood of $g_0 \in \mathcal U_{g_0}$ coincides with a neighborhood of 
        $$g_0 \in \{ g \in C^{2, \gamma}_{-\delta} (M, g_0) \, | \, \mathcal M(g) = 0\}.$$
    Proposition \ref{prop decay rate of L2 kernel elements} also implies $\ker_{L^2} L \subseteq C^{2, \gamma}_{-\mu}$ for $\mu = n-1$.
    It therefore follows that if $g_0$ is integrable in the sense of \cite[Definition 2.4]{DeruelleKroncke20}, then $g_0$ has integrable deformations in $C^{2, \gamma}_{-\delta}(M, g_0)$ in the sense of Definition \ref{defn integrable ALE} above for any $\delta \in (1, n-1]$.
    
    In the case that additionally $n=4$ or $(M^n,g_0)$ is spin with a nonzero parallel spinor, an analogous argument shows the same statement is true for any $\delta \in (0, n]$.
\end{proof}

\begin{prop} \label{prop Calabi-Yau implies integrable}
    Let $(M^n, g_0)$ be a Ricci-flat ALE manifold which is also K{\"a}hler.
    Then $g_0$ has integrable deformations in $C^{2, \gamma}_{-\delta}(M, g_0)$ for any $\delta \in (0, n]$ and $\gamma \in (0,1)$.
\end{prop}
\begin{proof}
    \cite[Theorem 2.18]{DeruelleKroncke20} implies such metrics are integrable in the sense of \cite[Definition 2.14]{DeruelleKroncke20}.
    Ricci-flat K{\"a}hler metrics are also spin with a nonzero parallel spinor (see the proof of Proposition \ref{prop decay rate of L2 kernel elements}).
    Lemma \ref{lem comparing integrability defns} therefore implies $g_0$ has integrable deformations in $C^{2, \gamma}_{-\delta}(M, g_0)$ for any $\delta \in (0, n]$ and $\gamma \in (0,1)$.
\end{proof}

\begin{rmk}
    If $(M^n, g_0)$ is a (non-flat) Ricci-flat ALE space which is spin with a nonzero parallel spinor, then $(M^n, g_0)$ has special holonomy $\text{Hol}(M, g_0) = \text{SU}(n/2)$, $\text{Sp}(n/4)$, or $\text{Spin}(7)$ \cite[Section 6]{KronckePetersen22}.
    In particular, if such an $(M^n, g_0)$ has $\text{Hol}(M, g_0) \ne \text{Spin}(7)$, then Proposition \ref{prop Calabi-Yau implies integrable} implies $(M^n, g_0)$ has integrable deformations in $C^{2, \gamma}_{-\delta}$ for any $\delta \in (0, n]$ and $\gamma \in (0,1)$.
\end{rmk}

\vspace{10mm}

\section{Exponential convergence of the Ricci flow (compact case)}\label{s:exponentialconvergence}

\begin{thm}[\cite{Sesum06}] \label{thm stability compact case}
    Ricci flow starting near a compact, linearly stable, Ricci-flat metric with integrable deformations converges exponentially.
\end{thm}
\begin{proof}
    Let $V = C^{2, \alpha}$ and $W = C^\alpha.$ The result for Ricci-DeTurck flow follows from Corollary \ref{cor:exponentialconvergence} and the calculations in \S \ref{ss:riccideturckoperator}. 
    
To obtain the result for the Ricci flow itself, let $g_\infty = \lim_{t \to \infty} g(t).$ Since $\calM(g_\infty) = 0,$ the argument e.g. of Deruelle-Kr{\"o}ncke \cite[Theorem 2.7]{DeruelleKroncke20} implies that $g_\infty$ is Ricci-flat and $V(g_\infty, g_0) = 0.$ Since $g(t) - g_\infty$ tends to zero exponentially, so does $V(g(t), g_0).$ Integrating, we obtain a convergent family $\varphi_t$ of diffeomorphisms such that $\varphi_t^* g(t)$ is an exponentially convergent solution of the Ricci flow.
\end{proof}

\vspace{10mm}

\section{Quantitative almost-orthogonality (ALE case)} \label{sec: almost-orthogonality ALE case}

\begin{lemma}\label{lemma:orthogonalellipticest1}
    Let $(M^n, g_0, x_*)$ be a pointed Ricci-flat ALE manifold. 
    Let $\mu \in \{n-1,n\},$ $\gamma \in (0,1)$, $\delta \in \left( n-\mu, n-2 \right),$ 
    and
        $$\mathcal F = \{ g \in C^{2, \gamma}_{-\delta} ( S^2 M , g_0) \, | \, \mathcal M (g ) = 0 \}.$$
            Assume 
that
$$\ker_{C^{2, \gamma}_{-\delta}} L = \ker_{L^2} L  \subseteq C^{2, \gamma}_{-\mu} (S^2M, g_0)$$ and
        $g_0$ has integrable deformations in $C^{2, \gamma}_{-\delta}(S^2M, g_0)$.
    Then there exist $C^{2,\gamma}_{-\delta}(S^2M, g_0)$-neighborhoods $U, U' \ni g_0$ such that any metric $g \in U$ can be written as
    \begin{equation}\label{orthogonalellipticest:orthogonality}
        g = g_1 + h_1,
    \end{equation}
    where $g_1 \in U',$ $\calM(g_1) = 0,$ and $h_1 \perp_{g_0} \ker_{L^2} L.$
    Moreover,
    \begin{equation} \label{lem orthogonal elliptic est, eqn comparable topologies}
        \| g_1 - g_0 \|_{C^{2, \gamma}_{-\delta}} + \| h_1 \|_{C^{2, \gamma}_{-\delta}} \to 0 \qquad \text{as } \| g - g_0 \|_{C^{2, \gamma}_{-\delta}} \to 0
    \end{equation}
    and there exists $C = C(M^n,g_0 , \gamma , \delta ) > 0$ such that 
    \begin{equation} \label{orthogonalellipticest:est}
        \| h_1 \|_{C^{2,\gamma}_{-\delta}(S^2M, g_0)} \leq C \|L h_1 \|_{C^{0,\gamma}_{-\delta - 2}(S^2M, g_0)}.
    \end{equation}

    
\end{lemma}
\begin{proof}
    By the integrability assumption, $\mathcal F \subset C^{2, \gamma}_{-\delta}$ is a $C^1$-submanifold in a neighborhood $g_0$ with $T_{g_0} \mathcal F = \ker_{C^{2, \gamma}_{-\delta}} L$.
    Any $e \in T_{g_0} \mathcal F = \ker_{C^{2, \gamma}_{-\delta}} L =  \ker_{L^2}L$ lies in $C^{2, \gamma}_{-\mu}(S^2M, g_0).$ 
    In particular, since $n - \mu - \delta < 0,$ for any $h \in C^{2, \gamma}_{-\delta}(S^2M, g_0)$, and any $e \in T_{g_0} \mathcal F$, we have
    \begin{equation} \label{orthogonal elliptic est proof, eqn 1}
        \int_M \langle h, e \rangle d V
        \le C \| h \|_{C^{0}_{-\delta}} \| e \|_{C^0_{-\mu}} \int_0^\infty \rho^{-\delta} \rho^{-\mu} \rho^{n-1} d \rho < \infty.
    \end{equation}
    We may therefore define
        $$\mathcal X : = \left\{ h \in C^{2, \gamma}_{-\delta} ( S^2M, g_0) \, | \, \int_M \langle h, e \rangle_{g_0} d V_{g_0} = 0 \text{ for all } e \in \ker_{L^2} L \right\} .$$
    Consider the map
        $$\Phi : \mathcal F \times \mathcal X \to C^{2, \gamma}_{-\delta}(S^2M , g_0)  \quad\text{given by} \quad  \Phi (g, h ) = g + h,$$
    which is a smooth map of Banach manifolds.
    Clearly, $\Phi(g_0, 0) = g_0$ and 
        $$d \Phi_{(g_0, 0)} : T_{g_0} \mathcal F \oplus \mathcal X \to C^{2, \gamma}_{-\delta}(S^2 M, g_0) \text{ is } d \Phi_{(g_0, 0)} ( e, h ) = e + h.$$
    $T_{g_0} \mathcal F = \ker_{L^2} L$ is finite-dimensional (see e.g. \cite[Theorem 2.11]{DeruelleKroncke20}), and so there exists a finite basis, say $\{ e_j \}_{j =1}^k$, of $T_{g_0} \mathcal F = \ker_{L^2} L$.
    By assumption, $e_j \in C^{2, \gamma}_{-\mu}(S^2M, g_0)\cap C^{2, \gamma}_{-\delta}(S^2M, g_0)$ for all $1 \le j \le k$.
    The linearization of $\Phi$ has an explicit inverse
        $$( d \Phi_{(g_0,0)} )^{-1} : C^{2, \gamma}_{-\delta} (S^2 M, g_0) \to T_{g_0} \mathcal F \oplus \mathcal X \quad \text{ given by } $$
        $$( d \Phi_{(g_0, 0)} )^{-1}(g) = \left( \sum_{j=1}^k \langle g, e_j \rangle_{L^2 (g_0)} e_j, \, g - \sum_{j=1}^k \langle g, e_j \rangle_{L^2 (g_0)} e_j \right)$$
    which is well-defined by \eqref{orthogonal elliptic est proof, eqn 1}.
    The inverse function theorem for Banach manifolds now implies there are $C^{2, \gamma}_{-\delta}(S^2M, g_0)$-neighborhoods $U, U' \ni g_0$ and $U'' \ni 0$ such that 
        $$\Phi : ( \mathcal F \cap U') \times (\mathcal X \cap U'') \to U \quad \text{is a diffeomorphism}.$$
    Thus, any $g \in U$ can be written as $g = g_1 + h_1$ where $g_1 \in U' \cap \mathcal F$ and $h_1 \perp_{g_0} \ker_{L^2} L$, and additionally
    \begin{equation*}
        \| g_1 - g_0 \|_{C^{2, \gamma}_{-\delta}} + \| h_1 \|_{C^{2, \gamma}_{-\delta}} \to 0 \qquad \text{as } \| g - g_0 \|_{C^{2, \gamma}_{-\delta}} \to 0.
    \end{equation*}

    To obtain the estimate (\ref{orthogonalellipticest:est}), note that $h_1 \in \mathcal{X}$ and $L : \mathcal{X} \to C^{0, \gamma}_{-\delta - 2}$ is bounded and injective by definition. Moreover, it is Fredholm  \cite[Prop. 5.1]{DeruelleOzuch20} and thus has closed image. The desired estimate then follows from the Bounded Inverse Theorem.
\end{proof}

\begin{lemma}\label{lemma:orthogMest}
    For $g = g_1 + h_1$ as in Lemma \ref{lemma:orthogonalellipticest1}, after possibly shrinking $U,$ we have
    \begin{gather}
    \label{hsmallness}
        \|h_1\|_{C^{2,\gamma}_{-\delta}} \leq C \| \calM(g) \|_{C^{0, \gamma}_{-\delta - 2}},
    \end{gather}
    where $C = C(M^n, g_0, x_*, \gamma, \delta) > 0.$

\end{lemma}
\begin{proof} 
    Fix $\eps_0 > 0$ to be some small constant that is to be determined, and shrink $U$ so that Lemma \ref{lemma:orthogonalellipticest1} \eqref{lem orthogonal elliptic est, eqn comparable topologies} ensures 
    \begin{equation} \label{proof orthog M est, eqn 0}
        \| g_1 - g_0\|_{C^{2, \gamma}_{-\delta}} + \| h_1 \|_{C^{2, \gamma}_{-\delta} } < \eps_0.    
    \end{equation}
    Write 
    $g(t) = g_1 + t h_1$ for $0 \leq t \leq 1.$ Since $\calM(g_1) = 0,$ we have
    \begin{equation}
    \calM(g) = \calM(g) - \calM(g_1) = \int_0^1 d\calM(g(t))(h_1) \, dt. 
    \end{equation}
    Referring to the expression (\ref{dMcompactexpression}), we have
    \begin{equation}\label{Ricciexpansion}
    \begin{split}
    \calM(g) &= \int_0^1 \Delta_{g(t), g_0} h_1 - 2 \Rm(g_0) \# h_1 + F(g(t)) \# \nabla h_1 + G(g(t)) \# h_1 \, dt \\
    & = L h_1 + \int_0^1 \left( g(t)^{ij} - g_0^{ij} \right) \nabla_i \nabla_j h_1 + F(g(t)) \# \nabla h_1 + G(g(t)) \# h_1 \, dt.
    \end{split}
    \end{equation}
    We now estimate the three terms on the RHS of (\ref{Ricciexpansion}). We have
    $$\| \left( g(t)^{ij} - g_0^{ij} \right) \nabla_i \nabla_j h_1 \|_{C^{0, \gamma}_{-\delta - 2}} \leq \| g(t)^{-1} - g_0^{-1} \|_{C^{0, \gamma}_0} \| \nabla^{(2)} h_1 \|_{C^{0, \gamma}_{-\delta - 2}} \leq C \eps_0 \| h_1 \|_{C^{2, \gamma}_{-\delta}},$$
    where we have used (\ref{proof orthog M est, eqn 0}).
    Observe from (\ref{Fexpression}) and (\ref{proof orthog M est, eqn 0}) that
    $$\|F(g(t)) \|_{C^{0, \gamma}_{-1}} \leq C \|g(t) - g_0 \|_{C^{1, \gamma}_0} \leq C \eps_0$$
    for each $0 \leq t \leq 1.$ By Lemma \ref{lemma:interpolation}, this gives
    $$\| F(g(t)) \# \nabla h_1 \|_{C^{0, \gamma}_{- \delta - 2}} \leq \|F(g(t)) \|_{C^{0, \gamma}_{-1}} \| h_1 \|_{C^{1, \gamma}_{-\delta}} \leq C \eps_0 \| h_1 \|_{C^{2, \gamma}_{-\delta}}.
    $$
    Last, from (\ref{Gexpression}), we have
    $$\|G(g(t)) \|_{C^{0,\gamma}_{-2}} \leq C \left( \|g(t) - g_0 \|_{C^{2, \gamma}_0} + \|\Rm(g_0) \# (g(t) - g_0) \|_{C^{0,\gamma}_{-2}} \right) + \|g(t) - g_0 \|^2_{C^{1, \gamma}_0} \leq C \eps_0,$$
    since $| \Rm(g_0) | = O(\rho^{-2}).$
    This gives
    $$\| G(g(t)) \# h_1 \|_{C^{0, \gamma}_{- \delta - 2}} \leq \|G(g(t)) \|_{C^{0, \gamma}_{-2}} \| h_1 \|_{C^{0, \gamma}_{-\delta}} \leq C \eps_0 \| h_1 \|_{C^{2, \gamma}_{-\delta}}.
    $$
    Overall, we may now conclude from (\ref{Ricciexpansion}) that
    $$\| \calM(g) \|_{C^{0, \gamma}_{-\delta - 2}} \geq \|L h_1 \|_{C^{0, \gamma}_{-\delta - 2}} - C \eps_0 \| h_1 \|_{C^{2, \gamma}_{-\delta}}.$$
    Applying Lemma \ref{lemma:orthogonalellipticest1}, we obtain
        $$\| \calM(g) \|_{C^{0, \gamma}_{-\delta - 2}} \geq \left( \frac{1}{C} - C \eps_0 \right) \| h_1 \|_{C^{2, \gamma}_{-\delta}}.$$
        For $\eps_0$ sufficiently small, this implies the desired estimate.
\end{proof}

\begin{lemma}\label{lemma:RicandVbounds}
    For $\|g - g_0\|_{C^{2, \gamma}_{-\delta}(M, g_0)} < \eps_0 = \eps_0 (M^n, g_0, x_*, \gamma, \delta),$ we have
    \begin{equation}
        \|\Ric_g \|_{C^{0,\gamma}_{-\delta-2}} + \|V(g,g_0) \|_{C^{1,\gamma}_{-\delta-1}} \leq C \| \calM(g) \|_{C^{0, \gamma}_{-\delta - 2}},
    \end{equation}
    where $C = C(M^n, g_0, x_*, \gamma, \delta) > 0.$

\end{lemma}
\begin{proof}
    Lemma \ref{lemma:orthogonalellipticest1} implies that, when $\eps_0$ is sufficiently small and $\| g - g_0 \|_{C^{2, \gamma}_{-\delta}} < \eps_0$, we can write $g$ as $g = g_1 + h_1$ for some $g_1, h_1 \in  C^{2, \gamma}_{-\delta}$ with $\mathcal M (g_1) = 0$ and $h_1 \perp_{g_0} \ker_{L^2} L $.
    Moreover, since $\delta > 0$,
    \begin{equation} \label{proof Ric and V bounds, eqn small g_1-g_0 and h_1}
        \| g_1 - g_0 \|_{C^{2, \gamma}_{0}} + \| h_1 \|_{C^{2, \gamma}_0} 
        \le \| g_1 - g_0 \|_{C^{2, \gamma}_{-\delta}} + \| h_1 \|_{C^{2, \gamma}_{-\delta}} \to 0 \qquad \text{as } \eps_0 \to 0. 
    \end{equation}

    By Deruelle-Kr{\"o}ncke \cite[Theorem 2.7]{DeruelleKroncke20}, 
    we know that $V(g_1, g_0) = 0.$ We can expand $V$ as follows:
    \begin{equation}\label{Vexpansion}
    \begin{split}
        V(g, g_0) = V(g, g_0) - V(g_1, g_0) & = g^{ij} \left( \Gamma(g)^k_{ij} - \Gamma(g_0)^k_{ij}\right) - g_1^{ij} \left( \Gamma(g_1)^k_{ij} - \Gamma(g_0)^k_{ij}\right) \\
        & = \left( g^{ij} - g_1^{ij} \right) \left( \Gamma(g)^k_{ij} - \Gamma(g_0)^k_{ij} \right) + g_1^{ij} \left( \Gamma(g)^k_{ij} - \Gamma(g_1)^k_{ij} \right).
        \end{split}
    \end{equation}
    We further have
    \begin{align*}
        \Gamma(g)^k_{ij} - \Gamma(g_0)^k_{ij} 
        ={}& \frac12 \left( g^{k\ell} - g_0^{k\ell} \right) \left( \p_i g_{j \ell} + \p_j g_{i \ell} - \p_\ell g_{ij}  \right) \\
        &+ \frac12 g_0^{k \ell} \left( \p_i (g - g_0)_{j\ell} + \p_j (g - g_0 )_{i \ell} - \p_\ell ( g-g_0)_{ij}  \right)
    \end{align*}
    and similarly for $\Gamma(g) - \Gamma(g_1)$,
    so that
    $$ \| \Gamma(g) - \Gamma(g_0) \|_{C^{1,\gamma}_{- 1}} \leq C \|g - g_0\|_{C^{2,\gamma}_{0}} \leq C\eps_0$$
    and
      $$ \| \Gamma(g) - \Gamma(g_1) \|_{C^{1,\gamma}_{-\delta - 1}} \leq C \|h_1\|_{C^{2,\gamma}_{-\delta}}.$$
    Combining these estimates in \eqref{Vexpansion} and using Lemmas \ref{lemma:interpolation} and \ref{lemma:orthogMest}, we obtain
      $$\|V(g,g_0) \|_{C^{1,\gamma}_{-\delta-1}} \leq C\| h_1 \|_{C^{2,\gamma}_{-\delta}} \leq C \| \calM(g) \|_{C^{0, \gamma}_{-\delta - 2}}.$$
    The bound on Ricci then follows from the definition \eqref{Mdefn}.
\end{proof}

The following theorem is morally similar to \cite[Lemma 3.8]{DeruelleKroncke20}.

\begin{thm}[Quantitative almost-orthogonality]\label{thm:ALEalmostorthog} 
    Let $(M^n, g_0, x_*)$ be a pointed Ricci-flat ALE manifold as in Lemmas \ref{lemma:orthogonalellipticest1}-\ref{lemma:RicandVbounds}.
    There exists $\eps_0 = \eps_0 (M^n, g_0, x_*, \gamma, \delta) > 0$ and $C = C(M^n, g_0, x_*, \gamma, \delta) > 0$ such that the following holds:

    For all $g \in C^{2, \gamma}_{-\delta}$ with
        $$\| g - g_0 \|_{C^{2, \gamma}_{-\delta}} \le \eps_0,$$
    there exists $g_1, h_1 \in C^{2, \gamma}_{-\delta}$ with $\mathcal M (g_1) = 0$ and $h_1 \perp_{L^2(g_0)} \ker_{L^2} L $ such that $g= g_1 + h_1$ and
    \begin{equation}
        ( \calM(g), e_1 )_{g_1} \leq C \|\calM(g) \|_{C^{0, \gamma}_{-\delta-2} }^2
        \qquad \forall e_1 \in \ker L_{g_1} \text{ with } \| e_1 \|_{L^2(g_1)} = 1.
    \end{equation}
\end{thm}
\begin{proof}
We have
$$\calM(g) = -2 \Ric_g + \div_g^* V(g, g_0).$$
Also write
$$\calM_{g_1}(g) = -2 \Ric_g + \div_g^* V(g, g_1).$$
We have
\begin{equation}\label{M(g)M(g1)}
\calM(g) = \calM_{g_1}(g) + \div_g^* \left( V(g,g_0) - V(g,g_1) \right).
\end{equation}
We calculate
$$V(g,g_0) - V(g,g_1) = \cancelto{0}{V(g_1,g_0)} + \left( g^{ij} - g_1^{ij} \right) \left( \Gamma(g_1)^k_{ij} - \Gamma(g_0)^k_{ij} \right),$$
so that by Lemma \ref{lemma:orthogMest}
\begin{equation*}
\| V(g,g_0) - V(g,g_1) \|_{C^{1,\gamma}_{-\delta - 1}} \leq C \| h_1 \|_{C^{2,\gamma}_{-\delta}} \| g_1 - g_0 \|_{C^{2, \gamma}_{0}} \leq C \eps_0 \| \calM(g) \|_{C^{0,\gamma}_{-\delta - 2}}.
\end{equation*}

Next, by \cite[Lemma 2.5]{DeruelleKroncke20}, we have $\div_{g_1} e_1 = 0.$ So for any decaying vector field $X,$ we have
\begin{equation*}
\begin{split}
 \left( \div_g^* X, e_1 \right)_{g_1} & = \left( \div_{g_1}^* X + \left( \div^*_g - \div_{g_1}^* \right) X , e_1 \right)_{g_1} \\
 & = 0 + \left( \left( \Gamma(g) - \Gamma(g_1) \right) \# X , e_1 \right)_{g_1}.
    \end{split}
\end{equation*}
Using $\delta > 0$, this gives the estimate
\begin{gather} \label{divg*v1est} \begin{aligned}
    &\left| \left( \div_g^* \left( V(g,g_0) - V(g,g_1) \right), e_1 \right)_{g_1} \right| \\
    &\leq C \| \Gamma(g) - \Gamma(g_1) \|_{C^{0, \gamma}_{-\delta-1}} \| V(g, g_1) - V(g, g_0) \|_{C^{0, \gamma}_{-\delta-1} } \| e_1 \|_{C^{0, \gamma}_{-\mu}} 
    \int \rho^{-2(\delta+1) - \mu} \rho^{n-1}  dr \\
    & \leq C \|h_1 \|_{C^{2, \gamma}_{- \delta }}^2 \| g_1 - g_0 \|_{C^{2, \gamma}_{0}} \| e_1 \|_{C^{0, \gamma}_{-\mu}}  \int \rho^{-2\delta - 2 - \mu + n-1}  \, dr \\
    & \leq C \eps_0 \|h_1 \|_{C^{2, \gamma}_{- \delta }}^2
    \int \rho^{-\delta - \mu + n-1}  \, dr \\ 
    & \leq C \eps_0\|\calM(g)\|^2_{C^{0, \gamma}_{-\delta - 2}}
\end{aligned} \end{gather}
where the last line follows from the fact that $\mu > n - \delta$.
Now, by the expansion (\ref{Ricciexpansion}) with $g_1$ in place of $g_0,$ we have
\begin{equation*}
    \calM_{g_1}(g) = L_{g_1} h_1 + \int_{0}^1 \left( g(t)^{ij} - g_1^{ij} \right)\nabla^{g_1}_i \nabla^{g_1}_j h_1 + F_{g_1}(g(t)) \# \nabla^{g_1} h_1 + G_{g_1}(g(t)) \# h_1 \, dt.
\end{equation*}
Estimating as in the proof of Lemma \ref{lemma:orthogMest}, we get
\begin{equation*}
    \| \calM_{g_1}(g) - L_{g_1} h_1 \|_{C^{0, \gamma}_{-2\delta - 2}} \leq C \|h_1\|^2_{C^{2, \gamma}_{-\delta}} \leq C \|\calM(g)\|^2_{C^{0, \gamma}_{-\delta - 2}}.
\end{equation*}
This gives the estimate
\begin{equation}\label{M-Lg1est}
\begin{split}
    \left| \left( \calM_{g_1}(g), e_1 \right)_{g_1} \right| & = \left| \left( \calM_{g_1}(g) - L_{g_1} h_1, e_1 \right)_{g_1} \right| \\
    & \leq C \|\calM_{g_1}(g) - L_{g_1} h_1 \|_{C^{0, \gamma}_{-2 \delta - 2}} \| e_1 \|_{C^{0, \gamma}_{-\mu}}\int_1^\infty \rho^{-2(\delta + 1) - \mu} \rho^{n-1} \, dr \\
    & \leq C \|\calM(g)\|^2_{C^{0, \gamma}_{-\delta - 2}}.
    \end{split}
\end{equation}
Finally, we take the inner product with $e_1$ in (\ref{M(g)M(g1)}) and insert (\ref{divg*v1est}) and (\ref{M-Lg1est}), to obtain
\begin{equation*}
    \left| \left( \calM(g), e_1 \right)_{g_1} \right| = \left| \left( \calM_{g_1}(g), e_1 \right)_{g_1} + \left( \div_g^*\left( V(g,g_0) - V(g,g_1) \right) , e_1 \right)_{g_1} \right| \leq C \|\calM(g)\|^2_{C^{0, \gamma}_{-\delta - 2}},
\end{equation*}
as desired.
\end{proof}

\vspace{10mm}

\section{Parabolic estimates} \label{sec: parabolic estimates}

\subsection{Norms of evolving tensors}

Define
$$| h |^2 = |h |^2_{g_0}, \qquad | \nabla^{g_0} h |_{g,g_0}^2 = g^{ab} g_0^{ik}g_0^{j\ell} \nabla^{g_0}_a h_{ij} \nabla^{g_0}_b h_{k \ell}.$$
The proof of the standard Kato inequality goes through, giving 
\begin{lemma}\label{lemma:kato} For any symmetric 2-tensor $h,$ we have
$| \nabla^{g_0} |h| |_g \leq | \nabla^{g_0} h |_{g,g_0}.$
\end{lemma}

\begin{lemma}\label{lemma:hpevolution} 
Fix $q > 1.$ Let $g(t) = g_0 + h(t)$ solve Ricci-DeTurck flow with respect to $g_0,$ and write $|h|_{g_0} = |h|$ and $\nabla = \nabla^{g_0}.$ Suppose that
\begin{equation}\label{gg0lambdaassumption}
\lambda^{-2} g_0 \leq g \leq \lambda^2 g_0,
\end{equation}
where $\lambda$ is sufficiently close to $1$ (depending on $n$ and $q$).
Then
    \begin{equation}\label{evolutionof|h|^q}
        \left( \partial_t - \Delta_{g(t), g_0} \right) |h|^{q} + \frac{4(q-1)}{q} \left| \nabla |h|^{\frac{q}{2}} \right|_{g(t)}^2 \leq 2q \lambda^4 |\Rm(g_0)| |h|^{q}
    \end{equation}
    and
        \begin{equation}\label{g0evolutionof|h|^q}
        \left( \partial_t - \Delta_{ g_0} \right) |h|^{q} + \frac{3(q-1)}{q} \left| \nabla |h|^{\frac{q}{2}} \right|_{g_0}^2 \leq 2q \lambda^4 |\Rm(g_0)| |h|^{q} + \nabla_i \left( \left( g^{ij} - g_0^{ij} \right) \nabla_j |h|^q \right).
    \end{equation}
\end{lemma}
\begin{proof}
We have
   \begin{equation}
       \begin{split}
           \left( \frac{\p}{\p t} - \Delta_{g, g_0} \right) |h|^2 & = 2 \LA \left( \frac{\p}{\p t} - \Delta_{g, g_0} \right) h, h \RA - 2 |\nabla h|_{g,g_0}^2.
       \end{split}
   \end{equation}
   For the remaining terms in the first line of \eqref{calMfullexpression}, we can estimate
   \begin{equation*}
     \begin{split}
     & \left| h_{ij} g_0^{im} g_0^{in} h_{ab} g^{ka} g_0^{\ell b} g_{ip} g_0^{pq} \Rm(g_0)_{jk\ell q} \right| \\
     & \quad \leq \lambda^4 \left| h_{ij} g_0^{im} g_0^{in} h_{ab} g_0^{ka} g_0^{\ell b} (g_0)_{ip} g_0^{pq} \Rm(g_0)_{jk\ell q} \right| \\
     & \quad \leq \lambda^4 |\Rm(g_0)| |h|^2.
     \end{split}
     \end{equation*}
   Similar calculations with the other terms in \eqref{calMfullexpression} give
      \begin{equation}\label{evolutionof|h|^2}
       \begin{split}
           \left( \frac{\p}{\p t} - \Delta_{g, g_0} \right) |h|^2 & \leq - 2 |\nabla h|_{g,g_0}^2 + 4 \lambda^4 |\Rm(g_0)| |h|^2 + \frac{9}{2} \lambda^8 |h| |\nabla h|_{g,g_0}^2 \\
           & \leq \left( - 2 +  \frac{9}{2} \lambda^8 |h| \right) |\nabla h|_{g,g_0}^2 + 4 \lambda^4 |\Rm(g_0)| |h|^2,
       \end{split}
   \end{equation}
   which is (\ref{evolutionof|h|^q}) with $q = 2.$
For general $q,$ we have
      \begin{equation*}
       \begin{split}
           \left( \frac{\p}{\p t} - \Delta_{g, g_0} \right) |h|^q & = -q(q-2) |h|^{q-2}|\nabla |h||_{g}^2 + \frac{q}{2} |h|^{q-2} \left( \frac{\p}{\p t} - \Delta_{g,g_0} \right) |h|^2 \\
           & \leq q |h|^{q-2} \left( (2-q) | \nabla |h| |_g^2 - \left( 1 - \frac{9}{4} \lambda^8 |h| \right) |\nabla h|_{g,g_0}^2 + 2\lambda^4 |\Rm(g_0)| |h|^2 \right). 
       \end{split}
   \end{equation*}
    We now assume that $\sup |h| \leq \sqrt{n} \left(\lambda^2 - 1 \right) \leq \frac{2 (q-1)}{9 \lambda^8},$ so that 
   $$\frac{9}{4} \lambda^8 |h| - 1 \leq \frac12 (q - 1) - 1.$$
 Applying the Kato inequality of Lemma \ref{lemma:kato}, we obtain
         \begin{equation}\label{hpevolution:almostdone}
       \begin{split}
           \left( \frac{\p}{\p t} - \Delta_{g, g_0} \right) |h|^q 
           & \leq q |h|^{q-2} \left( (1-q) | \nabla |h| |_g^2 + \frac{1 - q}{2} |\nabla h|_{g,g_0}^2 + 2\lambda^4 |\Rm(g_0)| |h|^2 \right) \\
           & \leq \frac{4(1 - q)}{q} \left| \nabla |h|^{\frac{q}{2}} \right|_g^2  + \frac{q(1 - q)}{2} |h|^{q-2} |\nabla h |^2_{g,g_0} + 2 q \lambda^4 |\Rm(g_0) ||h|^q.
       \end{split}
   \end{equation}
   This implies (\ref{evolutionof|h|^q}).

   To prove (\ref{g0evolutionof|h|^q}), we note that
   \begin{equation}
   \begin{split}
   \Delta_{g,g_0} |h|^q & = \Delta_{g_0} |h|^q + (g^{ij} - g_0^{ij}) \nabla_i \nabla_j |h|^q \\
   & = \Delta_{g_0}|h|^q + \nabla_i \left( (g^{ij} - g_0^{ij})  \nabla_j|h|^q \right) + g^{ik} \nabla_i h_{k\ell} g^{\ell j} q |h|^{q-1}\nabla_j |h|.
   \end{split}
   \end{equation}
   By the Kato inequality, we can estimate the last term as
   $$ | g^{ik} \nabla_i h_{k\ell} g^{\ell j} q |h|^{q-1}\nabla_j |h| | \leq q \lambda^2 \sqrt{n} |h| |h|^{q-2} |\nabla h |^2,$$
   which can be absorbed into the second term on the RHS of (\ref{hpevolution:almostdone}) if $|h| \leq \frac{q-1}{2 \lambda^2 \sqrt{n}}.$
\end{proof}

\begin{lemma}\label{Mpevolution} Let $q, \lambda > 1$ and $g(t) = g_0 + h(t)$ be as above, satisfying (\ref{gg0lambdaassumption}). We have
        \begin{equation}\label{evolutionof|M|^q}
        \left( \partial_t - \Delta_{g(t),g_0} \right) |\calM|^{q} \leq  C_{q, \lambda} \left( \left( 1 + |h| \right) |\Rm(g_0)| + |\nabla h|^2 + |\nabla^2 h| \right) |\calM|^{q} .
    \end{equation}
\end{lemma}
\begin{proof}
From (\ref{Mevolution}), we have
   \begin{equation}
       \begin{split}
           \left( \frac{\p}{\p t} - \Delta_{g, g_0} \right) |\calM|^2 & = 2 \LA \left( \frac{\p}{\p t} - \Delta_{g, g_0} \right) \calM , \calM \RA - 2 |\nabla \calM |_{g,g_0}^2 \\
           & = 2 g_0^{im} g_0^{jn}\left( -2  g_0^{ka}g_0^{\ell b} \Rm(g_0)_{ik\ell j} \mathcal{M}_{ab} + F_{ij}{}^{cab} \nabla^{g_0}_c \mathcal{M}_{ab} + G_{ij}{}^{ab} \mathcal{M}_{ab} \right) \calM_{mn} \\
           & \quad - 2 |\nabla \calM|_{g,g_0}^2.
       \end{split}
   \end{equation}
   Applying Young's inequality with $\eps > 0$ to the second term, we get
      \begin{equation}
       \begin{split}
           \left( \frac{\p}{\p t} - \Delta_{g, g_0} \right) |\calM|^2 
           & \leq \left( \frac{\lambda^2 \eps}{2} - 2 \right) |\nabla \calM|_{g,g_0}^2 + \left( 4 |\Rm (g_0)| + \frac{|F|^2}{2 \eps} + |G| \right) |\calM|^2.
       \end{split}
   \end{equation}
   Calculating as in the proof of Lemma \ref{lemma:hpevolution}, we get
         \begin{equation}
       \begin{split}
           \left( \frac{\p}{\p t} - \Delta_{g, g_0} \right) |\calM|^q 
           & \leq q |\calM|^{q-2} \left( \left( 1 + \frac{\lambda^2 \eps}{4} - q \right) |\nabla \calM|_{g,g_0}^2 + \left( 4 |\Rm (g_0)| + \frac{|F|^2}{2 \eps} + |G| \right) |\calM|^2 \right).
       \end{split}
   \end{equation}
   Taking $\eps = \frac{\lambda^2}{4(q-1)}$ gives
            \begin{equation}
       \begin{split}
           \left( \frac{\p}{\p t} - \Delta_{g, g_0} \right) |\calM|^q 
           & \leq q \left( 4 |\Rm (g_0)| + \frac{\lambda^2 |F|^2}{8 (q-1)} + |G| \right) |\calM|^q.
       \end{split}
   \end{equation}
   The identity (\ref{evolutionof|M|^q}) follows by inspection from (\ref{Fexpression}-\ref{Gexpression}).
\end{proof}

\vspace{2mm}

\vspace{2mm}

\subsection{Well-posedness}

In this subsection, we prove short-time well-posedness of Ricci-DeTurck flow with respect to a Ricci-flat ALE background metric in the weighted H{\"o}lder spaces $C^{2, \gamma}_{-\delta}$ (Theorem \ref{thm:wellposedness}).
The proof essentially follows from the well-posedness of complete, bounded-curvature Ricci-DeTurck flows \cite{Shi89, ChenZhu06} and the Schauder estimates for weighted H{\"o}lder norms developed in Appendix \ref{section Schauder ests}.

\begin{thm}\label{thm:wellposedness}
    Let $(M^n, g_0, x_*)$ be a pointed Ricci-flat ALE manifold.
    For any $\delta \geq 0$ and $\gamma \in (0,1)$, we have short-time well-posedness of $C^{2, \gamma}_{-\delta}(M, g_0)$ Ricci-DeTurck flows starting in small $C^{2, \gamma}_{-\delta}(M, g_0, x_*)$-neighborhoods of $g_0$.

    In other words, there exists $\eps_0 \in (0,1)$ and $K > 0$ (depending only on $M^n ,g_0, \gamma, \delta, x_*$) such that if $g(0) \in C^{2, \gamma}_{-\delta}(M, g_0, x_*)$ with 
    \begin{equation} \label{wellposedness:C2gammadelta_initial_est}
        \| g(0) - g_0\|_{C^{2, \gamma}_{-\delta}(M, g_0, x_*)} \leq \eps \leq \eps_0,
    \end{equation}
    then there is a $T > -K^{-1} \log \eps$ and a solution $(g(t))_{t \in [0, T]}$ to
    Ricci-DeTurck flow \eqref{eqn Ricci-DeTurck flow} with reference metric $g_0$ and initial data $g(0)$ such that 
    \begin{equation}\label{wellposedness:C2gammadeltaest}
        \| g(t ) - g_0 \|_{C^{2, \gamma}_{-\delta} } \leq C \eps e^{Kt} \quad  \forall t \in [0,T].
    \end{equation}
    Additionally, (up to restricting time intervals) this solution is unique among all complete, bounded-curvature Ricci-DeTurck flows.
\end{thm}

\begin{proof}
(Uniqueness)
    Let $(g(t))_{t \in [0, T]}$ be a Ricci-DeTurck flow as in the statement of the theorem and let $(g'(t))_{t \in [0, T']}$ be a complete, bounded curvature Ricci-DeTurck flow with the same initial data $g'(0 ) = g(0)$.
    The $C^{2, \gamma}_{-\delta}$-bounds \eqref{wellposedness:C2gammadeltaest} ensure $g(t)$ is also complete with bounded curvature, and
    the uniqueness of complete, bounded curvature Ricci-DeTurck flows then implies $g'(t) = g(t)$ for all $t \in [0, \min (T, T') ]$ \cite{ChenZhu06}.

(Existence)
    The assumed $C^{2, \gamma}_{-\delta}$-estimates \eqref{wellposedness:C2gammadelta_initial_est} on $g(0) - g_0$  imply $g(0)$ is complete with bounded curvature, say $|Rm_{g(0)}|_{g(0)} \le \mathcal K < \infty$ where $\mathcal K = \mathcal K(M^n, g_0, \eps_0)$.
    By \cite[Theorems 4.3 and 6.6]{Shi89},
    there then exists a smooth, complete, bounded curvature Ricci-DeTurck flow $(g(t))_{t \in [0, T)}$ defined up to some maximal time $T > 0$.
    It must be the case that 
        $$\sup_{t \in [0, T)} \| g(t) - g_0 \|_{C^{2, \gamma}_0(M, g_0)} > \sqrt{\eps_0},$$
    else $g(t)$ would be complete with uniformly bounded curvature $|Rm_{g(t)}|_{g(t)} \le \mathcal K'(M^n, g_0,\eps_0) < \infty$ for all $t \in [0, T)$ and one could then extend the flow beyond time $T$ \cite{Shi89}.
    We can therefore define $T_0$ to be
        $$T_0 := \inf \left\{ t \in [0, T) \, \colon \, \| g(t) - g_0 \|_{C^{2, \gamma}_0} \ge \sqrt{\eps_0} \right\},$$
    and observe, using also the initial data bounds \eqref{wellposedness:C2gammadelta_initial_est}, that
    \begin{equation} \label{well-posedness proof, eqn properties of T_0}
        T_0 \in (0, T) \qquad \text{and} \qquad \sup_{t \in [0, T_0] } \| g(t) - g_0 \|_{C^{2, \gamma}_0} = \sqrt{\eps_0}.
    \end{equation}
        

    Fix $q > 1.$ It is easy to see (as in the proof of Lemma \ref{lemma:comparisonprinciple}) that with the $C^{2, \gamma}_0$-bound in (\ref{well-posedness proof, eqn properties of T_0}), the function
        $$\frac{e^{q Kt} (C_n\eps)^q}{\rho(x)^{\delta q }}$$
    is a supersolution for the evolution (\ref{evolutionof|M|^q}) of $|\mathcal{M}|^q$ (where $K = K( M^n, g_0, \gamma, \delta, x_*) > 0$ is independent of $\eps$).
    From the comparison principle and the initial-data bounds \eqref{wellposedness:C2gammadelta_initial_est}, we learn that
        $$| \calM (x,t)| \leq \frac{C_n e^{Kt} \eps}{\rho(x)^{\delta}} \qquad \forall t \in [0, T_0].$$
    Integrating this estimate in time, we obtain the $C^0$ bound
    \begin{equation}\label{wellposedness:C0deltabound}
        \|g(t) - g(0) \|_{C^0_{-\delta}} \leq \frac{\left( e^{Kt} - 1 \right) C_n \eps}{K}
        \qquad \forall t \in [0, T_0].
    \end{equation}
    Applying Lemma \ref{lemma:globalC2gamma-deltaestimate} with the given $\delta,$ it follows that
        $$\| g(t) - g_0 \|_{C^{2, \gamma}_{-\delta} } \le \frac{C \eps ( e^{Kt} - 1) }{K} + \eps \le C' \eps e^{Kt} \qquad \forall t \in [0, T_0]$$
    where $C, C' > 0$ depend only on $M^n, g_0, \gamma, \delta, x_*$.
    Combining this estimate with the fact that $\sup_{t \in [0, T_0]} \| g(t) - g_0 \|_{C^{2, \gamma}_0} = \sqrt{\eps_0} \ge \sqrt{\eps}$ \eqref{well-posedness proof, eqn properties of T_0}, it follows that $ T_0 > - K' \log \eps$ for some $K' = K'(M^n, g_0, \gamma, \delta , x_*) > 0$.
    This completes the proof.
\end{proof}

\subsection{Liouville Theorem} We need the following generalization of a key theorem in \cite{BrendleKapouleas17}.

\begin{thm}[{\cite[Proposition 5.2]{BrendleKapouleas17}}] \label{thm:liouville}
    Fix a linearly stable, pointed, Ricci-flat, ALE manifold $(M^n,g, x_*).$ 
    If $h(x,t)$ is an ancient solution of the Lichnerowicz heat equation
        $$\left( \frac{\p}{\p t} - L_g \right) h = 0 \qquad \text{on } M \times(-\infty, 0),$$
    which, for some $\ell \in (0, n-2)$, satisfies
        $$|h(x,t)|_g \leq \rho(x)^{-\ell} \qquad \forall t \in (-\infty, 0)$$
    and $h(t_0) \perp_{L^2(g)} \ker_{L^2} L_g$ for some $t_0 < 0$, 
    then $h \equiv 0$ on $M \times (-\infty, 0)$.
\end{thm}
\begin{proof}
    As observed by Deruelle and Ozuch \cite[Prop. 5.3]{DeruelleOzuch25}, the proof in \cite[Prop. 5.2]{BrendleKapouleas17} goes through for dimension $n=4$. 
    The same holds for $n \ge 5$ using also Proposition \ref{prop L2 kernel equals Holder kernel}.
    One can further observe that the orthogonality condition is only required for a fixed time $t_0 < 0.$
\end{proof}

\subsection{Comparison principle}


    In this subsection, we begin by focusing on the subset of Euclidean space $(0, \infty) \subseteq \R$ with the standard coordinate $r$.
For a real $\mu \geq 1,$ we write
$$\Delta_\mu :=\p_r^2 + \frac{(\mu - 1)}{r} \p_r.$$
We will construct radial barriers for small perturbations of the equation $\left( \partial_t - \Delta_\mu \right) f = 0.$ 
Together with the above Liouville theorem, this will be the basis for our proof of stability in weighted H\"older spaces. For $L^p$-spaces, we use an alternative heat-kernel approach in \S \ref{sec:Lpstability}.

\begin{lemma}\label{lemma:chilemma} Given $a, b > 0,$ define ``Kummer's confluent hypergeometric function''
$$\chi(u) = \chi_{a,b}(u) = \int_0^1 x^{a-1} (1 - x)^{b-1} e^{-ux} \, dx.$$
The function $\chi(u)$ is smooth, positive, strictly decreasing, and concave-up for $u \geq 0,$ and satisfies the ODE
\begin{equation}\label{chilemma:chiuode}
u \chi''(u) + \left(a + b + u\right) \chi'(u) + a \chi(u) = 0.
\end{equation}
We have
\begin{equation}\label{chilemma:chigammaasymp}
\chi(u) = 
\Gamma(a) u^{-a} + O(u^{-a-1}) \quad (u \to \infty),
\end{equation}
where $\Gamma$ is the Euler Gamma function, as well as
\begin{equation}\label{chilemma:chiupperlower}
\frac{\chi(0)}{1 + C_{a,b} u^{a}} \leq \chi(u) \leq \frac{\chi(0)}{1 + c_{a,b} u^{a}},
\end{equation}
for approriate constants $C_{a,b}, c_{a,b} > 0,$ and
\begin{equation}\label{chilemma:chi'est}
c_{a,b} \chi(u) \leq -(1 + u) \chi'(u) \leq C_{a,b} \chi(u),
\end{equation}
\begin{equation}\label{chilemma:chi''est}
c_{a,b} \chi(u) \leq (1 + u^2) \chi''(u) \leq C_{a,b} \chi(u).
\end{equation}
\end{lemma}
\begin{proof}
We have
\begin{equation}\label{chilemma:chi'chi''}
\chi'(u) = -\int_0^1 x^{a} (1 - x)^{b-1} e^{-ux} \, dx, \quad \chi''(u) = \int_0^1 x^{a + 1} (1 - x)^{b-1} e^{-ux} \, dx.
\end{equation}
Using integration by parts, we have
\begin{equation*}
\begin{split}
-\chi'(u) & = \cancelto{0}{\left. x^{a} e^{-ux} \left( \frac{-(1-x)^{b}}{b} \right) \right|_0^1} + \frac{1}{b}\int_0^1 \left( a x^{a-1} e^{-ux} - ux^{a} e^{-ux} \right) (1 - x) (1 - x)^{b-1} \, dx \\
& =\frac{1}{b} \left( a \chi(u) + (a + u) \chi'(u) + u \chi''(u) \right).
\end{split}
\end{equation*}
Rearranging gives the ODE (\ref{chilemma:chiuode}). The formula (\ref{chilemma:chigammaasymp}) follows by changing variables and, in view of (\ref{chilemma:chi'chi''}), clearly implies (\ref{chilemma:chiupperlower}-\ref{chilemma:chi''est}).
\end{proof}

\begin{prop}\label{prop:firstsupersol}
Fix $\nu \geq 1$ and $0 < k < \nu.$ 
The function 
$$F(r,t) = F_{k, \nu}(r,t) = \frac{\chi_{\frac{k}{2}, \frac{\nu-k}{2}} \left( \dfrac{r^2}{4t} \right) }{2^k \Gamma\left( \frac{k}{2} \right)  t^{\frac{k}{2} }}$$
satisfies
\begin{equation}\label{firstsupersol:pde}
\left( \frac{\p}{\p t} - \Delta_{\nu} \right) F(r,t) = 0
\end{equation}
with
\begin{equation}
F(r,t) \to r^{-k} \quad (t \searrow 0).
\end{equation}
We have
\begin{equation}\label{firstsupersol:Fbounds}
\frac{c_{k,\nu}}{(r^2 + t)^{k/2}} \leq F(r,t) \leq \frac{C_{k, \nu} }{(r^2 + t)^{k/2}},
\end{equation}
\begin{equation}\label{firstsupersol:dFdrbound}
\frac{c_{k, \nu} r F(r,t)}{r^2 + t }  \leq - \frac{\p F}{\p r} \leq \frac{C_{k, \nu} r F(r,t)}{r^2 + t} ,
\end{equation}
and
\begin{equation}\label{firstsupersol:d2Fdr2bound}
\left| \frac{\p^2 F}{\p^2 r} \right| \leq \frac{C_{k, \nu} F(r,t)}{r^2 + t} .
\end{equation}
Further, for $\mu \geq \nu,$ 
we have
\begin{equation}\label{firstsupersol:evolutionlowerbound}
\left( \frac{\p}{\p t} - \Delta_{\mu} - \frac{ c_{k, \nu} (\mu - \nu) } { r^2 + t } \right) F(r,t) \geq 0  
\end{equation}
\end{prop}
\begin{proof}
For a general function $g(x),$ we have
\begin{equation}
\Delta_{\nu} t^{-\frac{k}{2} } g\left( \frac{r}{\sqrt{t}} \right)  = t^{- \frac{k}{2} - 1} \left( g''\left( \frac{r}{\sqrt{t}} \right) + \frac{\nu -1}{r/\sqrt{t}} g'\left( \frac{r}{\sqrt{t}} \right) \right)
\end{equation}
and
\begin{equation}
\left( \frac{\partial}{\partial t} - \Delta_{\nu} \right) t^{-\frac{k}{2} } g\left( \frac{r}{\sqrt{t}} \right) = -t^{- \frac{k}{2} - 1} \left( g''\left( \frac{r}{\sqrt{t}} \right) + \left( \frac{ \nu -1}{r/\sqrt{t}} + \frac{r}{2 \sqrt{t}} \right) g'\left( \frac{r}{\sqrt{t}} \right) + \frac{k}{2} g\left( \frac{r}{\sqrt{t}} \right) \right).
\end{equation}
Let $u = \frac{v^2}{4}$ and take $g(v) = \chi_{a,b}(u).$
Observe that
$$g'(v) = \frac{v}{2} \chi_{a,b}'(u) ,$$
$$g''(v) = \frac{v^2}{4} \chi_{a,b} ''(u) + \frac12 \chi_{a,b}' (u) = u \chi_{a,b}''(u) + \frac12 \chi_{a,b}'(u) .$$
We obtain
$$g''\left(v \right) + \left( \frac{\nu-1}{v} + \frac{v}{2} \right) g'\left( v\right) + \frac{k}{2} g\left( v \right) = u \chi_{a,b}''(u) + \left( \frac{\nu}{2} + u \right) \chi_{a,b}'(u) + \frac{k}{2}\chi_{a,b}(u). $$
Taking $a = \frac{k}{2}$ and $b = \frac{\nu - k}{2},$ we obtain (\ref{firstsupersol:pde}) by the previous lemma. The remaining formulae follow from (\ref{chilemma:chiupperlower}-\ref{chilemma:chi''est}).
\end{proof}

\begin{prop}\label{prop:secondsupersol}
Suppose $\nu > 2$ and let $0 < \ell \leq k < \nu$ with $\ell < \nu - 2.$ For $\eps > 0$ sufficiently small and each $\mu \geq \nu + \eps,$ 
the function
$$G(r,t) = G_{k, \ell, \nu, \eps}(r,t) = \frac{\eps^2}{(r^2 + 1)^{\frac{\ell}{2} } (r^2 + t + 1)^{\frac{k - \ell}{2}}} + F_{k, \nu}(r,t + 1)$$
satisfies
\begin{equation}\label{secondsupersol:evolutionlowerbound}
\left( \frac{\p}{\p t} - \Delta_{\mu} - \frac{ c_{k, \ell, \nu} \eps^2 } { r^2 + 1 } \right) G(r,t) > 0
\end{equation}
with
\begin{equation}\label{secondsupersol:Gupperlowerbound}
\frac{\eps^2}{(r^2 + 1)^{\frac{\ell}{2} } (r^2 + t + 1)^{\frac{k - \ell}{2}}} \leq G(r,t) \leq  \left( \frac{\eps^2}{(r^2 + 1)^{\frac{\ell}{2} }} + \frac{C_{k, \nu} }{ (r^2 + t + 1)^{\frac{\ell}{2}} }  \right) \frac{1}{ (r^2 + t + 1)^{\frac{k - \ell}{2}} },
\end{equation}
\begin{equation}\label{secondsupersol:dGdrbound}
\frac{c r G(r,t)}{r^2 + 1} \leq -\frac{\p G(r,t) }{\p r} \leq \frac{C r G(r,t)}{r^2 + 1} ,
\end{equation}
and
\begin{equation}\label{secondsupersol:d2Gdr2bound}
\left| \frac{\p^2 G(r,t) }{\p r^2} \right| \leq \frac{C G(r,t)}{r^2 + 1}.
\end{equation}
\end{prop}
\begin{proof} For $a> 0,$ we calculate
$$\frac{\p}{\p r} (r^2 + t + 1)^{-\frac{a}{2}} = \frac{-ar}{ (r^2 + t + 1)^{\frac{a}{2} + 1} },$$
\begin{equation*}
\begin{split}
\frac{\p^2}{\p r^2} (r^2 + t + 1)^{-\frac{a}{2} } & = \frac{a \left( (a + 2)r^2 - (r^2 + t + 1) \right) }{ (r^2 + t + 1)^{\frac{a}{2} + 2} } = \frac{a\left( (a + 1)r^2 - (t + 1) \right) }{ (r^2 + t + 1)^{\frac{a}{2} + 2} },
\end{split}
\end{equation*}
\begin{equation*}
\begin{split}
\Delta_\mu  (r^2 + t + 1)^{-\frac{a}{2} } & = \frac{a\left( (a + 1)r^2 - (t + 1) - (\mu -1)(r^2 + t + 1) \right) }{ (r^2 + t + 1)^{\frac{a}{2} + 2} } = \frac{a\left( (a + 2 - \mu)r^2 - \mu (t + 1) \right) }{ (r^2 + t + 1)^{\frac{a}{2} + 2} },
\end{split}
\end{equation*}
\begin{equation*}\label{secondsupersol:drequation}
\frac{\p}{\p r} (r^2 + 1)^{-\frac{\ell}{2}} (r^2 + t + 1)^{-\frac{k - \ell }{2}} = \frac{-r}{ (r^2 + 1)^{\frac{\ell}{2}} (r^2 + t + 1)^{\frac{k - \ell }{2} } } \left( \frac{\ell}{r^2 + 1} + \frac{k - \ell }{r^2 + t + 1}\right),
\end{equation*}
\begin{equation*}\label{secondsupersol:d2requation}
\frac{\p^2}{\p r^2} (r^2 + 1)^{-\frac{\ell}{2}} (r^2 + t + 1)^{-\frac{k - \ell }{2}} = \frac{\frac{\ell\left( (\ell + 1)r^2 - 1 \right) }{ (r^2 + 1)^{ 2} } + \frac{(k - \ell )\left( (k - \ell  + 1)r^2 - (t + 1) \right) }{ (r^2 + t + 1)^{2} } + \frac{2\ell (k - \ell ) r^2}{(r^2 + 1) (r^2 + t + 1)} }{ (r^2 + 1)^{\frac{\ell }{2}} (r^2 + t + 1)^{\frac{k - \ell }{2} } }
\end{equation*}
$$\Delta_\mu (r^2 + 1)^{-\frac{\ell }{2}} (r^2 + t + 1)^{-\frac{k - \ell }{2} } = \frac{ \frac{\ell \left( (\ell  + 2 - \mu)r^2 - \mu  \right) }{ (r^2 + 1)^{2} } + \frac{(k - \ell )\left( (k - \ell  + 2 - \mu)r^2 - \mu (t + 1) \right) }{ (r^2 + t + 1)^{2} } +  \frac{2\ell (k - \ell )r^2}{(r^2 + 1) (r^2 + t + 1)} }{ (r^2 + 1)^{\frac{\ell }{2}} (r^2 + t + 1)^{\frac{k - \ell }{2}} }$$
$$\left( \frac{\p}{\p t} - \Delta_\mu \right) (r^2 + 1)^{-\frac{\ell }{2}} (r^2 + t + 1)^{-\frac{k - \ell }{2} } = - \frac{ \frac{\ell \left( (\ell  + 2 - \mu)r^2 - \mu  \right) }{ (r^2 + 1)^{2} } + \frac{(k - \ell )\left( (k - \ell  + \frac{5}{2} - \mu)r^2 - \left( \mu - \frac12 \right) (t + 1) \right) }{ (r^2 + t + 1)^{2} } + \frac{2\ell (k - \ell ) r^2}{(r^2 + 1) (r^2 + t + 1)} }{ (r^2 + 1)^{\frac{\ell }{2}} (r^2 + t + 1)^{\frac{k - \ell }{2}} }.$$
Since
$$0 < \ell  < \nu - 2 < \mu - 2,$$
the first term on the right-hand side dominates the other two as long as
$$t + 1 \geq D_{k,\ell, \nu} r^2.$$
We obtain
\begin{equation*}
\left( \frac{\p}{\p t} - \Delta_\mu - \frac{d_{k, \ell, \nu}}{r^2 + 1} \right) \frac{1}{(r^2 + 1)^{\frac{\ell }{2}} (r^2 + t + 1)^{\frac{k - \ell }{2} } } \geq 0 
\end{equation*}
for $t + 1 \geq D_{k,\ell, \nu} r^2,$ and
\begin{equation*}
\left( \frac{\p}{\p t} - \Delta_\mu + \frac{E_{k, \ell, \nu}}{r^2 + t + 1} \right) \frac{1}{(r^2 + 1)^{\frac{\ell }{2}} (r^2 + t + 1)^{\frac{k - \ell }{2} } } \geq 0 
\end{equation*}
for general $(r,t).$ 
Next, since $\mu \geq \nu + \eps,$ (\ref{firstsupersol:evolutionlowerbound}) gives
\begin{equation*}
\left( \frac{\p}{\p t} - \Delta_\mu - \frac{c_{k, \nu} \eps}{r^2 + t + 1} \right) F_{k, \nu}(r,t+1) \geq 0.
\end{equation*}
Combining these, we have for $t + 1 \geq D_{k,\ell, \nu} r^2,$
\begin{equation*}
\begin{split}
\left( \frac{\p}{\p t} - \Delta_\mu \right) G(r,t) & \geq \frac{d_{k,\ell,\nu}}{r^2 + 1} \frac{\eps^2}{ (r^2 + 1)^{\frac{\ell }{2}} (r^2 + t + 1)^{\frac{k - \ell }{2} }} > \frac{d'_{k,\ell,\nu} \eps^2}{r^2 + 1} G
\end{split}
\end{equation*}
by (\ref{firstsupersol:Fbounds}). Meanwhile, for $t + 1 \leq D_{k,\ell,\nu}r^2,$ we have
\begin{equation*}
\begin{split}
\left( \frac{\p}{\p t} - \Delta_\mu \right) G(r,t) & \geq -\frac{E_{k,\ell,\nu}}{r^2 + 1} \frac{\eps^2}{ (r^2 + 1)^{\frac{\ell }{2}} (r^2 + t + 1)^{\frac{k - \ell }{2} }} + \frac{c_{k,\nu} \eps}{r^2 + t + 1} F_{k,v}(r, t + 1) \\
& \geq \frac{c'_{k,\nu} \left( 1 -E_{k,\ell,\nu} \eps \right) \eps}{r^2 + 1}  F_{k,v}(r, t + 1) \\
& > \frac{c''_{k,\nu} \eps}{r^2 + 1} G
\end{split}
\end{equation*}
for $\eps$ sufficiently small. Combining these gives (\ref{secondsupersol:evolutionlowerbound}).


The other bounds follow from the calculations above.
\end{proof}

    The above results can now be applied to develop barriers for ALE manifolds.
    
    \begin{lemma}[ALE Comparison Principle]\label{lemma:comparisonprinciple} Let $(M^n, g_0)$ be a Ricci-flat ALE space of dimension $n \geq 3$ and order $\tau > 0.$ 
    Fix coordinates at infinity $\psi : ( \R^n \setminus \bar{B}_1 ) / \Gamma \to M \setminus K  $.
    Given $0 < \ell \leq k < n$ with $\ell < n - 2,$ there exist $\eps_0, \gamma_0 > 0$ and $R_0 \geq 1$ as follows.

    Let $g(t),$ $t_0 \leq t < T,$ be any smooth 
    family of metrics with
    \begin{equation}\label{g(t)eps0closeness}
        \| g(t) - g_0 \|_{C^0} < \eps_0
    \end{equation}
    for all $t_0 \leq t < T.$ 
    Let $\underline{R}(t), \overline{R}(t)$ be any smooth functions with 
    $$R_0 \leq \underline{R}(t) \leq \overline{R}(t)$$
    for $t_0 \leq t \leq T,$ where we allow the possibility $\overline{R}(t) \equiv \infty.$
    Write
$$W = \left\{(x,t)  \in M \times [t_0, T) \mid  x \in \psi \left( \overline{B}_{\overline{R}(t) } \setminus B_{\underline{R}(t)}(0)  \right) \right\}.$$
Let $G(x,t) = G_{k,\ell, n - \gamma_0, \gamma_0}(r,t)$ be as above, where $r = r\left( \psi^{-1} (x) \right)$ is the radial coordinate at infinity.
    Suppose that $f(x,t)$ is a function satisfying
    \begin{equation}\label{comparisonprinciple:eps0evolution}
        \left( \frac{\p}{\p t} - \Delta_{g(t),g_0} - \frac{\eps_0}{r^2 + 1} \right) f \leq 0
    \end{equation}
    on $W,$ as well as
    \begin{equation}
        f(x,0) \leq A G(r,0)
    \end{equation}
    for $\underline{R}(0) \leq r \leq \overline{R}(0),$ where $A > 0,$ and
    \begin{equation}
    f(x,t) \leq A G(r, t) 
    \end{equation}
    for all $t_0 \leq t < T$ and $r = \underline{R}(t) \text{ or } \overline{R}(t).$    
    Then 
 \begin{equation}
        f(x,t) < A G(r,t) 
    \end{equation}
    for all $(x,t) \in W^\circ.$ 
    \end{lemma}

    \begin{proof} Let $0 < \eps < \gamma,$ where $\gamma$ is sufficiently small in the sense of Proposition \ref{prop:secondsupersol}, and assume that 
     $\| g(t) - g_0 \|_{C^0} < \eps.$
    Consider $G = G_{k, \ell, n - \gamma, \gamma} (r, t)$ as a radial function on the coordinate chart at infinity in $M.$ By the maximum principle, it suffices to check that $G$ is a strict supersolution for (\ref{comparisonprinciple:eps0evolution}) in $W.$

By Lemma \ref{lemma: difference of Laplacians} and (\ref{secondsupersol:dGdrbound}-\ref{secondsupersol:d2Gdr2bound}), we have
$$\Delta_{g_0} G = \Delta_n G + O \left(\frac{1}{r^{2 + \tau}} \right) G.$$
For $r \geq R_0$ sufficiently large, this gives
$$\Delta_{g_0} G \leq \Delta_n G + \frac{\eps}{r^{2}}  G.$$
By (\ref{g0inversegrelation}), we further have
$$\Delta_{g, g_0} G = \Delta_{g_0} G - g_0 h g^{-1} \nabla^{g_0,2} G.$$
Applying the bounds (\ref{secondsupersol:dGdrbound}-\ref{secondsupersol:d2Gdr2bound}) again gives
$$\Delta_{g, g_0} G \leq \Delta_{g_0} G + \frac{C \eps}{r^2} G \leq \Delta_{n} G + \frac{(C + 1) \eps}{r^2} G.$$
From (\ref{secondsupersol:evolutionlowerbound}), with $\mu = n,$ we now have
\begin{equation}
0 < \left( \frac{\p}{\p t} - \Delta_{n} - \frac{ c_{k, \ell, \nu} \gamma^2 } { r^2 + 1 } \right) G(r,t) \leq \left( \frac{\p}{\p t} - \Delta_{g, g_0} - \frac{ c_{k, \ell, \nu} \gamma^2 - C \eps} { r^2 + 1 } \right) G(r,t).
\end{equation}
For $\eps \leq \frac{c_{k,\ell, \nu} \gamma^2}{2 C},$ we have a supersolution. In the statement, we let $\gamma_0 = \gamma$ and let $\eps_0 = \eps.$ 
    \end{proof}

\vspace{10mm}

\section{$C^{2,\gamma}_{-\ell}$-stability} \label{sec: weighted Holder stability}

\subsection{Stability of Ricci-DeTurck flow}

\begin{rmk} 
    Throughout the remainder of the paper, unless indicated otherwise, we use $B_R$ to denote the metric ball $B_R := B_R(x_*) \subseteq M$ in the pointed Riemannian manifold $(M, g_0, x_*)$.
\end{rmk}

\begin{thm}\label{thm:Mestimate}
    Let $(M^n, g_0 , x_*),$ $\mu \in \{n-1, n\},$ $\gamma \in \left( 0,1 \right),$ and $n - \mu + 1 < \ell < n-2$ be as in the statement of Theorem \ref{main thm weighted Schauder spaces}, satisfying the assumptions (\ref{main thm weighted Schauder spaces, hypothesis stable})--(\ref{main thm weighted Schauder spaces, hypothesis integrable}). 
    Choose $\ell'$ such that 
    $$\max\{ 0, 2(n - \mu + 1) -\ell \} <  \ell' < \ell .$$
    There exists $\eps_1 > 0$ as follows.
    
    Let $\eps_0$ be the constant of Lemma \ref{lemma:comparisonprinciple}. Given any solution $g(t)$ of the Ricci-DeTurck flow on $\LB 0, T \right)$ with
    \begin{equation}
        \| g(0) - g_0 \|_{C^{2, \gamma}_{-\ell}} \leq \eps \leq \eps_1^2
    \end{equation}
    and
    \begin{equation}\label{Mestimate:C20bound}
        \sup_{0 \leq t < T} \| g(t) - g_0 \|_{C^{2}_0} < \eps_0,
    \end{equation}
    we have
        \begin{equation}\label{Mestimate:Mestimate}
        |\calM(x,t)| \leq \frac{\eps_1^{-1} \eps}{ \rho(x)^{\ell'} (\rho(x)^2 + t)^{\frac{\ell - \ell'}{2} + 1}}
    \end{equation}
    and
     \begin{equation}\label{Mestimate:gestimate}
        \| g(t) - g_0 \|_{C^{2, \gamma}_{-\ell}} \leq \eps_1^{-1} \eps
    \end{equation}
    for all $0 \leq t < T.$
\end{thm}
\begin{proof} 

Let $q > 1$ be small enough that $q(\ell + 2) < n.$ 
Let $\psi: (\R^n \setminus \bar{B}_1(0)) / \Gamma \to M \setminus K $ be coordinates at infinity for $M$.
Let $\varphi(x)$ be a smooth cutoff function supported on  $\psi\left( B_2(0) \right) \subseteq M$ with $\varphi(x) \equiv 1$ on $K$.
Define 
\begin{equation}\label{Mestimate:Gdef}
G(x,t) = G_{q(\ell + 2), q\ell', n - \gamma_0, \gamma_0} \left(\varphi(x) + (1 - \varphi(x) ) r(x),t \right),
\end{equation}
where $r(x)$ is the radial coordinate at infinity.
Then $G$ coincides with the function of Lemma \ref{lemma:comparisonprinciple} outside $\psi( B_2(0))$ and satisfies the bound (\ref{secondsupersol:Gupperlowerbound}) with $\rho(x)^2$ in place of $r^2 + 1.$

    Suppose for contradiction that (\ref{Mestimate:Mestimate}) fails for $\eps_1 = \frac{1}{i},$ $i \in \N,$ Ricci-DeTurck solutions $g(t) = g_i(t),$ $t \in \LB 0, T_i \right),$ and $\eps = \eps_i \leq \frac{1}{i^2}.$ 
    Write $h_i(t) = g_i(t) - g_0$ and $\calM_i(x,t) = \calM(g_i(t))(x).$
    Let $t_i > 0$ be the first time such that
    \begin{equation}\label{Mestimate:Mnonestimate}
        |\calM_i(x,t_i)|^q = (i \eps_i)^q G(x,t_i)
    \end{equation}
    for some $x \in M.$ 
Since $t_i$ is the first such time, $\calM_i(x,t)$ satisfy
    \begin{equation}\label{Mestimate:MiC2bound}
        | \calM_i(x,t) | \leq i \eps_i G(r,t)^{\frac{1}{q}} \leq \frac{C i \eps_i}{\rho(x)^{\ell'} (\rho(x)^2 + t)^{\frac{\ell - \ell'}{2} + 1} }
    \end{equation}
    for all $t \leq t_i.$ By Theorem \ref{thm:wellposedness}, we must have $t_i \to \infty$ as $i \to \infty.$
    
    By integrating (\ref{Mestimate:MiC2bound}) from $0$ to $t_i,$ we get
    \begin{equation}\label{Mestimate:hiC2-ellbound}
    \begin{split}
        |h_i(x,t) | & \leq |h_i(x,0)| + \int_0^t \left| \calM_i(x,s) \right| \, ds \\
        & \leq  \frac{\eps_i}{\rho(x)^{\ell} } + \frac{C i \eps_i}{\rho(x)^{\ell'} } \cdot \frac{1}{\rho(x)^{\ell - \ell' } } \leq \frac{C i \eps_i}{\rho(x)^{\ell } }.
        \end{split}
    \end{equation}
    By Lemma \ref{lemma:globalC2gamma-deltaestimate}, we can improve this $C^0$ bound immediately to
    \begin{equation}\label{Mestimate:hC2gammaellbound}
        \| h_i(t) \|_{C^{2,\gamma}_{-\ell}} \leq C i \eps_i \leq \frac{C}{i}.
    \end{equation}

    \vspace{2mm}

\noindent   \underline{ {\bf Case 1:} There exists a bounded sequence of $x_i \in M$ such that (\ref{Mestimate:Mnonestimate}) holds with $x = x_i.$}

Let $R_i = \sqrt{t_i}.$ We have
    \begin{equation}\label{Mestimate:Aidef}
        A_i : = |\calM_i(x_i,t_i) | = i \eps_i G(x_i,t_i)^\frac{1}{q} \geq \frac{d i \eps_i}{t_i^{\frac{\ell - \ell'}{2} + 1} } = \frac{d i \eps_i}{R_i^{\ell - \ell' + 2}  },
    \end{equation}
    since $x_i$ is uniformly bounded. Here $d > 0$ depends on the contradicting sequence.

We first estimate the projection of $\calM_i$ onto the kernel of the linearized operator. 
Choose any $\delta \in \left( n - \mu , \frac{\ell + \ell'}2 -1 \right).$ 
We have
\begin{equation}
    \label{Mestimate:kdeltaell}
    2 + 2 \delta < \ell + \ell'  \quad \text{and} \quad 
    n - \mu < \delta < \frac{\ell + \ell'}2 -1 < \ell - 1 < \ell < n-2.
\end{equation}
Notice that by Lemma \ref{lem integrability and changing decay rate},
the assumptions of Theorem \ref{main thm weighted Schauder spaces}, and so too of Theorem \ref{thm:ALEalmostorthog}, hold with $\delta$ in place of $\ell.$

By applying the interior estimate Lemma \ref{lem parab smoothing in weighted Holder} to $u = \mathcal M_i$, we can improve (\ref{Mestimate:MiC2bound}) to
    \begin{equation}\label{Mestimate:MiC2gammabound}
        \| \mathcal M_i(t_i) \|_{C^{0, \gamma}_{-\ell'} \left(B_{2R_i} \right) } \leq 
        \frac{C i \eps_i}{R_i^{\ell - \ell' + 2} } \qquad (\forall i \gg 1).
    \end{equation}
    Meanwhile, 
    (\ref{Mestimate:hC2gammaellbound}) gives
    \begin{equation}\label{Mestimate:MiC2gammaouterbound}
        \| \calM_i(t_i) \|_{C_{-\ell - 2}^{0,\gamma} \left( M \setminus B_{R_i} \right) } \leq C i \eps_i. 
    \end{equation}
Combining (\ref{Mestimate:MiC2gammabound}-\ref{Mestimate:MiC2gammaouterbound}), we get
    \begin{equation}\label{Mestimate:Mideltaest}
    \begin{split}
        \| \calM_i(t_i) \|_{C_{-\delta - 2}^{0,\gamma} } & \leq C \sup_{1 \leq r \leq 2R_i} r^{\delta + 2} \frac{i\eps_i r^{-\ell'}}{R_i^{\ell - \ell' + 2}} + C \frac{R_i^{\delta + 2} i\eps_i}{R_i^{\ell + 2} } \\
        & \leq C i \eps_i \max\{ R_i^{-(\ell - \ell' + 2)}, R_i^{\delta - \ell}  \} \\
        & \leq C i \eps_i R_i^{-\min \{\ell - \ell' + 2, \ell - \delta \}}.
        \end{split}
    \end{equation}

    Observe \eqref{Mestimate:hC2gammaellbound} and \eqref{Mestimate:kdeltaell} imply
        $$\| g_i (t) - g_0 \|_{C^{2, \gamma}_{-\delta}} \le \| g_i (t) - g_0 \|_{C^{2, \gamma}_{-\ell} } \le \frac C i.$$
    We can therefore apply Theorem \ref{thm:ALEalmostorthog} to deduce that, for all $i \gg1,$ 
    there exists $g_{1,i} \in C^{2, \gamma}_{-\delta}$ with $\mathcal M(g_{1, i} ) = 0$ such that,
    for any any unit element $e_i$ of $\ker L_{g_{1,i}},$ we have
    \begin{equation}
        \left( \calM_i, e_i \right)_{g_{1,i}} \leq C (i \eps_i)^2 R_i^{-2 \min \{\ell - \ell' + 2, \ell - \delta \}} \leq C i \eps_i R_i^{-2 \min \{\ell - \ell' + 2, \ell - \delta \}}
    \end{equation}
    since $i \eps_i \leq 1.$ Observe that $2 \ell - 2 \delta > \ell - \ell' + 2$ by (\ref{Mestimate:kdeltaell}). In view of (\ref{Mestimate:Aidef}), we therefore have
    \begin{equation}\label{Mestimate:almostorthog}
    \left( \calM_i, e_i \right)_{g_{1,i}} 
    = o(A_i)
    \end{equation}
    as $i \to \infty.$

    We now consider the solution
    $$f_i(x,s) = \frac{ \calM_i(x, t_i + s)}{A_i}$$
    of (\ref{Mevolution}), defined for $s \in \LB - t_i, 0 \RB.$ The definition (\ref{Mestimate:Aidef}) gives
    \begin{equation}
|f_i(x_i, 0)| = 1
    \end{equation}
    and the bound (\ref{Mestimate:MiC2bound}) goes over to
    \begin{equation}
    \begin{split}
        | f_i(x,s) | \leq \frac{C i \eps_i}{A_i \rho(x)^{\ell' }(\rho(x)^2 + t_i + s)^{\frac{\ell - \ell'}{2} + 1} } 
        & \leq \frac{C}{d \rho(x)^{\ell' } \left(1 + \frac{s}{t_i} \right)^{\frac{\ell - \ell'}{2} + 1} },
        \end{split}
    \end{equation}
    while the estimate (\ref{Mestimate:almostorthog}) gives
    \begin{equation}\label{Mestimate:innerproducttendstozero}
    \left( f_i, e_i \right)_{g_{1,i}} \to 0
    \end{equation}
    as $i \to \infty.$ We now pass to a subsequential limit $x_i \to x_\infty,$ $f_i \to f$ in $C^\infty_{loc}\left( M \times \left( -\infty, 0 \right] \right)$ as $i \to \infty.$ In view of (\ref{Mestimate:hC2gammaellbound}), we have $g_i(t_i) \to g_0$ as $i \to \infty.$ We obtain an ancient solution $f(x,s)$ of (\ref{Mevolution}) with $g(t) \equiv g_0,$ which satisfies
    \begin{equation}
        | f(x,s) | \leq \frac{C}{\rho(x)^{\ell'} }, \qquad |f(x_\infty, 0)| = 1, \qquad
    \left( f, e \right) = 0
    \end{equation}
    for all $e \in \ker L_{g_0}.$\footnote{Here we have let $e_i \to e$ in $C^0_{-\mu}$ 
    by the argument of Kr\"oncke-Petersen \cite[Prop. 2.3]{KronckePetersen20}, which allows us to pass to the limit in (\ref{Mestimate:innerproducttendstozero}).} But this contradicts the Liouville Theorem \ref{thm:liouville}.

\vspace{2mm}
    
\noindent   \underline{ {\bf Case 2:} There exists no bounded sequence of $x_i \in M$ such that (\ref{Mestimate:Mnonestimate}) holds with $x = x_i.$}

In particular, given any $R_1 > 0,$ for $i$ sufficiently large, we have
\begin{equation}
| \calM_i(x,t) |^q < (i \eps_i)^q G(x,t)
\end{equation}
for all $x \in \bar{B}_{2R_1}$ and $0 \leq t \leq t_i.$ But, choosing $R_1 \geq R_0$ (from Lemma \ref{lemma:comparisonprinciple}) such that the evolution equation (\ref{comparisonprinciple:eps0evolution}) is satisfied by $|\calM_i|,$ and taking $\underline{R}(t) \equiv R_1$ and $\overline{R}(t) \equiv \infty,$ the comparison principle (Lemma \ref{lemma:comparisonprinciple}) gives
\begin{equation}
| \calM_i(x,t) |^q < (i \eps_i)^q G(x,t)
\end{equation}
for all $x \in M \setminus \bar{B}_{R_1}$ and $0 \leq t \leq t_i.$ This contradicts (\ref{Mestimate:Mnonestimate}), completing the proof of (\ref{Mestimate:Mestimate}).

In view of the $C^0_{-\ell}$-bound (\ref{Mestimate:hiC2-ellbound}), we can now use Lemma \ref{lemma:globalC2gamma-deltaestimate} to deduce (\ref{Mestimate:gestimate}), completing the proof.
\end{proof}

    \begin{cor}\label{cor:existence}
    Let $\eps_1 > 0$ be as in the previous theorem. Given any initial metric $g(0)$ with
        \begin{equation}
        \| g(0) - g_0 \|_{C^{2, \gamma}_{-\ell}} \leq \eps \leq \eps_1^2,
    \end{equation}
    there exists a complete and convergent Ricci-DeTurck solution $g(t)$ with respect to $g_0$ and initial data $g(0)$
    satisfying (\ref{Mestimate:Mestimate}-\ref{Mestimate:gestimate}), as well as a metric $g_\infty$ with $\calM(g_\infty) = 0$ such that
    \begin{equation}\label{existence:Ckgammaest}
        \| g(t) - g_\infty \|_{C^{k,\gamma}_{-\ell'}(B_{\sqrt{t}})} \leq C_{k,\gamma, g_0} \eps t^{-\frac{\ell - \ell'}{2}}
    \end{equation}
    for each $k \in \N.$ Here the norms may be taken with respect to $g_0$ or $g_\infty.$
\end{cor}
\begin{proof}
    According to Theorem \ref{thm:wellposedness}, there exists such a solution with bounded $h$ for $T > 0.$ Applying Lemma \ref{lemma:globalC2gamma-deltaestimate} with $\delta = 0,$ we see that (\ref{Mestimate:C20bound}) is satisfied; in particular, we can take $T$ to be the maximal such time. But then the previous Theorem gives the bounds (\ref{Mestimate:Mestimate}-\ref{Mestimate:gestimate}) on $ \left[ 0, T \right).$ Since (\ref{Mestimate:gestimate}) implies (\ref{Mestimate:C20bound}), we must have $T = \infty.$ The estimate (\ref{existence:Ckgammaest}) then follows by integrating (\ref{Mestimate:Mestimate}) in time and applying parabolic regularity.
\end{proof}

\subsection{Stability of Ricci flow}

To pass from Ricci-DeTurck flow back to Ricci flow, we need a more precise estimate on the DeTurck vector field. We use an elliptic approach along the lines of Deruelle-Kr\"oncke \cite[Prop. 2.6]{DeruelleKroncke20}.

\begin{lemma}\label{lemma:ellipticdeTurck}
    The DeTurck vector field $V = V(g,g_0)$ satisfies
    \begin{equation*}
        -\Delta_g |V|_g \leq |\Ric(g)||V|_g + \left|\left( \div - \frac12 \nabla \tr \right) \calM \right|.
    \end{equation*}
\end{lemma}
\begin{proof}
    By definition, we have
    $$\Ric(g) = \calL_{\frac12 V} g - \frac12 \calM.$$
    Taking the trace, we get
    $$R_g = \div V - \frac12 \tr_g \calM.$$
    Also note that
    \begin{equation*}
    \begin{split}
        \div \calL_V g & = \nabla^i \nabla_i V_j + \nabla^i \nabla_j V_i \\
        & = \Delta V + \Ric(g)(V) + \nabla \div V \\
        & = \Delta V + \Ric(g)(V) + \nabla \left( R_g + \frac12 \tr_g \calM \right).
        \end{split}
    \end{equation*}
Now, the contracted Bianchi identity reads
\begin{equation*}
\begin{split}
    \frac12 \nabla R = \div \Ric & = \div \left(\calL_{\frac12 V} g - \frac12 \calM \right) \\
    & = \frac12 \left( \Delta V + \Ric(g)(V) + \nabla R + \frac12 \nabla \tr_g \calM\right) - \frac12 \div \calM.
    \end{split}
\end{equation*}
Canceling $\frac12 \nabla R$ from both sides, we have
\begin{equation}\label{Vequation}
    \Delta V + \Ric(g)(V) = \left( \div - \frac12 \nabla \tr \right) \calM.
\end{equation}
We can now get the result by applying the standard Kato inequality.
\end{proof}

\begin{lemma}
Let $R \geq 1$ and $2 \neq \beta < n.$ There exists a positive solution $u$ on $B_R{(x_*)}$ of the equation
$$-\Delta_{g_0} u = \rho(x)^{-\beta}$$
satisfying
$$|u(x)| \lesssim \begin{cases} R^{2 - \beta} & \beta < 2 \\ \rho(x)^{2 - \beta} & \beta > 2. \end{cases}$$
\end{lemma}
\begin{proof}
    Let $G(x,y)$ be the Green's function of $g_0.$ This satisfies
    $$0 \leq G(x,y) \leq \frac{C}{d(x,y)^{2-n}},$$
    as can be seen by integrating (\ref{heatkernelest}) in time. We let
    $$u(x) = \int_{B_R} G(x,y) \rho(y)^{-\beta} \, dV_y,$$
    which is the required solution. 
    Recall from the definition of ALE space (Definition \ref{defn: ALE}) that there exists a coordinate system at infinity
        $$\psi : ( \R^n \setminus \overline B_1 ) / \Gamma \to M \setminus K $$
    where $K \subseteq M$ is compact.
    We now estimate $u$ as follows.

    \vspace{2mm}

   \noindent \underline{{\bf Case 1.} $x \not\in K.$} We have
    \begin{equation}\label{udecomp}
    |u(x)| \lesssim \left( \int_{B_{\rho(x)/2}} + \int_{B_{2\rho(x)} \setminus B_{\rho(x){/2}}} + \int_{B_R \setminus B_{2 \rho(x)} } \right) \frac{dV_y}{d(x,y)^{n-2} \rho(y)^\beta} =: I + II + III.
    \end{equation}
    For the first term, we have
    $$I \lesssim \int_{B_{\rho(x)/2}} \frac{1}{\rho(x)^{n-2}} \frac{dV_y}{\rho(y)^\beta} \lesssim \frac{1}{\rho(x)^{n-2}} \int_{1}^{\rho(x)/2} r^{n-1 - \beta} dr \lesssim \rho(x)^{2 - \beta}.$$
    For the second term, we have
    $$II \lesssim \frac{1}{\rho(x)^\beta} \int_{B_{2\rho(x)} \setminus B_{\rho(x){/2}}} \frac{dV_y}{d(x,y)^{n-2}}.$$
    By Bishop-Gromov, we can estimate this by an integral on $\R^n:$
    $$\int_{B_{2\rho(x)} \setminus B_{\rho(x){/2}}} \frac{dV_y}{d(x,y)^{n-2}} \lesssim \int_{0}^{4 \rho(x)} \frac{r^{n-1} dr}{r^{n-2}} \lesssim \rho(x)^2.$$
    This gives
    $$II \lesssim \rho(x)^{2 - \beta}.$$
    For the third term, we have
    $$III \lesssim \int_{2\rho(x)}^R \frac{r^{n-1} dr}{r^{n - 2 + \beta}} \lesssim \int_{2\rho(x)}^R r^{1 - \beta} dr \lesssim \begin{cases} R^{2 - \beta} & \beta < 2 \\ \rho(x)^{2 - \beta} & \beta > 2. \end{cases}$$
    The desired estimate now follows by combining the above.

\vspace{2mm}
    
    \noindent \underline{{\bf Case 2.} $x \in K.$} In lieu of (\ref{udecomp}), we can decompose $u(x)$ as
        \begin{equation*}
    |u(x)| \lesssim \left( \int_{K} + \int_{B_R \setminus K } \right) \frac{dy}{d(x,y)^{n-2} \rho(y)^\beta}.
    \end{equation*}
    The first term is estimated as in $II$ above, and the second term as in $III$ above.
\end{proof}

\begin{prop}\label{prop:ellipticestonV}
    Fix $R \geq 1$ and $0 < \ell, \ell' <n-2 $ such that $\ell' \ne 1$. 
    Let $V = V(g,g_0)$ be the DeTurck vector field for a metric $g$ on $B_{2R} \subset M$ with $\|g - g_0\|_{C^{2,\gamma}_{-\ell}(B_{2R})} < \eps_0.$ We have
    $$\| V \|_{C^{2,\gamma}_0(B_{R})} \leq C \left( \sup_{x \in B_{2R} \setminus B_R} |V(x)| + \| \calM \|_{C^{1,\gamma}_{- \ell'}(B_{2R})} \left( 1 + R^{1 - \ell'} \right) \right).$$
\end{prop}
\begin{proof}
Let
    $$A = \sup_{x \in B_{2R} \setminus B_R} |V(x)|, \qquad B = \| \calM \|_{C^{1,\gamma}_{- \ell'}(B_R)},$$
    and
    $$D = \| V \|_{C^{2, \gamma}_0(B_{R})}.$$
    Note that by applying standard elliptic regularity on rescaled balls (as in the proof of Lemma \ref{lem parab smoothing in weighted Holder}), from (\ref{Vequation}), we have
    \begin{equation}\label{Dsupestimate}
        D \leq { C \big(} \sup_{B_{2R}} |V| + \| \nabla \calM \|_{C^{0,\gamma}_{0}} {\big)}
        \qquad { (C = C(M^n, g_0, x_*, \gamma ) > 0)}.
    \end{equation}
    Now, we have
    $$\Delta_{g_0} |V|_{g} = \Delta_g |V| + (g^{-1} - g_0^{-1}) \nabla^2 |V| + (\Gamma_g - \Gamma_{g_0}) \# \nabla |V|.$$
    Combining with Lemma \ref{lemma:ellipticdeTurck}, we get 
    $$-\Delta_{g_0} |V|_{g} \leq \frac{{C_n} B}{\rho(x)^{1 + \ell'}} + \frac{{C_n} \eps_{0} D}{\rho(x)^{2 + \ell}}$$
    where $C_n$ here and below denotes a dimensional constant that may change from line to line.
    We can now use the previous lemma to construct positive functions $u_1$ and $u_2$ on $B_{2R}$ satisfying
    $$-\Delta_{g_0} u_1 = \frac{1}{\rho(x)^{1 + \ell'}}, \qquad \sup_{B_{2R}} u_1 {\lesssim} R^{1 - \ell'}$$
    and
    $$-\Delta_{g_0} u_2 = \frac{1}{\rho(x)^{2 + \ell}}, \qquad u_2 (x) {\lesssim} \rho(x)^{-\ell}.$$
    We then have
    $$-\Delta_{g_0} \left( |V|_g - {C_n} B u_1 - {C_n} \eps_{0} D u_2 \right) \leq 0.$$
    Applying the maximum principle, we get
    $$\sup_{B_{2R}} |V|_g \leq A + {C}B R^{1 - \ell'} + {C}\eps_{0} D.$$
    Plugging back in to (\ref{Dsupestimate}), we get
    $$D \leq {C \big( }A + B (1 + R^{1 - \ell'}) + \eps_0 D { \big)}.$$
    Rearranging gives the desired estimate.
\end{proof}
    
    \begin{cor}\label{cor:Ricciexistence}
In the setup of Corollary \ref{cor:existence}, there exists a $C^\infty_{loc}$-convergent family of diffeomorphisms $\theta_t$ such that $\theta_t^*g(t)$ solves the Ricci flow and converges in $C^\infty_{loc}$ to a Ricci-flat metric $g'_\infty$ diffeomorphic to $g_\infty,$ with 
$\|g_\infty' - g_0 \|_{C^{0,\gamma}_{-\ell}} \leq C \eps.$ 
\end{cor}
\begin{proof} Assume without loss of generality that $\ell' < 1.$ 
First note that since the DeTurck vector field $V(t) = V(g(t), g_0)$ corresponds to one derivative of $g(t),$ we have
\begin{equation}\label{Vgeneraltestimate}
    \|V(t)\|_{C^{1,\gamma}_{-\ell - 1}} \leq C \eps
    \qquad (C = C(M^n, g_0, x_*, \gamma, \ell) > 0)
\end{equation}
for all times $t \geq 0.$ In particular, we have
\begin{equation*}
\sup_{B_{2\sqrt{t}} \setminus B_{\sqrt{t}/2}} |V(t)| \leq  C \eps t^{- \frac{\ell + 1}{2}}.
\end{equation*}
Note that from (\ref{Mestimate:Mestimate}) combined with parabolic regularity, 
for $t \geq 1,$ we have
\begin{equation*}
\begin{split}
\|\calM(t) \|_{C^{1, \gamma}_{-\ell'}(B_{2\sqrt{t}} )} & \leq \frac{C\eps}{t^{\frac{\ell - \ell'}{2} + 1}}.
\end{split}
\end{equation*}
Plugging in the last two estimates to Proposition \ref{prop:ellipticestonV} with $R = \sqrt{t} \geq 1,$ we obtain
\begin{equation*}
    \|V(t) \|_{C^{2,\gamma}_0(B_{\sqrt{t}})} \leq C \eps \left( t^{- \frac{\ell + 1}{2}} + t^{-1 - \frac{\ell - \ell'}{2}} \left( 1 + t^{\frac{1 - \ell'}{2}} \right) \right).
\end{equation*}
Observe that $-1 - \frac{\ell  - \ell'}{2} + \frac{1 - \ell'}{2} = - \frac{1 + \ell}{2}.$ So we obtain
\begin{equation}\label{Vpostparabestimate}
    \|V(t) \|_{C^{2,\gamma}_0(B_{\sqrt{t}})} \leq C \eps t^{- \min \{ \frac{\ell + 1}{2} , \frac{\ell - \ell'}{2} + 1 \} } 
    \qquad ( \forall t \ge 1).
\end{equation}
Using (\ref{Vgeneraltestimate}) and (\ref{Vpostparabestimate}),  for any $x \in M,$ we obtain
\begin{equation}\label{Vfirstestimate}
\begin{split}
    & \int_{0}^\infty \|V(t) \|_{C^{1, \gamma}_0\left( B\left(x, \frac12 \rho(x) \right) \right) } \, dt \\
    \leq{}& \int_{0}^{\left( \frac32 \rho(x)\right)^2} C \eps \rho(x)^{-1-\ell} \, dt + \int_{\left( \frac32 \rho(x)\right)^2}^\infty C \eps t^{- \min \{ \frac{\ell + 1}{2} , \frac{\ell - \ell'}{2} + 1 \} } \, dt \\
    \leq{}& C \eps \rho(x)^{1 - \ell} + C \eps \rho(x)^{2 - \min\{\ell + 1, \ell - \ell' +2\}} \\
    \leq{}& C \eps \rho(x)^{1 - \ell}.
    \end{split}
\end{equation}

We now integrate $V(t)$ in time to obtain a family of diffeomorphisms $\theta_t$ 
 satisfying
    $$\partial_t \theta_t = V(t) \circ \theta_t , \qquad \theta_0 = \mathbbm{1}.$$
It is clear from (\ref{Vfirstestimate}) that for $\eps$ sufficiently small, $\theta_t$ remains $C^{1,\gamma}_{1-\ell}$-close to the identity. Abusing notation, we can write
$$\|\theta_t - \mathbbm{1} \|_{C^{1, \gamma}_{1 - \ell}} \leq C \eps.$$

We also have $\theta_t \to \theta_\infty$ in $C^\infty_{loc}$ by \eqref{existence:Ckgammaest}, and set $g'_\infty := \theta_\infty^* g_\infty.$ We then have $\theta_t^*g(t) \to g'_\infty$ in $C^\infty_{loc}$ by \eqref{existence:Ckgammaest}. Since $\theta_t^*g(t)$ involves one derivative of $\theta_t$, we also have
$$\| \theta_t^*g(t) - g_0 \|_{C^{0, \gamma}_{-\ell}} \leq C \eps,$$
which implies the same of $g'_\infty.$
\end{proof}

\vspace{10mm}

\section{$L^p$-stability}\label{sec:Lpstability} 

To adapt our proof from the Schauder setting to the $L^p$ setting, we will use the following pointwise estimates on the scalar heat kernel in combination with a Duhamel-type formula. 

\begin{lemma}\label{lemma:heatkernelest}
    Let $H(x,y,t)$ denote the scalar heat kernel on $(M, g_0),$ a Ricci-flat ALE space. We have
    \begin{equation}\label{heatkernelest}
        H(x,y,t) \leq \frac{Ce^{-\frac{c d(x,y)^2}{t}} }{t^{\frac{n}{2}}} 
    \end{equation}
    and
    \begin{equation}\label{Kotschwar}
        \left| \nabla H(x,y,t) \right| \leq C \left( \frac1{\sqrt t} + \frac{d(x,y)}{t} \right)  H(x,y,t). 
    \end{equation}
\end{lemma}
\begin{proof}
    The first is a direct consequence of the Li-Yau inequality in nonnegative Ricci curvature, see e.g. \cite[Theorem 13.4]{LiBook}. The second is a sharpened version of an estimate of Souplet-Zhang \cite{SoupletZhang06} 
    due to Kotschwar \cite{Kotschwar07}.
\end{proof}

\begin{lemma}\label{lemma:Duhamel} 
Fix $q > 1.$ Let $h$ solve Ricci-DeTurck flow with respect to $g_0,$ and write $|h|_{g_0} = |h|$ and $\nabla = \nabla^{g_0}.$ Suppose that (\ref{gg0lambdaassumption}) is satisfied, where $\lambda$ is sufficiently close to one (depending only on $n, q$).
Then
    \begin{multline*}
        |h|^q(x,t)  \leq \int_M H(x,y,t) |h|^q(y,0) \, dV_y \\
         + C_{\ref{lemma:Duhamel}} \int_0^t \!\!\!\! \int_M H(x,y,t - s) \left( |\Rm(g_0)| |h|^q + \left( \frac1{\sqrt {t - s} } + \frac{d(x,y)}{t-s} \right) |h|^{\frac{q}{2} + 1} |\nabla |h|^{\frac{q}{2}} | - c_{\ref{lemma:Duhamel}} |\nabla |h|^{\frac{q}{2}} |^2 \right) \, dV_y ds
    \end{multline*}
    where 
    $C_{\ref{lemma:Duhamel}}, c_{\ref{lemma:Duhamel}}$ are constants depending only on $M, g_0, q$.
\end{lemma}
\begin{proof}
    Applying Duhamel's principle (i.e. integrating against $H(x,y,t-s)$ and using Green's formula) in (\ref{g0evolutionof|h|^q}), and integrating by parts in the final term, we have
            \begin{equation}\label{Duhamel}
        \begin{split}
        & |h|^q(x,t)  \leq \int_M H(x,y,t) |h|^q(y,0) \, dV_y \\
        & \qquad + C_q \int_0^t \!\!\!\! \int_M H(x,y,t - s) \left( |\Rm(g_0)| |h|^q - c_q |\nabla |h|^{\frac{q}{2}} |^2 \right) + \left| \nabla H(x,y,t - s) \right| |h|^{\frac{q}{2} + 1} |\nabla |h|^{\frac{q}{2}} | \, dV_y ds.
        \end{split}
    \end{equation}
    The desired formula follows by inserting (\ref{Kotschwar}).
\end{proof}

We need the following straightforward adaptation of Theorem \ref{thm:Mestimate}.

\begin{prop}\label{prop:cutoffMestimate}
        Let $(M,g_0)$ be as above, where we assume the decay order $\mu = n.$ Given $1< \ell < n-2,$ choose $\max\{ 0, 2-\ell \} <  \ell' < \ell .$ There exists $\eps_1 > 0$ as follows.
    
    Let $0 < \eps \leq \eps_1^2.$ Given any solution $g(t) = g_0 + h(t)$ of the Ricci-DeTurck flow on $\left[ 0, T \right)$ with
    \begin{equation}\label{cutoffMestimate:C0bound}
        \sup_{\frac12 \le t < T} \|h(\cdot ,t) \|_{C^0 ( B_{2 \sqrt t} ) } < \eps_1,
    \end{equation} 
    as well as
    \begin{equation}\label{cutuffMestimate:C20bound}
        \sup_{\frac12 \le t < T} \| h(\cdot, t) \|_{C^0_{-\ell}( B_{2 \sqrt t} \setminus \overline B_{\frac12 \sqrt t})}
        + \sup_{\frac12 \le t \le 1} \| h(\cdot, t) \|_{C^0( B_{2}) }
        \le \eps,
    \end{equation}
    then we have both
    \begin{equation}\label{cutoffMestimate:Mestimate}
        |\calM(x,t)| \leq \frac{\eps_1^{-1} \eps }{ \rho(x)^{\ell'} t^{\frac{\ell - \ell'}{2} + 1}}
        \qquad \forall x \in B_{\sqrt t} , \, 1 \le t < T,
    \end{equation}
    and
    \begin{equation} \label{cutoffMestimate:gestimate}
        \sup_{1 \le t < T} \| h (\cdot, t) \|_{C^{2, \gamma}_{-\ell}( B_{\sqrt t} ) } \le \eps_1^{-1} \eps.
    \end{equation}
\end{prop}
\begin{proof} We argue as in the proof of Theorem \ref{thm:Mestimate}. Let $q > 1$ and $G$ be as in (\ref{Mestimate:Gdef}).

    Suppose for contradiction that (\ref{cutoffMestimate:Mestimate}) fails for $\eps_1 = \frac{1}{i},$ $i \in \N,$ Ricci-DeTurck solutions $g(t) = g_i(t),$ $t \in \left[ 0, T_i \right),$ and $\eps = \eps_i \leq \frac{1}{i^2}.$ 
    Write $h_i(t) = g_i(t) - g_0$ and $\calM_i(x,t) = \calM(g_i(t))(x).$
    By Lemma \ref{lemma:localC2gamma-deltaestimate} \eqref{lem local C2gamma-delta est, item local in time}, we can pass from (\ref{cutuffMestimate:C20bound}) to 
    \begin{equation}\label{cutuffMestimate:C220bound}
       \sup_{\frac34 \le t < T} \| h_i(t)\|_{C^{2,\gamma}_{-\ell}\left( B_{\sqrt{2t}} \setminus B_{\sqrt{t/2} } \right) } 
       + \sup_{\frac34 \le t \le 1} \| h_i(t)\|_{C^{2,\gamma}_{-\ell}\left( B_{\sqrt{2t}}   \right) } 
       \leq C\eps_i.
    \end{equation}
    In particular, this gives
    \begin{equation}\label{cutoffMestimate:Mboundatparabolicscale}
         \sup_{r = \sqrt{t}} | \calM_i (x, t) | \leq C \eps_i t^{- \frac{\ell}{2} - 1}
        \qquad { \forall \frac34 \le t < T}.
 \end{equation}
    
    Let $t_i \ge 1$ be the first time such that
    \begin{equation}\label{cutoffMestimate:Mnonestimate}
        |\calM_i(x,t_i)|^q = (i \eps_i)^q G(x,t_i)
    \end{equation}
    for some $x \in B_{\sqrt{t_i}}.$ 
Since $t_i$ is the first such time, $\calM_i(x,t)$ satisfies
    \begin{equation}\label{cutoffMestimate:MiC2bound}
        | \calM_i(x,t) | \leq i \eps_i G(x,t)^{\frac{1}{q}} \leq \frac{C i \eps_i}{\rho(x)^{\ell' }(\rho(x)^2 + t)^{\frac{\ell - \ell'}{2} + 1} }
    \end{equation}
    for $1 \leq t \leq t_i$ and $x \in B_{\sqrt{t}}.$ In view of (\ref{cutuffMestimate:C220bound}), by the comparison principle (Lemma \ref{lemma:comparisonprinciple}), we have
    $$\sup_{B_{\sqrt{t}}}|\calM_i(x,t)| \leq C \eps_i e^{Kt} 
    \qquad { \forall \frac34 \le t < T},$$
    so we must have $t_i \to \infty$ as $i \to \infty.$
    
    By integrating (\ref{cutoffMestimate:MiC2bound}) from $t = \max \{ r(x)^2, 1\}  { \ge 1}$ to $t_i,$ we get
    \begin{gather}\label{cutoffMestimate:hiC2-ellbound}
    \begin{aligned}
        |h_i(x,t) | & \leq |h_i(x,\max \{ r(x)^2, 1\} )| + \int_{\max\{ r(x)^2, 1\} }^t \left| \calM_i(x,s) \right| \, ds \\
        & \leq  \frac{\eps_i}{\rho(x)^{\ell} } + \frac{C i \eps_i}{\rho(x)^{\ell'} } \cdot \frac{1}{\rho(x)^{\ell - \ell' } } 
        && ( \text{using \eqref{cutuffMestimate:C20bound}} )\\
        & \leq \frac{C i \eps_i}{\rho(x)^{\ell} }
        \end{aligned}
    \end{gather}
    for all $(x,t)$ such that $ 1 \le t \le t_i$ and $x \in B_{\sqrt t}$.

    By Lemma \ref{lemma:localC2gamma-deltaestimate} 
    together with the bounds \eqref{cutuffMestimate:C220bound},
    we can improve this $C^0_{-\ell}$ bound \eqref{cutoffMestimate:hiC2-ellbound} to
    \begin{equation}\label{cutoffMestimate:hC2gammaellbound}
        \sup_{1 \le t \le t_i} \| h_i(t) \|_{C^{2,\gamma}_{-\ell}(B_{\sqrt{t}})} \leq C i \eps_i \leq \frac{C}{i}.
    \end{equation}

    \vspace{2mm}

\noindent   \underline{ {\bf Case 1:} There exists a bounded sequence of $x_i \in M$ such that (\ref{cutoffMestimate:Mnonestimate}) holds with $x = x_i.$}

Let $R_i = \sqrt{t_i}.$ We have
    \begin{equation}\label{cutoffMestimate:Aidef}
        A_i : = |\calM_i(x_i,t_i) | = i \eps_i G(x_i,t_i)^\frac{1}{q} \geq \frac{d i \eps_i}{t_i^{\frac{\ell - \ell'}{2} + 1} } = \frac{d i \eps_i}{R_i^{\ell - \ell' + 2}  },
    \end{equation}
    since $x_i$ is uniformly bounded. Here $d > 0$ depends on the contradicting sequence.

    We first estimate the projection of $\calM_i$ onto the kernel of the linearized operator. Choose $0 < \delta < n - 2$ 
such that
\begin{equation}\label{cutoffMestimate:kdeltaell}
2 + 2 \delta < \ell + \ell'.
\end{equation}
Let $\chi_i(r) = \chi\left( \frac{r}{ R_i } \right)$ be a standard cutoff for $B_{R_i} \Subset B_{2 R_i},$ with $|\nabla^k \chi_i| \leq \frac{C_k}{ R_i^{k} }$ and supported on $B_{2 R_i} \setminus B_{R_i}$ for $k \geq 1.$ Define
$$\tilde{h}_i = \chi_i h_i, \qquad \tilde{g_i} = g_0 + \tilde{h}_i, \qquad \tilde{\calM}_i = \calM(\tilde{g}_i).$$
We have
\begin{equation}\label{cutoffMestimate:tildecalMformula}
    \tilde{\calM}_i = \chi_i \calM_i + \nabla \chi_i \# \nabla h_i + \nabla^2 \chi_i \# h_i.
\end{equation}
Observe that by applying a standard parabolic estimate, we can improve (\ref{cutoffMestimate:MiC2bound}) to
    \begin{equation}\label{cutoffMestimate:MiC2gammabound}
        \| \calM_i(t_i) \|_{C_{-\ell}^{0,\gamma} \left( B_{2R_i} \right) } \leq \frac{C i \eps_i}{R_i^{\ell - \ell' + 2} }.
    \end{equation}
    Applying (\ref{cutoffMestimate:hC2gammaellbound}) and (\ref{cutoffMestimate:MiC2gammabound}) to (\ref{cutoffMestimate:tildecalMformula}), we get
    \begin{equation}\label{cutoffMestimate:Mideltaest}
    \begin{split}
        \| \tilde{\calM}_i(t_i) \|_{C_{-\delta - 2}^{0,\gamma} } & = C \sup_{1 \leq r \leq 2R_i} r^{\delta + 2} \frac{i\eps_i r^{-\ell}}{R_i^{\ell - \ell' + 2}} + C \frac{ R_i^{\delta + 2} (i\eps_i)}{ R_i^{\ell + 2} } \\
        & \leq C i \eps_i \max\{ R_i^{-(\ell - \ell' + 2)}, R_i^{\delta - \ell} \} \\
        & \leq C i \eps_i R_i^{-\min \{\ell - \ell' + 2, \ell - \delta \}}.
        \end{split}
    \end{equation}
    We now apply Theorem \ref{thm:ALEalmostorthog}. Letting $g_{1,i}$ be the $L^2$-nearest Ricci-flat metric to $\tilde{g}_i,$ and $v_i$ be any unit element of $\ker L_{g_{1,i}},$ we get
    \begin{equation}
        \left( \tilde{\calM}_i, v_i \right)_{g_{1,i}} \leq C (i \eps_i)^2 R_i^{-2 \min \{\ell - \ell' + 2, \ell - \delta \}} \leq C i \eps_i R_i^{-2 \min \{\ell - \ell' + 2, \ell - \delta \}}
    \end{equation}
    since $i \eps_i \leq 1.$ Observe that $2 \ell - 2 \delta > \ell - \ell' + 2$ by (\ref{cutoffMestimate:kdeltaell}). In view of (\ref{cutoffMestimate:Aidef}), we therefore have
    \begin{equation}\label{cutoffMestimate:almostorthog}
    \left( \tilde{\calM}_i, v_i \right)_{g_{1,i}} \leq C i \eps_i R_i^{-2 \min \{k - \ell, \ell - \delta \}} = o(A_i)
    \end{equation}
    as $i \to \infty.$

    We now consider the solution
    $$f_i(x,s) = \frac{\tilde{\calM}_i(x, t_i + s)}{A_i}$$
    of (\ref{Mevolution}), defined for $s \in \LB - t_i, 0 \RB.$ The definition (\ref{cutoffMestimate:Aidef}) gives
    \begin{equation}
|f_i(x_i, 0)| = 1
    \end{equation}
    and the bound (\ref{cutoffMestimate:MiC2bound}) goes over to
    \begin{equation}
    \begin{split}
        | f_i(x,s) | \leq \frac{C i \eps_i}{A_i \rho(x)^{\ell}(\rho(x)^2 + t_i + s)^{\frac{k - \ell}{2}} } 
        & \leq \frac{C}{d \rho(x)^{\ell} \left(1 + \frac{s}{t_i} \right)^{\frac{k - \ell}{2}} },
        \end{split}
    \end{equation}
    while the estimate (\ref{cutoffMestimate:almostorthog}) gives
    \begin{equation}
    \left( f_i, v_i \right)_{g_{1,i}} \to 0
    \end{equation}
    as $i \to \infty.$ We now pass to a subsequential limit $x_i \to x_\infty,$ $f_i \to f$ in $C^\infty_{loc} \left( M \times \left( -\infty, 0 \RB \right)$ as $i \to \infty.$ In view of (\ref{cutoffMestimate:hC2gammaellbound}), we have $g_i(t_i) \to g_0$ as $i \to \infty.$ We obtain an ancient solution $f(x,s)$ of (\ref{Mevolution}) with $g(t) \equiv g_0,$ which satisfies
    \begin{equation}
        | f(x,s) | \leq \frac{C}{\rho(x)^{\ell} }, \qquad f(x_\infty, 0) = 1, \qquad
    \left( f, v \right) = 0
    \end{equation}
    for all $v \in \ker L_{g_0}.$ But this contradicts the Liouville Theorem \ref{thm:liouville}.

    \vspace{2mm}

    \noindent   \underline{ {\bf Case 2:} There exists no bounded sequence of $x_i \in M$ such that (\ref{Mestimate:Mnonestimate}) holds with $x = x_i.$}

In particular, given any $R_1 > 0,$ for $i$ sufficiently large, we have
\begin{equation}
| \calM_i(x,t) |^q < (i \eps_i)^q G(x,t)
\end{equation}
for all $x \in \bar{B}_{2R_1}$ and $0 \leq t \leq t_i.$ But, choosing $R_1 \geq R_0$ (from Lemma \ref{lemma:comparisonprinciple}) such that the evolution equation (\ref{comparisonprinciple:eps0evolution}) is satisfied by $|\calM_i|,$ and taking $\underline{R}(t) \equiv R_1$ and $\overline{R}(t) = \sqrt{t},$ the comparison principle (Lemma \ref{lemma:comparisonprinciple}) gives
\begin{equation}
| \calM_i(x,t) |^q < (i \eps_i)^q G(x,t)
\end{equation}
for all $R_1 < r \leq \sqrt{t}$ and $R_1^2 < t \leq t_i.$ This contradicts (\ref{Mestimate:Mnonestimate}), completing the proof of (\ref{Mestimate:Mestimate}).

In view of the $C^0_{-\ell}$-bound (\ref{cutoffMestimate:hiC2-ellbound}), we can now use Lemma \ref{lemma:localC2gamma-deltaestimate} to deduce (\ref{cutoffMestimate:gestimate}), completing the proof.
\end{proof}

    \begin{thm}[$L^p \cap L^\infty$-stability, $p < n,$ \cite{DeruelleKroncke20}-\cite{KronckePetersen20}]\label{thm:Lqstability}
        Fix an $n$-dimensional Ricci-flat ALE space $M$ of dimension and order $n \geq 4,$ satisfying (1-3) of 
        Theorem \ref{main thm weighted Schauder spaces} with $\mu = n.$ Given $1 < p < n,$ let $1 < \ell \leq \frac{n}{p}$ with $\ell < n-2$ and also choose $\ell'$ with $\max\{ 0, 2- \ell \} < \ell' < \ell.$ For each $\eps > 0,$ there exists $\delta > 0$ as follows.
        
Define
\begin{equation}\label{v(x,t)def}
v(x,t) := \frac{1}{\rho(x)} + \frac{1}{\sqrt{1 + t}}.
\end{equation}
Given any initial metric $g(0)$ with
        \begin{equation}\label{Lpstability:deltainitialbound}
        \| g(0) - g_0 \|_{L^p\cap L^\infty(M)} \leq \delta,
    \end{equation}
    the corresponding solution $g(t) = g_0 + h(t)$ of Ricci-DeTurck flow exists for all time and converges in $L^p$ and $C^\infty$ to a gauged Ricci-flat metric $g_\infty.$ We have the estimates
      \begin{equation}\label{Lpstability:hLpestimate}
        \sup_{0 \leq t < \infty} \| h(t) \|_{L^p(M)} \leq \eps,
    \end{equation}
          \begin{equation}\label{Lpstability:hpointwiseestimate}
       | h(x,t) | \leq \eps v(x,t)^{\frac{n}{p}},
    \end{equation}
    and, for $1 \leq t < \infty,$
        \begin{equation}\label{Lpstability:Mestimate}
        |\calM(x,t)| \leq \frac{\eps v(x,t)^{\ell'}}{ t^{\frac{\ell - \ell'}{2} + 1} }
    \end{equation}
    and
\begin{equation}\label{Lpstability:ginftyestimate}
        | \nabla^{g_\infty,(k)} \left( g(t) - g_\infty \right) | \leq \frac{ C_{k,g_0} \eps v(x,t)^{k + \ell'} }{ t^{\frac{\ell - \ell'}{2}} }
    \end{equation}
    for each $k \in \N.$
    
    Moreover, there exists a family $\theta_t$ of diffeomorphisms as in the statement of Theorem \ref{main thm L^p spaces}.
    \end{thm}
\begin{proof} 
\underline{{\bf Case 1:} $p > \frac{n}{n-2}.$} 
Let $\ell < \frac{n}{2}.$ 
Assume that $q > 1$ is such that $\ell < \frac{n}{q^2} - \frac{2}{q} $ and $\left( 1 - \frac{q}{p} \right) n > 2.$
Since $\ell < \frac{n}{2},$ it is possible to choose an $r$ as in (\ref{2-relllessthan...}).

    \vspace{2mm}
    
\noindent \emph{Claim.} 
   For $\eps > 0$ sufficiently small, assuming (\ref{Lpstability:deltainitialbound}), the following is true. If, for $t_1 \geq 1,$ we have
    \begin{equation}\label{Lpstability:etahLpestimate}
        \sup_{0 \leq t < t_1} \| h(t) \|_{L^p(M)} \leq \eps,
    \end{equation}
    and
          \begin{equation}\label{Lpstability:etahpointwiseestimate}
       | h(x,t) | \leq \eps v(x,t)^{\ell}
    \end{equation}
    for $1 \leq t \leq t_1,$ then in fact
    \begin{equation}\label{Lpstability:halfetahLpestimate}
        \sup_{0 \leq t < t_1} \| h(t) \|_{L^p(M)} \leq \frac{\eps}{2}
    \end{equation}
    and
     \begin{equation}\label{Lpstability:halfetahpointwiseestimate}
       | h(x,t) | \leq \frac{\eps}{2} v(x,t)^{\ell}.
    \end{equation}
\begin{claimproof} To prove the claim, we first apply Proposition \ref{prop:Lpbootstrap} with the assumptions (\ref{Lpstability:deltainitialbound}), (\ref{Lpstability:etahLpestimate}), and (\ref{Lpstability:etahpointwiseestimate}). For $t \geq 2,$ this gives us
\begin{equation}\label{Lpstability:Linftybound}
\begin{split}
    |h|(x,t) & \leq C_1 \left( t^{-\frac{\ell}{2} } \delta + \eps \rho(x)^{-\ell - \alpha}  + \eps^{1 + \frac{1}{q} } v(x,t)^{\ell + \frac{1}{q}} \right),
    \end{split}
\end{equation}
and, for $R \geq 1,$ 
        \begin{equation}\label{Lpstability:Lpbound}
        \begin{split}
        \|h(t) \|_{L^p(M \setminus B_R)} & \leq \delta + C_2 \eps R^{-\alpha} + C_3 \eps^{1 + \frac{1}{q}}.
        \end{split}
    \end{equation}
    Here $\alpha > 0$ is a positive power whose precise value is not significant.
    
Now, given $0 < \kappa \leq \frac12 $ (to be determined), let 
\begin{equation}\label{Lpstability:Rdef}
R_1 = \left(\frac{C_1}{3 \eps_1 \kappa } \right)^{\frac{1}{\alpha}}. 
\end{equation}
Since the $L^p$ and $L^\infty$ norms grow only exponentially while they are small (by Deruelle-Kr{\"o}ncke \cite[Lemma 3.3]{DeruelleKroncke20} or an adaptation of Proposition \ref{prop:Lpbootstrap}), we may choose $\delta$ sufficiently small so that
    \begin{equation}\label{Lpstability:kappaetahLpestimate}
        \sup_{0 \leq t \leq R_1^2} \| h(t) \|_{L^p(M)} \leq \kappa \eps
    \end{equation}
    and, for $0 \leq t \leq R_1^2,$
          \begin{equation}\label{Lpstability:kappaetahpointwiseestimate}
       | h(x,t) | \leq \kappa \eps v(x,t)^{\ell}.
    \end{equation}
    In particular, (\ref{Lpstability:halfetahLpestimate}-\ref{Lpstability:halfetahpointwiseestimate}) are satisfied for $0 \leq t \leq R_1^2.$
    
    Now, observe that (\ref{Lpstability:Linftybound}), with $R_1$ as in (\ref{Lpstability:Rdef}) and $\delta, \eps$ sufficiently small, implies
    \begin{equation}\label{Lpstability:secondLinftybound}
\begin{split}
    |h|(x,t) & \leq \eps_1 \kappa \eps t^{-\frac{\ell}{2}}
    \end{split}
\end{equation}
    for $t \geq R_1^2$ and $\rho(x) \geq \sqrt{t}/2.$ 
    In particular, this gives the assumption (\ref{cutuffMestimate:C20bound}) for all $1 \leq t \leq t_1$ (with $\eps_1 \kappa \eps$ in place of $\eps$). Applying Proposition \ref{prop:cutoffMestimate}, we now have
            \begin{equation}\label{Lpstability:secondMestimate}
        |\calM(x,t)| \leq \frac{ \kappa \eps v(x,t)^{\ell'}}{t^{\frac{\ell - \ell'}{2} + 1}}
    \end{equation}
    and
     \begin{equation}\label{Lpstability:gestimate}
        | h(x,t) | \leq \kappa \eps v(x,t)^\ell
    \end{equation}
    for all $1 \leq t < t_1$ and $x \in B_{\sqrt{t}}.$ In particular, if $\kappa \leq \frac12,$ this establishes (\ref{Lpstability:halfetahpointwiseestimate}) for $x \in B_{\sqrt{t}}$ and $t \geq R_1^2.$ For $x \not\in B_{\sqrt{t}}$ and $t \geq R_1^2,$ (\ref{Lpstability:halfetahpointwiseestimate}) follows from (\ref{Lpstability:secondLinftybound}).

    It remains to prove (\ref{Lpstability:halfetahLpestimate}). We let
    $$R_2 = \left( \frac{\eps_1}{4C_2} \right)^{\frac{1}{\alpha}},$$
    where $\alpha$ is the exponent from (\ref{Lpstability:Lpbound}).
    Notice that $R_2$ does not depend on $\kappa.$ Integrating (\ref{Lpstability:gestimate}) over $B_{R_2},$ we have
    \begin{equation}
        \| h (t) \|_{L^p(B_{R_2})} \leq C \kappa \eps.
    \end{equation}
    Also, applying (\ref{Lpstability:Lpbound}) with $R = R_2,$ we get
    \begin{equation}
        \| h (t) \|_{M \setminus B_{R_2}} \leq \frac{\eps}{4}.
    \end{equation}
    Supposing that $\kappa < \frac{1}{4C},$ this establishes (\ref{Lpstability:halfetahLpestimate}) and with it the claim.
    \end{claimproof}

Now, by Shi \cite{Shi89}, the solution will exist as long as  (\ref{Lpstability:etahpointwiseestimate}) is satisfied. Taking $T$ to be the maximal time such that (\ref{Lpstability:etahLpestimate}-\ref{Lpstability:etahpointwiseestimate}) are satisfied, by the claim, we must have $T = \infty.$ In particular, (\ref{Lpstability:hLpestimate}) is satisfied. By (\ref{Lpstability:secondMestimate}-\ref{Lpstability:gestimate}), we see that (\ref{Lpstability:Mestimate}) is satisfied, as well as
          \begin{equation}\label{Lpstability:hpointwiseestimateellonly}
       | h(x,t) | \leq \eps v(x,t)^{\ell},
    \end{equation}
which is (\ref{Lpstability:hpointwiseestimate}) with $\ell$ in place of $\frac{n}{p}.$ (Below we will prove (\ref{Lpstability:hpointwiseestimate}) with exponent $\frac{n}{p}.$) We can obtain (\ref{Lpstability:ginftyestimate}) with $k = 0$ by integrating (\ref{Lpstability:Mestimate}) in time, and for $k > 0$ by applying interior parabolic regularity to $\calM$ and integrating in time.

Last, we show that the solution $g(t)$ converges to $g_\infty$ in $L^p(M).$ Let $\eta > 0.$ The $L^p$-estimate (\ref{Lpbootstrap:Lpbound}) of Proposition \ref{prop:Lpbootstrap} again reads
\begin{equation}\label{Lpstability:detailedLpestimate}
        \begin{split}
        \|h(t) \|_{L^p(M \setminus B_R)} & \leq \|h(0) \|_{L^p ( M \setminus B_R )} + \min \left\{ 1, \left( \frac{C R^2}{t} \right)^{\alpha } \right\} \|h(0)\|_{L^p(B_R)} + C \eps \left( \frac{1}{R^{\alpha} } + \frac{\eps^{\frac{1}{q}}}{t^\alpha} \right).
        \end{split}
    \end{equation}
For $R$ sufficiently large, we have
$\|h(0) \|_{L^p ( M \setminus B_R )} < \frac{\eta}{4}$
and also
 $\|g_\infty - g_0 \|_{L^p(M \setminus B_R)} < \frac{\eta}{4}.$ Choosing $R$ larger and $t$ sufficiently large, we can also bound
 $$\left( \frac{C R^2}{t} \right)^{\frac{n}{2}\left( \frac{1}{q} - \frac{1}{p} \right)}  \|h(0)\|_{L^p(B_R)} + C \eps \left( \frac{1}{R^{\alpha} } + \frac{\eps^{\frac{1}{q}}}{t^\alpha} \right) < \frac{\eta}{4},$$
 which from (\ref{Lpstability:detailedLpestimate}) gives
 $$\|h(t) \|_{L^p(M \setminus B_R)} < \frac{\eta}{2}.$$
Finally, in view of (\ref{Lpstability:etahpointwiseestimate}), for $t$ sufficiently large, we also have
$$\|g(t) - g_\infty \|_{L^p(B_R)} < \frac{\eta}{4}.$$
Combining the above, we get
\begin{equation}
\begin{split}
\|g(t) - g_\infty \|_{L^p(M)} & \leq \|g(t) - g_\infty \|_{L^p(B_R)} + \|h(t) \|_{L^p(M \setminus B_R)} + \|g_\infty - g_0 \|_{L^p(M \setminus B_R)} \\
& \leq \frac{\eta}{4} + \frac{\eta}{2} + \frac{\eta}{4} < \eta.
\end{split}
\end{equation}
Since $\eta > 0$ was arbitrary, we are done.

\vspace{2mm}

\noindent \underline{{\bf Case 2:} $1 < p \leq \frac{n}{n-2}.$ }
We first apply Case 1 with $p'$ greater than $\frac{n}{n-2},$ to obtain an $L^{p'}$-convergent solution. We can in fact take $\ell$ arbitrarily close to $\frac{n}{2} \leq n - 2,$ and $\ell'$ arbitrarily close to zero, so that $m = \ell - \ell'$ satisfies $2 < qm < n-2.$ The assumption (\ref{smallpLpbootstrap:Linftybound}) is satisfied, and we can re-run the previous argument with $\tilde{h} = g(t) - g_\infty$ in place of $h,$ using Proposition \ref{prop:smallpLpbootstrap}.

The pointwise bounds (\ref{Lpstability:hpointwiseestimate}) and (\ref{Lpstability:ginftyestimate}) follow by letting $m$ increase to $\frac{n}{p}$ in (\ref{smallpLpbootstrap:Linftybound}).

Finally, we can pass from Ricci-DeTurck flow back to Ricci flow by the same argument as above. 
\end{proof}

\vspace{10mm}

\appendix

\section{Weighted Schauder estimates} \label{section Schauder ests}

In this appendix, we collect Schauder estimates for weighted H{\"o}lder norms (Definition \ref{defn weighted Holder}) that are used throughout the paper.

The first is an interior Schauder estimate.
Its proof combines Brandt's non-standard interior Schauder estimates \cite{Brandt69} with a rescaling argument.

\begin{lemma} \label{lem parab smoothing in weighted Holder}
    Let $(M^n, g_0)$ be a Ricci-flat ALE manifold, $\delta \in \R$, and $\gamma \in (0,1)$.
    For any $r_0 > 0$ and $\Lambda \ge 1$, there exists $C = C(M^n, g_0, \gamma, \delta, r_0, \Lambda) > 0$ such that if $u, g$ are time-dependent families of symmetric 2-tensor fields on $M$ and $F,G$ are time-dependent families of tensor fields on $M$ that satisfy
    \begin{gather}
        \label{lem parab smoothing in weighted Holder, eqn 1}
        \partial_t u_{ij} = \Delta_{g, g_0} u_{ij} + F_{ij}^{cab} \nabla_c u_{ab} + G_{ij}^{ab} u_{ab} \qquad \text{on } B_{g_0}(x_0, 2r_0 \rho_0 ) \times (-4r_0^2\rho_0^2, 0), \\
        \label{lem parab smoothing in weighted Holder, eqn 2}
        \begin{split}
            \sup_{t \in (-4r_0^2\rho_0^2 , 0)} \|  g ( t) - g_0\|_{C^{0, \gamma}_{0}( B_{g_0}( x_0, 2 r_0 \rho_0), g_0 ) }+ \sup_{t \in (-4r_0^2\rho_0^2 , 0)} \|  F ( t)\|_{C^{0, \gamma}_{-1}( B_{g_0}( x_0, 2 r_0 \rho_0) , g_0) } \\+ \sup_{t \in (-4r_0^2\rho_0^2, 0)} \|  G( t) \|_{C^{0, \gamma}_{-2}( B_{g_0}( x_0, 2 r_0 \rho_0) , g_0 ) } 
            \le \Lambda,
        \end{split} \\
        \label{lem parab smoothing in weighted Holder, eqn 3^}
        \text{and } \Lambda^{-1} g_0 \le g \le \Lambda g_0 \qquad \text{on } B_{g_0}(x_0, 2r_0 \rho_0 ) \times (-4r_0^2\rho_0^2, 0) ,
    \end{gather}
    for some $x_0 \in M$ and $\rho_0 := \rho(x_0)$,
    then 
    \begin{equation}
         \sup_{t \in(-r_0^2 \rho_0^2 , 0)}  \| u ( t) \|_{C^{2, \gamma}_{-\delta}(B_{g_0}(x_0, r_0 \rho_0) , g_0  )} 
        \le C  \sup_{t \in (-4r_0^2 \rho_0^2 , 0) } \| u ( t) \|_{C^0_{-\delta}(B_{g_0}(x_0, 2r_0 \rho_0)  , g_0)} .
    \end{equation}
\end{lemma}
\begin{proof}
    We first show that the statement of the lemma holds for small scales $r_0$ and centers $x_0$ outside a large compact set.
    \begin{claim} \label{proof parab smoothing in weighted Holder, claim 1}
        There exists a compact subset $K' \subset M$ (depending only on $M^n, g_0$) such that,
        when $0 < r_0 \le \frac1{100}$, there exists $C = C(M^n, g_0, \gamma, r_0, \delta, \Lambda) > 0$ such that, for any $x_0 \in M \setminus K'$ and $\rho_0 = \rho(x_0)$,
        if 
        \begin{gather*}
            \partial_t u_{ij} = \Delta_{g, g_0} u_{ij} + F_{ij}^{cab} \nabla_c u_{ab} + G_{ij}^{ab} u_{ab} \qquad \text{on } B_{g_0}(x_0, 2r_0 \rho_0 ) \times (-4r_0^2\rho_0^2, 0), \\
            \begin{split}
                \sup_{t \in (-4r_0^2\rho_0^2 , 0)} \|  g ( t) - g_0\|_{C^{0, \gamma}_{0}( B_{g_0}( x_0, 2 r_0 \rho_0), g_0 ) }+ \sup_{t \in (-4r_0^2\rho_0^2 , 0)} \|  F ( t)\|_{C^{0, \gamma}_{-1}( B_{g_0}( x_0, 2 r_0 \rho_0) , g_0) } \\+ \sup_{t \in (-4r_0^2\rho_0^2, 0)} \|  G( t) \|_{C^{0, \gamma}_{-2}( B_{g_0}( x_0, 2 r_0 \rho_0) , g_0 ) } 
                \le \Lambda,
            \end{split} \\
            \text{and } \Lambda^{-1} g_0 \le g \le \Lambda g_0 \qquad \text{on } B_{g_0}(x_0, 2r_0 \rho_0 ) \times (-4r_0^2\rho_0^2, 0) ,
        \end{gather*}
        then
        \begin{equation*}
            \sup_{t \in  (-r_0^2 \rho_0^2 /4, 0) } \| u( t) \|_{C^{2, \gamma}_{-\delta}( B_{g_0}(x_0, r_0 \rho_0/2)) } 
            \le C  \sup_{t \in  (-4r_0^2 \rho_0^2, 0) } \| u ( t) \|_{C^0_{-\delta}( B_{g_0}(x_0, 2 r_0 \rho_0) ) }  .
        \end{equation*}
    \end{claim}
    \begin{claimproof}
        Since $(M, g_0)$ is ALE, there is a coordinate system at infinity
            $$\Psi : \mathcal C_R = (R, \infty) \times (\mathbb{S}^{n-1}/ \Gamma) \to M \setminus K $$
        for some $K \subset M$ compact and $R > 0$ 
        such that
        \begin{equation} \label{proof parab smoothing in weighted Holder, eqn 2}
            \lambda^{-2} \phi_{\lambda}^* \Psi^* g_0 \xrightarrow[\lambda \to \infty]{C^\infty_{loc}(\mathcal C, g_{Euc})} g_{Euc}
        \end{equation}
        where $\mathcal C = \mathcal C_0 = (0, \infty) \times ( \mathbb S^{n-1} / \Gamma)$ and $\phi_\lambda : \mathcal C_{R / \lambda} \to \mathcal C_R$ is the dilation map $\phi_\lambda (r, \theta) = (\lambda r, \theta)$.
        Denote $\psi_\lambda = \Psi \circ \phi_\lambda$.
    
        By enlarging $K \subset M$ and $R > 0$ and using that $(M, g_0)$ is ALE, we can assume without loss of generality that $\rho \circ \Psi$ is comparable to the coordinate $r \in (0, \infty)$ on $\mathcal C$, namely,
        \begin{equation} \label{proof parab smoothing in weighted Holder, eqn 3.0}
            1 \le \frac12 r \le \rho \circ \Psi \le 2 r \qquad \forall (r, \theta) \in \mathcal C_R.
        \end{equation}
        Observe then that $r ( \psi_{\rho(x)}^{-1}(x) \in [\frac 12, 2]$ for all $x \in M \setminus K$.
        Additionally, there exists a larger compact subset $K'$ with $K \subset K' \subset M$ such that for all $x \in M \setminus K'$ we have
        \begin{gather}
            \label{proof parab smoothing in weighted Holder, eqn 3.15}
            B_{Euc}\left( \Psi^{-1} (x) , \frac14 \rho(x) \right) \subset B_{Euc} \left( \Psi^{-1} (x) , \frac 12 r( \Psi^{-1} (x) ) \right) \subset \mathcal C_R  \text{ and}\\
             \label{proof parab smoothing in weighted Holder, eqn 3.2}
            \frac12 g_{Euc} \le \rho(x)^{-2} \psi_{\rho(x)}^* g_0 \le 2 g_{Euc} \qquad \text{on } B_{Euc}\left(\psi_{\rho(x)}^{-1} (x) , \frac14 \right).
        \end{gather}
        
        Let $0 < r_0 \le \frac1{100}$ and $x_0 \in M \setminus K'$.
        Set $\rho_0 :=\rho(x_0) $.
        Consider the tensors
        \begin{gather*}
            \hat g_0 := \rho_0^{-2} \psi_{\rho_0}^*  g_0 , \qquad  
            \hat g(\tau) := \rho_0^{-2} \psi_{\rho_0}^*  g(\rho_0^2 \tau) , \qquad  
            \hat u(\tau) := \rho_0^{-2} \psi_{\rho_0}^* u ( \rho_0^2 \tau), \\ 
            \hat F(\tau) := \rho_0^2 \psi_{\rho_0}^* F(\rho_0^2 \tau) , \qquad 
            \hat G (\tau):= \rho_0^2 \psi_{\rho_0}^* G (\rho_0^2 \tau).
        \end{gather*}
            
        If
        $$\partial_t u_{ij} = \Delta_{g, g_0} u_{ij} + F_{ij}^{cab} \nabla_c u_{ab} + G_{ij}^{ab} u_{ab} \qquad \text{on } B_{g_0}(x_0, 2 r_0 \rho_0 ) \times (-4r_0^2\rho_0^2, 0),$$
        then it follows that
        \begin{multline}
            \partial_\tau \hat u 
            = \rho_0^{-2} \psi_{\rho_0}^* \partial_t u \cdot \rho_0^2 
            = \psi_{\rho_0}^* \left( \Delta_{g, g_0} u + F * \nabla u + G* u  \right) 
            = \Delta_{\hat g, \hat g_0} \hat u + \hat F_{ij}^{cab} \hat \nabla_c\hat u_{ab} + \hat G_{ij}^{ab} \hat u_{ab} \\
            \text{on } B_{\hat g_0} (\psi_{\rho_0}^{-1}(x_0), 2r_0) \times (-4r_0^2, 0),
        \end{multline} 
        where $\hat \nabla = \nabla_{\hat g_0}$.
        Observe \eqref{proof parab smoothing in weighted Holder, eqn 3.2} implies
        \begin{equation} \label{proof parab smoothing in weighted Holder, eqn 3.4}
            B_{Euc}\left(\psi_{\rho_0}^{-1}(x_0), \sqrt 2 r_0 \right) \subset B_{\hat g_0} (\psi_{\rho_0}^{-1}(x_0), 2r_0)  \subset B_{Euc}\left(\psi_{\rho_0}^{-1}(x_0), 2\sqrt 2 r_0 \right) .  
        \end{equation}

        Recall $\frac12 \le r ( \psi_{\rho_0}^{-1} (x_0) ) \le 2$, and hence
        \begin{equation} \label{proof parab smoothing in weighted Holder, eqn 3.5}
            B_{Euc} ( \psi_{\rho_0}^{-1} (x_0) , \sqrt 2 r_0) \subset \left( \frac12 - \frac{\sqrt 2}{100}, 2 + \frac{\sqrt 2}{100} \right) \times \mathbb S^{n-1} / \Gamma.
        \end{equation}
        By \eqref{proof parab smoothing in weighted Holder, eqn 2}, $\hat g_0 = \rho_0^{-2} \psi_{\rho_0}^* g_0$ therefore remains in a fixed $C^3(B_{Euc} ( \psi_{\rho_0}^{-1} (x_0) , \sqrt 2 r_0), g_{Euc})$-neighborhood of $g_{Euc}$ as $\rho_0 \to \infty$.
        Next, note that
        \begin{equation} \label{proof parab smoothing in weighted Holder, eqn 3.6}
            | \hat G |_{\hat g_0 } (y, \tau) = \sqrt{ (\hat g_0)_{a_1 a_2} (\hat g_0)_{b_1 b_2} (\hat g_0)^{i_1 i_2} (\hat g_0)^{j_1 j_2} \hat G_{i_1 j_1}^{a_1 b_1} \hat G_{i_2 j_2}^{a_2 b_2}  } = \rho_0^2 |G|_{g_0} ( \psi_{\rho_0}(y), \rho_0^2 \tau)  
        \end{equation}
        and similarly $|\hat F |_{\hat g_0}(y, \tau) = \rho_0 |F |_{g_0} ( \psi_{\rho_0}(y), \rho_0^2 \tau)$ and $|\hat g^{-1} |_{\hat g_0}(y, \tau) = |g^{-1}|_{g_0} ( \psi_{\rho_0}(y), \rho_0^2 \tau)$.
        Additionally, equations \eqref{proof parab smoothing in weighted Holder, eqn 3.0}--\eqref{proof parab smoothing in weighted Holder, eqn 3.2} imply $\rho$ is comparable to $\rho(x_0)$ on $ B_{g_0}(x_0, 2 r_0 \rho_0)$, that is,
        \begin{equation} \label{proof parab smoothing in weighted Holder, eqn 3.65}
            \frac12 \left( \frac12 - \frac{2 \sqrt 2}{100} \right) \rho(x_0) \le \rho(x) \le 2 \left( 2 + \frac{2 \sqrt 2}{100} \right) \rho(x_0) \qquad 
            \forall x \in B_{g_0}(x_0, 2 r_0 \rho_0).
        \end{equation}
        Combining \eqref{proof parab smoothing in weighted Holder, eqn 3.4}--\eqref{proof parab smoothing in weighted Holder, eqn 3.65},
        it follows that
        \begin{align*}
            & \sup_{\tau \in (-4r_0^2, 0)} \| \hat g - \hat g_0 \|_{C^{0, \gamma}_*( B_{Euc}( \psi_{\rho_0}^{-1}(x_0), \sqrt 2 r_0) ,  g_{Euc}) } + \sup_{\tau \in (-4r_0^2, 0)} \| \hat F \|_{C^{0, \gamma}_*( B_{Euc}( \psi_{\rho_0}^{-1}(x_0), \sqrt 2 r_0) ,  g_{Euc}) } \\
            &+ \sup_{\tau \in (-4r_0^2, 0)} \| \hat G \|_{C^{0, \gamma}_*( B_{Euc}( \psi_{\rho_0}^{-1}(x_0), \sqrt 2 r_0) ,  g_{Euc}) }
            \\
            \le C & \left(  \sup_{\tau \in (-4r_0^2, 0)} \| \hat g - \hat g_0 \|_{C^{0, \gamma}_*( B_{\hat g_0}( \psi_{\rho_0}^{-1}(x_0),  2 r_0) ,  \hat g_0) } + \sup_{\tau \in (-4r_0^2, 0)} \| \hat F \|_{C^{0, \gamma}_*( B_{\hat g_0}( \psi_{\rho_0}^{-1}(x_0),  2 r_0) , \hat g_0) } \right.\\
            & \left. + \sup_{\tau \in (-4r_0^2, 0)} \| \hat G \|_{C^{0, \gamma}_*( B_{\hat g_0}( \psi_{\rho_0}^{-1}(x_0),  2 r_0)  , \hat g_0) } \right)
            \\
            \le  C &\left(  \sup_{\tau \in (-4r_0^2, 0)} \|  g -  g_0 \|_{C^{0, \gamma}_0( B_{g_0}( \psi_{\rho_0}^{-1}(x_0), 2 r_0 \rho_0) ,  g_0) } + \sup_{\tau \in (-4r_0^2\rho_0^2 , 0)} \|  F \|_{C^{0, \gamma}_{-1}( B_{g_0}( x_0, 2 r_0 \rho_0) ,  g_0) } \right. \\
            &\left.+ \sup_{\tau \in (-4r_0^2\rho_0^2, 0)} \|  G \|_{C^{0, \gamma}_{-2}( B_{g_0}( x_0, 2 r_0 \rho_0)  ,  g_0) } \right) \\
            \le C & \Lambda,
        \end{align*}
        for some constant $C = C(M^n, g_0, \gamma, r_0) >0 $ which may vary from line to line but is always independent of $x_0 \in M \setminus K'$.
        Using also $C^3_{loc}$-closeness of $\hat g_0$ to $g_{Euc}$, it follows that, in Euclidean coordinates, the coefficients of the linear operator
            $$\hat u \mapsto  \Delta_{\hat g, \hat g_0} \hat u_{ij}+ \hat F^{cab}_{ij} \hat \nabla_c \hat u_{ab} + \hat G^{ab}_{ij} \hat u_{ab}$$
        have $\sup_\tau \| \cdot \|_{C^{0, \gamma}_*}$ bounds on $B_{Euc}(\psi_{\rho_0}^{-1} (x_0) , \sqrt 2 r_0) \times (-4r_0^2, 0)$ that are independent of $x_0 \in M \setminus K'$.
        Moreover, assumption \eqref{lem parab smoothing in weighted Holder, eqn 3^} implies the equation is uniformly elliptic on the region $B_{Euc}(\psi_{\rho_0}^{-1} (x_0) , \sqrt 2 r_0) \times (-4r_0^2, 0)$ with an ellipticity constant independent of $x_0 \in M \setminus K'$.
        Brandt's non-standard interior Schauder estimates \cite{Brandt69} (see also  \cite[Theorem 3.6]{ChodoshSchulze21} and references therein)
        therefore imply
        \begin{equation}  \label{proof parab smoothing in weighted Holder, eqn 4}
            \sup_{\tau \in (-r_0^2, 0)} \| \hat u \|_{C^{2, \gamma}_{ *} ( B_{Euc} ( \psi_{\rho_0}^{-1}(x_0) , r_0) , g_{Euc}) } 
            \le C \sup_{\tau \in (-4r_0^2, 0)}  \| \hat u \|_{C^0_{ *} (B_{Euc} ( \psi_{\rho_0}^{-1}(x_0) , \sqrt 2 r_0) , g_{Euc})} .
        \end{equation}
        where $C = C(M^n, g_0, \gamma, r_0, \Lambda) > 0$.
        Using \eqref{proof parab smoothing in weighted Holder, eqn 2} and \eqref{proof parab smoothing in weighted Holder, eqn 3.2}, we deduce
        \begin{equation} \label{proof parab smoothing in weighted Holder, eqn 5}
            \sup_{\tau \in \left( -\frac14 r_0^2, 0 \right) } \| \hat u \|_{C^{2, \gamma}_{ *} ( B_{\hat g_0} ( \psi_{\rho_0}^{-1}(x_0) , \frac1{ 2}r_0), \hat g_0) } \\
            \le C \sup_{\tau \in (-4r_0^2, 0)} \| \hat u \|_{C^0_{ *} (B_{\hat g_0} ( \psi_{\rho_0}^{-1}(x_0) , 2 r_0), \hat g_0)} 
        \end{equation}
        where $C = C(M^n, g_0, \gamma, r_0, \Lambda) > 0$ depends on $M^n, g_0, \gamma, r_0, \Lambda$ but is independent of $x_0 \in M \setminus K'$.
        By similar logic as in \eqref{proof parab smoothing in weighted Holder, eqn 3.6}, we have 
        \begin{equation*}
            | \hat \nabla \hat \nabla \hat u |_{\hat g_0}(y, \tau) 
            = \sqrt{ (\hat g_0)^{i_1 i_2} (\hat g_0)^{j_1 j_2} (\hat g_0)^{k_1 k_2} (\hat g_0)^{l_1 l_2}    \hat \nabla_{i_1} \hat \nabla_{j_1} \hat u_{k_1 l_1}\hat \nabla_{i_2} \hat \nabla_{j_2} \hat u_{k_2 l_2}} 
            = \rho_0^2 | \nabla \nabla u |_{g_0} ( \psi(y), \rho_0^2 \tau) ,
        \end{equation*}
        $|\hat \nabla \hat u |_{\hat g_0 } = \rho_0 |\nabla u |_{g_0}$, and $|\hat u |_{\hat g_0} = |u |_{g_0}$.
        After undoing these rescalings in \eqref{proof parab smoothing in weighted Holder, eqn 5}, multiplying by $\rho_0^\delta$, and using that $\rho$ is comparable to $\rho_0$ on $B_{g_0} ( x_0, 2 r_0 \rho_0)$ \eqref{proof parab smoothing in weighted Holder, eqn 3.65}, the estimate \eqref{proof parab smoothing in weighted Holder, eqn 5} then becomes
        \begin{equation*}
            \sup_{t \in (-r_0^2 \rho_0^2 /4, 0)} \| u \|_{C^{2, \gamma}_{-\delta}( B_{g_0}(x_0, r_0 \rho_0/2), g_0) } \\
            \le C \sup_{t \in (-4r_0^2 \rho_0^2, 0)} \| u \|_{C^0_{-\delta}( B_{g_0}(x_0, 2 r_0 \rho_0) , g_0) }  .
        \end{equation*}
        where $C = C(M^n, g_0, \gamma, r_0, \delta, \Lambda) > 0$.
        This completes the proof of the claim.
    \end{claimproof}

    On the other hand,  $1 \le \inf_{x \in K'} \rho(x) \le \sup_{x \in K'} \rho(x) < \infty$ since $K'$ is compact.
    Thus, Brandt's interior Schauder estimates \cite{Brandt69} directly apply to give 
    \begin{equation} \label{proof parab smoothing in weighted Holder, eqn 8}
        \sup_{t \in (-r_0^2 \rho_0^2 /4, 0)} \| u( t) \|_{C^{2, \gamma}_{-\delta}( B_{g_0}(x_0, r_0 \rho_0/2) , g_0) } 
        \le C \sup_{t \in (-4r_0^2 \rho_0^2, 0)} \| u(t) \|_{C^0_{-\delta}( B_{g_0}(x_0, 2 r_0 \rho_0) , g_0 ) } 
    \end{equation}
    for some constant $C = C(M^n, g_0, \gamma, \delta, r_0, \Lambda)>0$, whenever $0 < r_0 \le \frac1{100}$, $x_0 \in K'$, $\rho_0 = \rho(x_0)$, and 
    \begin{gather*}
        \partial_t u_{ij} = \Delta_{g, g_0} u_{ij} + F_{ij}^{cab} \nabla_c u_{ab} + G_{ij}^{ab} u_{ab} \qquad \text{on } B_{g_0}(x_0, 2r_0 \rho_0 ) \times (-4r_0^2\rho_0^2, 0), \\
        \begin{split}
            \sup_{t \in (-4r_0^2\rho_0^2 , 0)} \|  g ( t) - g_0\|_{C^{0, \gamma}_{0}( B_{g_0}( x_0, 2 r_0 \rho_0), g_0 ) }+ \sup_{t \in (-4r_0^2\rho_0^2 , 0)} \|  F ( t)\|_{C^{0, \gamma}_{-1}( B_{g_0}( x_0, 2 r_0 \rho_0) , g_0) } \\+ \sup_{t \in (-4r_0^2\rho_0^2, 0)} \|  G( t) \|_{C^{0, \gamma}_{-2}( B_{g_0}( x_0, 2 r_0 \rho_0) , g_0 ) } 
            \le \Lambda,
        \end{split} \\
        \label{lem parab smoothing in weighted Holder, eqn 3}
        \text{and } \Lambda^{-1} g_0 \le g \le \Lambda g_0 \qquad \text{on } B_{g_0}(x_0, 2r_0 \rho_0 ) \times (-4r_0^2\rho_0^2, 0) ,
    \end{gather*}
    A covering argument using \eqref{proof parab smoothing in weighted Holder, eqn 8} and Claim \ref{proof parab smoothing in weighted Holder, claim 1} now proves the statement of the lemma for general $r_0 > 0$ and $x_0 \in M$.
\end{proof}



We shall make use of the following weighted H{\"o}lder estimate for Ricci-DeTurck flows.
The proof uses a rescaling argument similar to the proof of Lemma \ref{lem parab smoothing in weighted Holder}, but with different interior estimates on the model Euclidean space.

\begin{lemma}[Local $C^{2, \gamma}_{-\delta}$ Estimate]\label{lemma:localC2gamma-deltaestimate}
    Let $(M^n, g_0)$ be a Ricci-flat ALE manifold.
    For any $\delta \geq 0$, $\gamma \in (0,1)$, and $r_0 > 0$,
    there exists small $0 < \eps_0$ and large $C>1$ 
    both depending only on $M^n, g_0, \gamma, \delta, r_0$ such that following implications hold for any $x_0 \in M$ and $\rho_0 := \rho(x_0)$:
    \begin{enumerate}
        \item \label{lem local C2gamma-delta est, item global in time}
        If $g(t)$ is a smooth Ricci-DeTurck flow (with respect to $g_0$) on $B_{g_0}(x_0, 2r_0 \rho_0) \times [0 , T]$, then 
        \begin{multline*}
            \| g(0) - g_0 \|_{C^{2, \gamma}_{-\delta}(B_{g_0}(x_0, 2 r_0 \rho_0), g_0) } + \sup_{t \in [0, T]} \| g(t) - g_0 \|_{C^0_{-\delta}(B_{g_0}(x_0, 2 r_0 \rho_0), g_0)} \le \eps \leq \eps_0 \\
            \implies 
            \sup_{t \in [0,T]} \| g(t) - g_0 \|_{C^{2, \gamma}_{-\delta}(B_{g_0}(x_0,  r_0 \rho_0), g_0)}
            \le C \eps. 
        \end{multline*}
        
        \item \label{lem local C2gamma-delta est, item local in time}
        If $g(t)$ is a smooth Ricci-DeTurck flow (with respect to $g_0$) on $B_{g_0}(x_0, 2r_0 \rho_0) \times [0 , T]$ and $T \ge 2 r_0^2 \rho_0^2$, then 
        \begin{multline*}
            \sup_{t \in [0, T]} \| g(t) - g_0 \|_{C^0_{-\delta}(B_{g_0}(x_0, 2 r_0 \rho_0), g_0)} \le \eps \leq \eps_0 \\
            \implies 
            \sup_{t \in [r_0^2 \rho_0^2,T]} \| g(t) - g_0 \|_{C^{2, \gamma}_{-\delta}(B_{g_0}(x_0,  r_0 \rho_0), g_0)}
            \le C \eps. 
        \end{multline*}
    \end{enumerate}
\end{lemma}
\begin{proof}
    We first prove that statement \eqref{lem local C2gamma-delta est, item global in time} of the lemma holds for $x_0$ outside a large compact set and for suitably small scales $r_0 > 0$.

    \begin{claim} \label{proof local higher order est, claim 1}
        There exists a compact subset $K' \subset M$ (depending only on $M^n, g_0)$ such that, when $0 < r_0 < \frac1{100}$, there exists $\epsilon, C > 0$ (both depending only on $M^n, g_0, \gamma, \delta, r_0 $) such that, 
        for all $x_0 \in M \setminus K'$ and $\rho_0 = \rho(x_0)$, if
        $g(t)$ is a smooth Ricci-DeTurck flow (with respect to $g_0$) on $B_{g_0}(x_0, 2 r_0 \rho_0) \times [0, T]$, and 
            $$\| g(0) - g_0 \|_{C^{2, \gamma}_{-\delta}(B_{g_0}(x_0, 2 r_0 \rho_0), g_0) } + \sup_{t \in [0, T]} \| g(t) - g_0 \|_{C^0_{-\delta}(B_{g_0}(x_0, 2 r_0 \rho_0), g_0)} \le \eps \leq \eps_0 ,$$
        then
            $$\sup_{t \in [0,T]} \| g(t) - g_0 \|_{C^{2, \gamma}_{-\delta} \left(B_{g_0}\left(x_0, \frac1{\sqrt 2}  r_0 \rho_0 \right), g_0 \right)}
        \le C \eps.$$ 
    \end{claim}
    \begin{claimproof}
        Because $(M, g_0)$ is ALE, there exists $K \subset M$ compact, $R > 0$, and a diffeomorphism
            $$\Psi :  \mathcal C_R = (R, \infty) \times (\mathbb{S}^{n-1}/ \Gamma) \to M \setminus K $$
        such that
        \begin{equation} \label{proof local higher order est, eqn 2}
            \lambda^{-2} \phi_{\lambda}^* \Psi^* g_0 \xrightarrow[\lambda \to \infty]{C^\infty_{loc}(\mathcal C, g_{Euc})} g_{Euc}
        \end{equation}
        where $\mathcal C  = (0, \infty) \times ( \mathbb S^{n-1} / \Gamma)$ and $\phi_\lambda : \mathcal C_{R / \lambda} \to \mathcal C_R$ is the dilation map $\phi_\lambda (r, \theta) = (\lambda r, \theta)$.
        Denote also $\psi_\lambda = \Psi \circ \phi_\lambda$.
    
        By enlarging $R > 0$ and $K \subset M$ and using that $(M, g_0)$ is ALE, we can further assume without loss of generality that $\rho \circ \Psi$ and the coordinate $r \in (0, \infty)$ on $\mathcal C_R $ are comparable,
        namely,
        \begin{gather}
            \label{proof local higher order est, eqn 2.1}
            1 \le \frac12 r \le \rho \circ \Psi \le 2 r \qquad \forall (r, \theta) \in \mathcal C_R. 
        \end{gather}
        In particular, $r( \psi_{\rho(x)}^{-1} (x) ) \in \left[ \frac12 , 2 \right]$ for all $x \in M \setminus K$.
        Additionally, there exists a larger compact subset $K'$ with $K \subset K' \subset M$ such that for all $x \in M \setminus K'$ we have
        \begin{gather}
            \label{proof local higher order est, eqn 4.15}
            B_{Euc} \left( \Psi^{-1} (x) , \frac14 \rho( x)  \right) \subset B_{Euc} \left( \Psi^{-1} (x)  , \frac12 r( \Psi^{-1}(x)) \right) \subset \mathcal C_R \text{ and}  \\
             \label{proof higher order est, eqn 4.2}
            \frac12 g_{Euc} \le \rho(x)^{-2} \psi_{\rho(x)}^* g_0 \le 2 g_{Euc} \qquad \text{on }  B_{Euc} \left( \psi_{\rho(x)}^{-1} (x) , \frac14  \right) .
        \end{gather}

        Next, let $(x_0, t_0) \in M \times [0, T]$, set $\rho_0 := \rho(x_0)$, and
        let $\tilde g(t) = g(t)$ denote the Ricci-DeTurck flow (with reference metric $g_0$) on $B_{g_0}(x_0, 2 r_0 \rho_0) \times [0,T]$ as in the statement of the claim.
        Then $h = \tilde  g(t) - g_0$ satisfies $\partial_t h = \calM(g_0 + h)$ on $B_{g_0}(x_0, 2 r_0 \rho_0)$, where $\calM$ is given as in \eqref{calMfullexpression}. 
        We can rewrite this evolution equation as
        \begin{equation} \label{proof local higher order est, eqn 1}
            \partial_t h = \tilde g^{ab} \nabla_a \nabla_b h + \tilde g^{-1} \star \tilde g \star g_0^{-1} \star Rm \star h + \tilde g^{-1} \star \tilde g^{-1} \star \nabla h \star \nabla h
            \quad \text{on } B_{g_0}(x_0, 2r_0 \rho_0) \times [0, T],
        \end{equation}
        where $\star$ denotes a linear combination of endomorphism traces \emph{without} raising or lowering indices, 
        $Rm= Rm_{g_0}$ is a (3,1)-tensor, and $\nabla = \nabla_{g_0}$.
        
        Consider the families of symmetric 2-tensor fields
            $$\hat g := \rho_0^{-2} \psi_{\rho_0}^*  g_0 , \qquad 
            G(s) := \rho_0^{-2} \psi_{\rho_0}^* \tilde g(t_0 + \rho_0^2 s) , \qquad 
            H (s) := \rho_0^{-2} \psi_{\rho_0}^*  h ( t_0 + \rho_0^2 s) = G(s) - \hat g.$$
        From \eqref{proof local higher order est, eqn 1}, it follows that
        \begin{gather} \label{proof local higher order est, eqn 3} \begin{aligned}
            \partial_s H 
            ={}& \rho_0^{-2} \psi_{\rho_0}^* (\partial_t h ) \rho_0^2 \\
            ={}& \psi_{\rho_0}^* \left\{ \tilde g^{ab} \nabla_a \nabla_b h + \tilde g^{-1} \star \tilde g \star g_0^{-1} \star Rm \star h + \tilde g^{-1} \star \tilde g^{-1} \star \nabla h \star \nabla h \right\} \\
            ={}& G^{ab} \hat \nabla_a \hat \nabla_b H + G^{-1} \star G \star \hat g \star \hat {Rm} \star H + G^{-1} \star G^{-1} \star \hat \nabla H \star \hat \nabla H \\
            ={}& \hat g^{ab} \hat \nabla_a \hat \nabla_b H + \hat g^{-1} \star \hat g \star \hat g \star \hat {Rm} \star H + G^{-1} \star G^{-1} \star \hat \nabla H \star \hat \nabla H \\
            &+ ( G^{-1} - \hat g^{-1} ) \star \hat g \star \hat g \star \hat {Rm} \star H
            + G^{-1} \star H \star \hat g \star \hat {Rm} \star H
            + ( G^{-1} - \hat g^{-1} )\star  \hat \nabla \hat \nabla H\\
        \end{aligned} 
        \end{gather}
        on the region $B_{\hat g} (\psi_{\rho_0}^{-1} (x_0), 2 r_0) \times \left[ \rho_0^{-2}(-t_0) , \rho_0^{-2} ( T - t_0) \right]$.

        We now restrict to the case that $x_0 \in M \setminus K'$ and $0 < r_0 < \frac1{100}$.
        Observe, by \eqref{proof higher order est, eqn 4.2}, we have
            $$B_{Euc} \left(\psi_{\rho_0}^{-1} (x_0) , \sqrt 2 r_0 \right) \subset 
            B_{\hat g } ( \psi_{\rho_0}^{-1} (x_0) , 2 r_0) \subset B_{Euc} ( \psi_{\rho_0}^{-1} (x_0), 2 \sqrt 2 r_0) .$$
        It follows that for all $s \in \left[ \frac{-t_0}{\rho_0^2} , \frac{T- t_0}{\rho_0^2} \right]$,
        \begin{align*}
             \sup_{B_{Euc}( \psi^{-1}_{\rho_0}(x_0) , \sqrt 2 r_0)} | G(s) - \hat g |_{g_{Euc}} 
             \le{}& C \sup_{ B_{\hat g} (\psi_{\rho_0}^{-1} (x_0) , 2r_0) }  | G(s) - \hat g |_{\hat g } \\
             \le{}& C\sup_{B_{g_0}( x_0, 2  r_0 \rho_0) } | \tilde g(t_0 + \rho_0^2 s) - g_0 |_{g_0} \\
             \le{}& C    \| \tilde g(t_0 + \rho_0^2 s) - g_0 \|_{C^0_{-\delta}( B_{g_0}(x_0, 2 r_0 \rho_0), g_0) }
             && (\text{since $\delta \ge 0$})
        \end{align*}
        where $C = C(n) > 0$ is a dimensional constant that may change from line to line.
        Taking a supremum in $s$ gives
        \begin{equation} \label{proof local higher order est, eqn 4}
            \sup_{s \in \left[ \frac{-t_0}{\rho_0^2} , \frac{T- t_0}{\rho_0^2} \right]}  \| H(s) \|_{C^0 ( B_{Euc}( \psi_{\rho_0}^{-1} (x_0), \sqrt 2 r_0) , g_{Euc}) } 
            \le C(n)   \sup_{t \in [0, T]} \| h(t) \|_{C^0_{-\delta}( B_{g_0}(x_0, 2 r_0 \rho_0), g_0) }.
        \end{equation}
        
        From \eqref{proof local higher order est, eqn 2.1}, $\frac12 \le r ( \psi_{\rho_0}^{-1}(x_0)) \le 2$, and so 
            $$B_{Euc}( \psi_{\rho_0}^{-1} (x_0) , \sqrt 2 r_0) \subset \left( \frac12 - \frac{\sqrt 2}{100}, 2 + \frac{\sqrt 2}{100}\right) \times \mathbb S^{n-1} / \Gamma .$$ 
        From \eqref{proof local higher order est, eqn 2}, $\hat g = \rho_0^{-2} \psi_{\rho_0}^* g_0$ remains in a fixed $C^3_{loc}$-neighborhood of $g_{Euc}$ as $\rho_0 \to \infty$.
        Therefore, the coefficients of the linear operator
            $$L = \hat g^{ab} \hat \nabla_a \hat \nabla_b + \hat g^{-1} \star \hat g \star \hat g \star \hat {Rm} \star  $$
        have $C^1(g_{Euc})$-bounds on $B_{Euc}( \psi_{\rho_0}^{-1} (x_0) , \sqrt 2 r_0) $ that are independent of $x_0 \in M \setminus K'$.
        Interior estimates \cite[Propositions C.3, C.4]{Stolarski25} now apply to \eqref{proof local higher order est, eqn 3} to give the following:
        there exist $\epsilon \in (0,1) $ and $ C > 1$ (both depending on $M^n , g_0, \gamma, r_0$ but independent of $\rho_0, x_0, t_0$) such that 
        for 
        $\Omega := B_{\hat g}( \psi_{\rho_0}^{-1} (x_0) , r_0/ \sqrt 2) $ and 
        $\Omega' := B_{\hat g}( \psi_{\rho_0}^{-1} (x_0) ,  2 r_0)$
        \begin{multline} \label{proof higher order est, eqn model interior est}
            \| H \|_{C^0 \left( \Omega' \times \left[ -\frac{t_0}{\rho_0^2}, 0\right] ,  \hat g \right)} +  \| H(\cdot, -t_0/\rho_0^2) \|_{C^{2, \gamma}_*( \Omega' , \hat g )} < \epsilon \\
            \implies \| H \|_{C^{2, \gamma}_* \left( \Omega \times \left[ -\frac{t_0}{\rho_0^2}, 0\right] ,  \hat g \right) }
            \le C \left(  \| H \|_{C^0 \left( \Omega' \times \left[ -\frac{t_0}{\rho_0^2}, 0\right] ,  \hat g \right) } + \| H(\cdot, -t_0/\rho_0^2) \|_{C^{2, \gamma}_*(\Omega')} \right).
        \end{multline}
    
        Undoing the scaling $\rho_0^{-2}$ and diffeomorphisms $\psi_{\rho_0}$, it follows that
        \begin{gather} 
        \label{proof higher order est, eqn 3.5}
            \sup_{t \in [0, T]} \| h(t) \|_{C^{0}_{-\delta} (\Omega_0' , g_0) } + \| h ( 0)\|_{C^{2, \gamma}_{-\delta}(\Omega', g_0)} <  \epsilon' \\
            \nonumber
            \implies \\
            \label{proof higher order est, eqn 3.6}
            \begin{split}
            \rho(x_0) \sup_{x \in \Omega_0} | \nabla h (x, t_0) |_{g_0} + \rho(x_0)^2 \sup_{x \in \Omega_0} | \nabla^2 h(x, t_0) |_{g_0} + \rho(x_0)^{2 + \gamma} [ \nabla^2 h (\cdot , t_0) ]_{C^{2, \gamma}_* (\Omega_0, g_0 ) } \\
            \le C'\rho(x_0)^{-\delta} \left( \sup_{t \in [0, T]} \| h(t) \|_{C^{0}_{-\delta} (\Omega_0', g_0)} + \| h ( 0)\|_{C^{2, \gamma}_{-\delta} (\Omega_0' , g_0)} \right) 
            \end{split}
        \end{gather}
        where 
        $\Omega_0 := B_{ g_0} \left(  x_0 , \frac1{\sqrt2} r_0 \rho_0 \right)$,
        $\Omega_0' := B_{ g_0}( x_0 , 2 r_0 \rho_0)$, and
        $\epsilon', C'>0$ depend on $M^n, g_0, \gamma, \delta, r_0$ but are independent of $(x_0, t_0) \in M \setminus K' \times [0, T]$.
        Indeed,
        \begin{multline*}
            |\hat \nabla H (\psi_{\rho_0}^{-1} (x) , 0) |_{\hat g}
            = \sqrt{\hat g^{ab} \hat g^{cd} \hat g^{ij} \hat \nabla_i H_{ac} \hat \nabla_j H_{bd} (\psi_{\rho_0}^{-1}(x), 0)} 
            = \sqrt{\rho_0^2 g_0^{ab} g_0^{cd} g_0^{ij} \nabla_i h _{ac} \nabla_j h_{bd}(x, t_0)} \\
            = \rho_0 | \nabla h(x, t_0) |_{g_0}
        \end{multline*}
        and similarly $| \hat \nabla^2 H (\psi_{\rho_0}^{-1}(x), 0) |_{\hat g} = \rho_0^2 | \nabla^2 h (x, t_0) |_{g_0}$ for all $x \in \Omega_0$.
        Multiplying both sides of \eqref{proof higher order est, eqn 3.6} by $\rho(x_0)^{-\delta}$,
        using that $\rho(x)$ is comparable to $\rho(x_0)$ for all $x \in \Omega_0$ (which follows from \eqref{proof local higher order est, eqn 2.1} and \eqref{proof local higher order est, eqn 4.15}), and taking a supremum in $t_0 \in [0, T]$ then gives 
        \begin{multline} \label{proof local higher order est, eqn 4.2}
            \sup_{t \in [0, T]} \| h \|_{C^{0}_{-\delta} (\Omega_0' , g_0) } + \| h ( 0)\|_{C^{2, \gamma}_{-\delta}(\Omega_0', g_0)} <  \epsilon' \\
            \implies 
            \sup_{t \in [0, T]} \| h(t) \|_{C^{2, \gamma}_{-\delta} ( \Omega_0 , g_0) } 
            \le C(M^n, g_0, \gamma,\delta, r_0) \cdot \left( \sup_{t \in [0, T]} \| h(t) \|_{C^{0}_{-\delta} (\Omega_0', g_0)} + \| h(0) \|_{C^{2, \gamma}_{-\delta} (\Omega_0', g_0)} \right) .
        \end{multline}
        This completes the proof of the claim.
    \end{claimproof}

    On the compact set $K' \subset M$, the weighted H{\"o}lder norms are comparable to the standard H{\"o}lder norms since $1 \le \inf_{x \in K'} \rho(x) \le \sup_{x \in K'} \rho(x) < \infty$.
    Therefore, standard interior estimates (see e.g. \cite[Propositions C.3, C.4]{Stolarski25}) can be applied directly to the PDE \eqref{proof local higher order est, eqn 1} to obtain an estimate identical to \eqref{proof local higher order est, eqn 4.2} when $x_0 \in K'$ and $0 < r_0 < \frac1{100}$.
    Such an estimate together with Claim \ref{proof local higher order est, claim 1} completes the proof of statement \eqref{lem local C2gamma-delta est, item global in time} from the lemma when $0 < r_0 < \frac1{100}$. 

    The proof of statement \eqref{lem local C2gamma-delta est, item local in time} from the lemma when $0 < r_0 < \frac1{100}$ follows by similar logic as the argument above, except that \cite[Proposition C.3]{Stolarski25} is used to replace \eqref{proof higher order est, eqn model interior est} above with
    \begin{equation} \tag{\ref{proof higher order est, eqn model interior est}'}
        \| H \|_{C^0 \left( \Omega' \times \left[ -4 r_0^2, 0\right] ,  \hat g \right)}  < \epsilon \quad 
        \implies 
        \quad \| H \|_{C^{2, \gamma}_* \left( \Omega \times \left[ -\frac{r_0^2}{2}, 0\right] ,  \hat g \right) }
        \le C  \| H \|_{C^0 \left( \Omega' \times \left[ - 4 r_0^2, 0\right] ,  \hat g \right) } 
    \end{equation}
    and remove the initial data terms from the remaining statements.

    Finally, a covering argument then proves the lemma for general $r_0 > 0$.
\end{proof}

An analogous global $C^{2, \gamma}_{-\delta}$ estimate also holds.

\begin{lemma}[Global $C^{2, \gamma}_{-\delta}$ Estimate]\label{lemma:globalC2gamma-deltaestimate}
    Let $(M^n, g_0)$ be a Ricci-flat ALE manifold.
    For any $\delta \geq 0$ and $\gamma \in (0,1)$,
    there exists small $0 < \eps_0$ and large $C>1$ 
    both depending only on $M^n , g_0, \gamma, \delta$ such that the following holds:

    If $g(t)$ is a smooth, complete Ricci-DeTurck flow with respect to $g_0$ on $M$ for $t \in [0,T]$ and 
        $$\| g(0) - g_0 \|_{C^{2, \gamma}_{-\delta}(M, g_0) } + \sup_{t \in [0, T]} \| g(t) - g_0 \|_{C^0_{-\delta}(M, g_0)} \le \eps \leq \eps_0 ,$$
    then
        $$\sup_{t \in [0,T]} \| g(t) - g_0 \|_{C^{2, \gamma}_{-\delta}(M, g_0)}
        \le C \eps.$$ 
\end{lemma}
\begin{proof}
    This follows from covering $M$ by balls of the form $B_{g_0}( x, \rho(x) )$ and applying the local $C^{2, \gamma}_{-\delta}$-estimate from Lemma \ref{lemma:localC2gamma-deltaestimate} \eqref{lem local C2gamma-delta est, item global in time} with $r_0 = \frac12$.
\end{proof}

\vspace{10mm}

\section{$L^p$-estimates for the scalar heat kernel on an ALE space}\label{sec:proofofLp}


\begin{lemma}\label{lemma:heatkernelLpests}
    Let $a \geq 1$ and $\alpha, \beta \geq 0$ with $a < \frac{n}{\beta}.$ 
    We have
        \begin{equation}\label{HLarhobnorm}
        \left\|  \frac{d(x,\cdot)^{2\alpha} H(x,\cdot, t) }{t^\alpha \rho(\cdot)^\beta} \right\|_{L^{a}} \leq \frac{C_{g_0,a,\alpha, \beta} }{t^{\frac{n(a - 1)}{2a}} \left( \rho(x)^{2} + t \right)^{\frac{\beta}{2}} }.
        \end{equation}
\end{lemma}
\begin{proof}
    Applying (\ref{heatkernelest}), we may write
    \begin{equation}\label{Laest}
    \begin{split}
        \left\|  \frac{d(x,\cdot)^{2\alpha} H(x,\cdot, t) }{t^\alpha \rho(\cdot)^\beta} \right\|_{L^{a}}^a & \leq C \int_M \frac{d(x,y)^{2a\alpha} e^{-a c d(x,y)^2/t} }{t^{a \left( \frac{n}{2} + \alpha \right)} \rho(y)^{a\beta} } \, dV_y \\
        & = \frac{C}{t^{\frac{n(a - 1)}{2}}} \int_M \frac{d(x,y)^{2a\alpha} e^{-a d(x,y)^2/4t} }{t^{\frac{n}{2} + a\alpha } \rho(y)^{a\beta} } \, dV_y
        \end{split}
        \end{equation}
    We split the last integral into two pieces: 
    \begin{equation}\label{I-IIdefinition}
    \begin{split}
    \int_M \frac{d(x,y)^{2a\alpha} e^{-a cd(x,y)^2/t} }{t^{\frac{n}{2} + a\alpha } \rho(y)^{a\beta} } \, dV_y & = \left( \int_{\left\{ \rho(y) \leq \frac{\rho(x)}{2} \right\} } + \int_{\left\{ \rho(y) \geq \frac{\rho(x)}{2} \right\} } \right)  \frac{d(x,y)^{2a\alpha} e^{-a cd(x,y)^2/t} }{t^{\frac{n}{2} + a\alpha } \rho(y)^{a\beta} } \, dV_y \\
    & =: I + II.
    \end{split}
    \end{equation}
    For $\rho(y) \leq \frac{\rho(x)}{2},$ we have
    $$\rho(x) \leq d(x,y) + \rho(y) \leq d(x,y) + \frac{\rho(x)}{2}$$
    which gives
    $$\frac12 \rho(x) \leq d(x,y) \leq \frac32 \rho(x).$$
    This gives
    \begin{equation}\label{Ifinalest}
    \begin{split}
    I & \leq C e^{-a \rho(x)^2/16t} \left( \frac{\rho(x)^{2a\alpha}}{t^{\frac{n}{2} + a\alpha}} \right) \int_{\left\{ \rho(y) \leq \frac{\rho(x)}{2} \right\} } \frac{dV_y}{\rho(y)^{a\beta}} \\
    & \leq C e^{-a \rho(x)^2/16t} \left( \frac{\rho(x)^2}{t} \right)^{\frac{n}{2} + a\alpha}  \rho(x)^{- a\beta},
    \end{split}
    \end{equation}
    where we have used $n -1 - a\beta > -1.$ 
    
    For $\rho(y) \geq \frac{\rho(x)}{2},$ we have
    $$d(x,y) \leq \rho(x) + \rho(y) \leq 3 \rho(y)$$
    and
    $$ \frac{1}{\rho(y)}\leq \frac{5}{d(x,y) + \rho(x)}.$$
    This gives
    \begin{equation*}
    \begin{split}
    II & \leq C\int_{\left\{ \rho(y) \geq \frac{\rho(x)}{2} \right\} } \frac{d(x,y)^{2a\alpha} e^{-a c d(x,y)^2/t} }{t^{\frac{n}{2} + a\alpha } \left( d(x,y) + \rho(x) \right)^{a\beta} } \, dV_y.
    \end{split}
    \end{equation*}
    Using polar coordinates $y = \exp_x(r \theta) $ centered at $x$ and Bishop-Gromov volume comparison, we can estimate this from above
    \begin{align*}
        II \le{}& C \int_0^\infty \frac{r^{2 a \alpha} e^{-ca r^2/t}}{t^{\frac n2 + a \alpha} ( r + \rho(x))^{a \beta}} r^{n-1} dr 
        \\
        \le{}& \frac C{\rho(x)^{a\beta}} \int_0^{\rho(x)} \frac{r^{2 a \alpha} e^{-ca r^2/t}}{t^{\frac n2 + a \alpha} } r^{n-1} dr 
        + \frac C{t^{a \beta/2}} \int_{\rho(x)}^\infty \frac{r^{2 a \alpha} e^{-ca r^2/t}}{t^{\frac n2 + a \alpha} ( r /\sqrt t)^{a \beta}} r^{n-1} dr \\
        ={}& \frac C{\rho(x)^{a\beta}} \int_0^{\rho(x)/\sqrt{t}} u^{2 a \alpha + n-1} e^{-ca u^2} du 
        + \frac C{t^{a \beta/2}} \int_{\rho(x)/\sqrt t}^\infty u^{2 a \alpha - a \beta + n-1} e^{- ca u^2 } du
        && ( u = r/ \sqrt t) \\
        \le{}& \frac C{\rho(x)^{a \beta}} \frac{(\rho(x)/\sqrt t)^{2 a \alpha + n}}{ (\rho(x)/\sqrt t)^{2 a \alpha + n} + 1  }
        + \frac C{t^{a \beta/2}} e^{- c a \rho(x)^2 / t} ,
    \end{align*}
    where we have used $2 a \alpha - a \beta + n - 1 > -1$ and $c$ may have decreased in the last line.

    This simplifies to 
    \begin{equation} \label{IIfinalest2}
        II \le
        \frac C{\rho(x)^{a \beta} } \left( \frac{(\rho(x)/\sqrt t)^{2 a \alpha + n}}{(\rho(x)/\sqrt t)^{2 a \alpha+n} + 1}
        +  \left( \frac {\rho(x)} {\sqrt t}\right)^{a \beta}  e^{- c a \left(\frac{\rho(x)}{\sqrt t} \right)^2} \right) .
    \end{equation}
    Combining \eqref{Ifinalest} and \eqref{IIfinalest2} and using that $a \beta \le 2 a \alpha + n$ gives
    \begin{multline}
        I + II \\
        \le C \rho(x)^{- a\beta} \left[  e^{-c a \left( \frac {\rho(x)}{\sqrt t} \right)^2 } \left( \frac{\rho(x)}{\sqrt t} \right)^{2 a \alpha + n} +\frac{(\rho(x)/\sqrt t)^{2 a \alpha + n}}{(\rho(x)/\sqrt t)^{2 a \alpha+n} + 1}
        +  \left( \frac {\rho(x)} {\sqrt t}\right)^{a \beta}  e^{- c a \left(\frac{\rho(x)}{\sqrt t} \right)^2}  \right] \\
        \le C \rho(x)^{- a\beta} \frac{ ( \rho(x)/ \sqrt t )^{a \beta} }{(\rho(x)/\sqrt t)^{a \beta} +1}.
    \end{multline}
    Substituting this inequality back into (\ref{Laest}), we obtain
    \begin{equation*}
    \begin{split}
        \left\|  \frac{d(x,\cdot)^{2\alpha} H(x,\cdot, t) }{t^\alpha \rho(\cdot)^\beta} \right\|_{L^{a}} & \leq \frac{C}{t^{\frac{n(a - 1)}{2a}} \rho(x)^\beta} \left( \frac{ \rho(x)^{2} }{ \rho(x)^{2} + t} \right)^{\frac{\beta}{2}} \\
        & \leq \frac{C}{t^{\frac{n(a - 1)}{2a}} \left( \rho(x)^{2} + t \right)^{\frac{\beta}{2}} },
        \end{split}
    \end{equation*}
    which is (\ref{HLarhobnorm}).
\end{proof}

\begin{lemma}\label{lemma:youngsconvolution} 
(a) Let $a,b \geq 1$ with $\frac{1}{a} + \frac{1}{b} = 1$ and $\alpha, \beta \ge 0$ with $\beta < \frac{n}{a}.$ For $f \in L^b(M),$ we have
\begin{equation*}
       \begin{split}
        \left| \int_M \frac{d(x, y)^{2 \alpha} H(x,y,t) f(y) }{t^\alpha \rho(y)^\beta} \, dV_y \right| & \leq \frac{C_{g_0,a,b, \alpha, \beta} }{t^{\frac{n(a - 1)}{2a} } \left(\rho(x)^2 + t \right)^{\frac{\beta}{2}} } \| f \|_{L^b(M)}.
        \end{split}
        \end{equation*}

        \vspace{2mm}
        
\noindent (b) Let $a,b,c \geq 1$ with $\frac{1}{a} + \frac{1}{b} = 1 + \frac{1}{c}$ and $\alpha, \beta \ge 0$ with $\beta < \frac{n}{a}.$ For $f \in L^b(M)$ and $R \geq R_0,$ we have
\begin{equation*}
       \begin{split}
        \left\| \int_M \frac{d(\cdot, y)^{2 \alpha} H(\cdot,y,t) f(y) }{t^\alpha \rho(y)^\beta} \, dV_y \right\|_{L^c(M \setminus B_R)} & \leq \frac{C_{g_0,a,b, \alpha, \beta} }{t^{\frac{n(a - 1)}{2a} } \left(R^2 + t \right)^{\frac{(c - a)\beta}{2c}} } \| f \|_{L^b(M)}.
        \end{split}
        \end{equation*}

        \vspace{2mm}

\noindent (c) For $b \geq 1$ and $1 \leq R \leq \infty,$ we have
\begin{equation*}
       \begin{split}
        \left\| \int_{B_R} H(\cdot,y,t) f(y) \, dV_y \right\|_{L^b(M)} & \leq \min \left\{ 1, \left( \frac{C_{g_0,b} R^2 }{t} \right)^{\frac{n(b - 1)}{2b}} \right\} \| f \|_{L^b(M)}.
        \end{split}
        \end{equation*}
\end{lemma}
\begin{proof}
The estimate of ($a$) follows by applying H\"older's inequality and (\ref{HLarhobnorm}) to the expression.

Next, note that given $\Omega_1, \Omega_2 \subset M,$ \cite[Proposition 3.64]{AubinBook} gives
       \begin{multline}\label{youngsconvolution:eq}
           \left\| \int_{\Omega_2} F(\cdot,y) f(y) \, dV_y \right\|_{L^c(\Omega_1)} \\ 
           \leq \left( \sup_{x \in \Omega_1} \|F(x, \cdot) \|_{L^a(\Omega_2)} \right)^{\frac{c - a}{c}} \left( \sup_{y \in \Omega_2} \|F(\cdot, y) \|_{L^a(\Omega_1)} \right)^{\frac{a}{c}} \| f \|_{L^b(\Omega_2)}.
       \end{multline}
        For ($b$), take $\Omega_1 = M \setminus B_R,$ $\Omega_2 = M,$ and $F(x,y) = \frac{d(x, y)^{2 \alpha} H(x,y,t) }{t^\alpha \rho(y)^\beta}.$ Then (\ref{HLarhobnorm}) gives the desired result.

        For ($c$), take $a = 1,$ $b = c,$ $\Omega_1 = M,$ and $\Omega_2 = B_R.$ Note that $\|H(x, \cdot, t)\|_{L^1(B_R)} \leq \|H(x, \cdot, t)\|_{L^1(M)} \equiv 1.$ For $t \geq R^2,$ note that
        $$\|H(x, \cdot, t)\|_{L^1(B_R)} \leq \frac{C R^n}{t^{\frac{n}{2} } }.$$
        Therefore (\ref{youngsconvolution:eq}) gives the result.        
\end{proof}

\vspace{10mm}

\section{$L^p$-estimates along Ricci-DeTurck flow}

\subsection{Estimate with $\frac{n}{n-2} < p < n$}\label{ss:proofofLpprop}

\begin{prop}\label{prop:Lpbootstrap} 
Let $(M, g_0)$ be a Ricci-flat ALE space of order $\tau \geq \tau_0,$ where $0 < \tau_0 < n-2.$ Fix $\frac{n}{n-2} < p \leq \infty$ and suppose that there exists $1 < q < \min\{p, \frac{n}{2} \}$ such that
\begin{equation}\label{Lpbootstrap:tauassumptions}
\begin{split}
\left( 1 - \frac{q}{p} \right)\left(2 + \tau_0 \right) & > 2.
\end{split}
\end{equation}
Also let $\ell > 0$ and $0 < r < \min \{p-q, 2\}$ be such that 
\begin{equation}\label{qlplusonelessthann}
q \ell + 1 < n
\end{equation}
and
\begin{equation}\label{2-relllessthan...}
\ell < \frac{n(p-r)}{(2 - r) p}.
\end{equation}

Define $v(x,t)$ as in (\ref{v(x,t)def}). Let $t_1 \geq 1$ and $0 < \eps < \lambda - 1,$ where $\lambda$ is sufficiently close to one as in Lemma \ref{lemma:hpevolution}. Suppose that $h$ solves the Ricci-DeTurck flow with
\begin{equation}\label{Lpbootstrap:Lpetaassumption}
    \sup_{0 \leq t \leq t_1} \|h(\cdot,t) \|_{L^p} \leq \eps,
\end{equation}
\begin{equation}\label{Lpbootstrap:gradassumption}
    \sup_{t_1 - 1 \leq t \leq t_1} | \nabla h(x,t) | \leq \eps v(x,t_1), 
\end{equation}
for $x \in M,$ and
\begin{equation}\label{Lpbootstrap:hpointwiseassumption}
    | h(x,t) | \leq \eps v(x,t)^\ell
\end{equation}
for $x\in M$ and $0 \leq t \leq t_1.$
We then have
\begin{equation}\label{Lpbootstrap:Linftybound}
\begin{split}
    |h|(x,t_1) & \leq C \left( t_1^{-\frac{n}{2p}} \| h(0) \|_{L^p} + \eps \rho(x)^{- \min \left\{ \frac{2}{q^2} + \frac{\tau}{q}, \frac{n}{p} + \frac{\tau}{q}, \frac{n}{q^2} - \frac{2}{q} \right\} } \right) \\
    & \qquad + C \left( \eps^{1 + \frac{1}{q}} v(x,t_1)^{\ell + \frac{1}{q}} + \eps^{1 + \frac{2}{q} } v(x,t_1)^{\min\left\{ \ell \left( 1 + \frac{1}{q} \right), \frac{2}{q} + \frac{n}{pq} \right\} } \right),
    \end{split}
\end{equation}
and, for $R \geq 1,$
        \begin{equation}\label{Lpbootstrap:Lpbound}
        \begin{split}
        \|h(t_1) \|_{L^p(M \setminus B_R)} & \leq \|h(0) \|_{L^p ( M \setminus B_R )} + \min \left\{ 1, \left( \frac{C R^2}{t_1 } \right)^{\frac{n}{2}\left( \frac{1}{q} - \frac{1}{p} \right)} \right\} \|h(0)\|_{L^p(B_R)} \\
        & + C \eps \left( \frac{1}{R^{\left( \frac{1}{q} - \frac{1}{p} \right) \left(2 + \tau_0 \right) - \frac{2}{q} } } + \eps^{\frac{1}{q}} \left( \frac{1}{R^{\frac{1}{q} - \frac{1}{p} }} + \frac{1}{t_1^{\frac{1}{2q}} } \right) + \eps^{\frac{2}{q}} \left( \frac{1}{R^{ \frac{ \left( p - q - r \right) \ell}{q(p-r)} }} + \frac{1}{t_1^{\frac{\min\left\{1 + \frac{nr}{2p} , \frac{(2 - r)\ell}{2} \right\} }{q} } } \right)\right).
        \end{split}
    \end{equation}
\end{prop}
\begin{proof}

In the formula of Lemma \ref{lemma:Duhamel}, we break up the right-hand side as follows:
\begin{equation}\label{Lpbootstrap:RHSsplitting}
    \begin{split}
        u_1^q(x) & := \int_M H(x,y,t_1) |h|^q(y,0) \, dV_y \\
        u_2^q(x) & := \int_0^{t_1} \!\!\!\! \int_M H(x,y,t_1 - t) |\Rm(g_0)| |h(y,t)|^q \, dV_y dt \\
        u_3^q(x) & := \int_{t_1 - 1}^{t_1} \!\! \int_M H(x,y,t_1 - t) \left( \frac{1}{\sqrt{t_1 - t}} + \frac{d(x,y)}{t_1-t} \right) |h|^{\frac{q}{2} + 1} |\nabla |h|^{\frac{q}{2}} | \, dV_y dt \\
        u_4^q(x) & := \int_0^{t_1 - 1} \!\!\!\! \int_M H(x,y,t_1 - t) \left( \left( \frac{1}{\sqrt{t_1 - t}} + \frac{d(x,y)}{t_1 - t} \right) |h|^{\frac{q}{2} + 1} |\nabla |h|^{\frac{q}{2}} | - c_q |\nabla |h|^{\frac{q}{2}} |^2 \right) \, dV_y dt.
    \end{split}
\end{equation}
According to Lemma \ref{lemma:Duhamel}, we have
\begin{equation}\label{hlessthanu1tou4}
    |h|(x,t_1) \leq u_1(x) + C \left( u_2(x) + u_3(x) + u_4(x) \right).
    \end{equation}
    We now use Lemma \ref{lemma:youngsconvolution} to estimate $u_i$ pointwise and in $L^{p}(M \setminus B_R)$ for $i = 1, \ldots, 4.$

       \vspace{2mm}
       
       \noindent {\bf First term.} \emph{Pointwise estimate.} Taking $\alpha = \beta = 0,$ $b = \frac{p}{q}$ and $a = \frac{1}{1 - \frac{q}{p}} = \frac{p}{p - q}$ in Lemma \ref{lemma:youngsconvolution}$a,$ we have
       $$u^q_1(x) \leq \frac{C}{t_1^{\frac{nq}{2p}}} \||h(0)|^q\|_{L^{p/q}} = \frac{C}{t_1^{\frac{nq}{2p}}} \||h(0)|\|^q_{L^{p}}.$$
       This gives
       \begin{equation}\label{u1xest}
          u_1(x) \leq \frac{C}{t_1^{\frac{n}{2p}}} \|h(0)\|_{L^{p}}.
          \end{equation}
          
\noindent \emph{$L^p$-estimate.} Taking $b = \frac{p}{q}$ in Lemma \ref{lemma:youngsconvolution}$c$, we have
     \begin{equation}\label{u1Lpest}
        \| u_1 \|_{L^p{(M)}} \leq \|h(0)\|_{L^{p}{(M)}}.
          \end{equation}
If $t_1 \geq R^2,$ we further break up this term as
$$u_1(x) \leq u_1'(x) + u_1''(x),$$
where
$$u_1'(x) := \sqrt[q]{\int_{M \setminus B_R} H(x,y,t_1) |h|^q(y,0) \, dV_y}$$
and
$$u_1''(x) := \sqrt[q]{\int_{B_R} H(x,y,t_1) |h|^q(y,0) \, dV_y}.$$
To estimate $u_1'(x),$ by the same argument as for $u_1$, we have
    $$\| u_1' \|_{L^p(M)} \leq \|h(0)\|_{L^{p}(M \setminus B_R)}.$$
For $u_1'',$ Lemma \ref{lemma:youngsconvolution}$c$ gives
$$ \|u_1''\|_{L^p(M)}^q \leq \min \left\{ 1, \left( \frac{C R^2}{t_1 } \right)^{\frac{n}{2}\left( 1- \frac{q}{p} \right)} \right\} \|h(0) \|_{L^p{(B_R)}}^q.$$
Taking $q$-th roots, we have
\begin{equation}
\|u_1''\|_{L^p(M)} \leq \min \left\{ 1, \left( \frac{C R^2}{t_1 } \right)^{\frac{n}{2}\left( \frac{1}{q} - \frac{1}{p} \right)} \right\} \|h(0) \|_{L^p{(B_R)}}.
\end{equation}

       \vspace{2mm}

     \noindent {\bf Second term.}   \emph{Pointwise estimate.}  Note that since $(M,g_0)$ is Ricci-flat ALE of order $\tau > 0,$ Lemma \ref{lemma: Rm est for ALE} implies
    \begin{equation*}
        u^q_2(x) \leq C \int_0^{t_1} \!\!\!\! \int_M \frac{H(x,y,t_1 - t)}{ \rho(y)^{2 + \tau} } |h(y,t)|^q \, dV_y dt.
    \end{equation*}
     Let
     $$f_2(x,t) = \int_M \frac{H(x,y,t_1 - t) }{ \rho(y)^{2 + \tau} } |h(y,t)|^q \, dV_y.$$

 \noindent \underline{{\bf Case 1:} $\frac{nq}{2} < p.$} Applying Lemma \ref{lemma:youngsconvolution}$a$ with $b = \frac{p}{q}$ and $a = \frac{p}{p - q},$ and
 \begin{equation}\label{Secondtermbetarestriction}
     \beta = \min \left\{ 2 + \tau, \frac{n (p - q)}{qp} \right\} < \frac{n}{a},
 \end{equation}
 we have
        \begin{equation}
            |f_2(x,t)| \leq \frac{C \|h(t) \|_{L^p}^q }{ \left( t_1 - t \right)^{\frac{nq}{2p}} \left( \rho(x)^2 + t_1 - t \right)^{\frac{\beta}{2}} }.
        \end{equation}
        We have
        \begin{equation}\label{dumbbetacalc}
            \frac{nq}{2p} + \frac{\beta}{2} = \frac{n}{2p}\left( \frac{p}{q} + q - 1 \right) > \frac{n}{2q} > 1 \qquad \text{when } \beta = \frac{n (p - q)}{qp}
        \end{equation}
        and 
        \begin{equation}
            \frac{nq}{2p} + \frac \beta 2 = \frac{nq}{2p} + 1+ \frac \tau2 > 1 \qquad \text{when } \beta = 2+ \tau.
        \end{equation}
     In either case, using this inside the time integral, we have
     \begin{equation*}
     \begin{split}
         u_2^q(x) \leq \int_0^{t_1} |f_2(x,t)| \, dt 
         & \leq C \sup_{0 \leq t \leq t_1} \|h(t) \|_{L^p}^q \int_0^{t_1} \frac{1}{ \left( t_1 - t \right)^{\frac{nq}{2p}} \left( \rho(x)^2 + t_1 - t \right)^{\frac{\beta}{2} }} \, dt \\
         & \leq C \sup_{0 \leq t \leq t_1} \|h(t) \|_{L^p}^q \left( \frac{1}{\rho(x)^{ \beta} } \int_0^{\rho(x)^2}  u^{-\frac{nq}{2p}} \, du + \int_{\rho(x)^2}^\infty u^{-\left( \frac{nq}{2p} + \frac{\beta}{2} \right) } \, du \right) \\
         & \leq \frac{C}{\rho(x)^{\frac{n q}{p} + \beta - 2} } \sup_{0 \leq t \leq t_1} \|h(t) \|_{L^p}^q.
         \end{split}
     \end{equation*}
     Here we have used $\frac{nq}{2p} < 1$ and $\frac{nq}{2p} + \frac{\beta}{2} > 1.$
     Taking $q$-th roots, we get
          \begin{equation}\label{u2xest}
     \begin{split}
         u_2(x) \leq \frac{C}{\rho(x)^{\frac{n}{p} + \frac{\beta - 2}{q}} } \sup_{0 \leq t \leq t_1} \|h(t) \|_{L^p}.
         \end{split}
     \end{equation}
     By (\ref{dumbbetacalc}), we have
     $$\frac{nq}{p} + \beta - 2 > \frac{n}{q} - 2 \qquad \text{if } \beta = \frac{n (p-q)}{qp}.$$
     On the other hand,
     $$\frac{nq}{p} + \beta - 2 = \frac{nq}{p} + \tau \qquad \text{if } \beta = 2+\tau.$$
     Therefore (\ref{u2xest}) simplifies to 
\begin{equation}\label{u2xcase1est}
     \begin{split}
         u_2(x) \leq \frac{C}{\rho(x)^{\min\{ \frac{n}{p} + \frac{\tau}{q}, \frac{n}{q^2} - \frac{2}{q} \} } } \sup_{0 \leq t \leq t_1} \|h(t) \|_{L^p}.
         \end{split}
     \end{equation}
     
     \noindent \underline{{\bf Case 2:} $\frac{nq}{2} \geq p.$} Let $p' := \frac{nq^2}{2} > \frac{nq}{2} \geq p.$ Note that by (\ref{Lpbootstrap:hpointwiseassumption}), we have
     $$\sup_{0 \leq t \leq t_1} \|h\|_{L^\infty(M) } \leq C \eps.$$
     By interpolation with {\eqref{Lpbootstrap:Lpetaassumption}}, we have
     $$\sup_{0 \leq t \leq t_1} \|h\|_{L^{p'}(M) } \leq C \eps.$$
     Running the argument of Case 1 with $p'$ in place of $p,$ we get
      \begin{equation}\label{u2xcase2est}
     \begin{split}
         u_2(x) \leq \frac{C}{\rho(x)^{\min \left\{ \frac{2}{q^2} + \frac{\tau}{q}, \frac{n}{q^2} - \frac{2}{q} \right\} } } \sup_{0 \leq t \leq t_1} \|h(t) \|_{L^{p'}}.
         \end{split}
     \end{equation}  
\noindent \emph{$L^p$-estimate.} We first apply Minkowski's integral inequality:
     \begin{equation}\label{u2Lpfirstest}
     \begin{split}
        \| u_2\|^q_{L^{p}(M \setminus B_R)} = \| u_2^q(x) \|_{L^{\frac{p}{q}}(M \setminus B_R)} & \leq C \left( \int_{M \setminus B_R} \left( \int_0^{t_1} f_2(x,t) dt \right)^{\frac{p}{q}} \, dV_x \right)^{\frac{q}{p}} \\
        & \leq C \int_0^{t_1} \|f_2(\cdot, t) \|_{L^{\frac{p}{q}}(M \setminus B_R)} \, dt.
        \end{split}
    \end{equation}
    Applying Lemma \ref{lemma:youngsconvolution}$b$, with $b = c = \frac{p}{q},$ $a = 1,$ and $\beta = 2 + \tau_0 < n = \frac{n}{a},$ 
    we get
    $$\|f_2(\cdot, t) \|_{L^{\frac{p}{q}}(M \setminus B_R)} \leq \frac{C \|h(t) \|_{L^p(M)}^q }{ \left( R^2 + t_1 -t \right)^{ \left( 1 - \frac{q}{p} \right)\left(1 + \frac{\tau_0}{2} \right)} } .$$
    Returning to (\ref{u2Lpfirstest}), we get
   \begin{equation}\label{u2Lpest}
       \begin{split}
           \| u_2 \|_{L^{p}(M \setminus B_R)} \leq \frac{C \sup_{0 \leq t \leq t_1} \|h(t) \|_{L^p(M)} }{ R^{ \frac{ \left( 1 - \frac{q}{p} \right)\left(2 + \tau_0 \right) - 2 }{q} } },
       \end{split}
   \end{equation}
   where we used (\ref{Lpbootstrap:tauassumptions}). 

       \vspace{2mm}
       
     \noindent {\bf Third term.} \emph{Pointwise estimate.} Applying (\ref{Lpbootstrap:gradassumption}), we have
    $$| \nabla |h|^{\frac{q}{2}} | \leq \frac{q}{2} |h|^{\frac{q}{2} - 1} | \nabla h| \leq \frac{q}{2}\eps |h|^{\frac{q}{2} - 1} \left( \frac{1}{\rho} + \frac{1}{\sqrt{t_1}} \right)$$
    for $t_1 - 1 \leq t \leq t_1.$
    This gives
    \begin{equation}\label{u3firstbound}
       u^q_3(x) \leq C \eps \int_{t_1 - 1}^{t_1} \!\! \int_M \left( \frac{1}{\sqrt{t_1 - t}} + \frac{d(x,y) }{ t_1 - t } \right)H(x,y,t_1 - t) \left( \frac{1}{\rho(y)} + \frac{1}{\sqrt{t_1}} \right) |h(y,t)|^{q} \, dV_y dt.
       \end{equation}
    To estimate $u_3(x)$ pointwise, we use (\ref{Lpbootstrap:hpointwiseassumption}) to give
    \begin{equation*}
    \begin{split}
       u^q_3(x) & \leq C \eps^{q + 1} \int_{t_1 - 1}^{t_1} \!\! \int_M \left( \frac{1}{\sqrt{t_1 - t}} + \frac{d(x,y) }{ t_1 - t } \right)H(x,y,t_1 - t) \left( \frac{1}{\rho(y)} + \frac{1}{\sqrt{t_1}} \right)^{q\ell + 1}\, dV_y dt \\
       & \leq C \eps^{q + 1} \int_{t_1 - 1}^{t_1} \!\! \int_M \left( \frac{1}{\sqrt{t_1 - t}} + \frac{d(x,y) }{ t_1 - t } \right)H(x,y,t_1 - t) \left( \frac{1}{\rho(y)^{q\ell + 1}} + \frac{1}{t_1^{\frac{q\ell + 1}{2}} } \right) \, dV_y dt \\
        & \leq C \eps^{q + 1} \int_{t_1 - 1}^{t_1} \!\! \left( \frac{1}{\sqrt{t_1 - t} \left( \rho(x)^2 + t_1 - t \right)^{\frac{q \ell + 1}{2} } } + \frac{1}{t_1^{\frac{q\ell + 1}{2}} \sqrt{t_1 - t} } \right) \, dt \\
        & \leq C \eps^{q + 1} \left( \frac{1}{\rho(x)} + \frac{1}{\sqrt{t_1}} \right)^{q \ell + 1},
       \end{split}
       \end{equation*}
    where we have used Lemma \ref{lemma:heatkernelLpests} with $a = 1$ and $\beta = q \ell + 1 < n.$ Taking $q$-th roots, we get
    \begin{equation}\label{u3xest}
        u_3(x) \leq C \eps^{\frac{1}{q}} \eps \left( \frac{1}{\rho(x)} + \frac{1}{\sqrt{t_1}} \right)^{\ell + \frac{1}{q} }.
    \end{equation}
    
    \noindent \emph{$L^p$-estimate.} Let $f_3$ be the integrand of (\ref{u3firstbound}).
     We first apply Minkowski's integral inequality:
     \begin{equation}\label{u3Lpfirstest}
     \begin{split}
        \| u_3\|^q_{L^{p}(M \setminus B_R)} 
        &= \| u_3^q(x) \|_{L^{\frac{p}{q}}(M \setminus B_R)} \\
        & \leq C \eps \left( \int_{M \setminus B_R} \left( \int_{t_1 - 1}^{t_1} f_3(x,t) dt \right)^{\frac{p}{q}} \, dV_x \right)^{\frac{q}{p}} \\
        & \leq C \eps \int_{t_1 - 1}^{t_1} \|f_3(\cdot, t) \|_{L^{\frac{p}{q}}(M \setminus B_R)} \, dt \\
        & \leq C \eps \int_{t_1 - 1}^{t_1} \left( \|f_3'(\cdot, t) \|_{L^{\frac{p}{q}}(M \setminus B_R)} + \|f_3''(\cdot, t) \|_{L^{\frac{p}{q}}(M \setminus B_R)} \right) \, \frac{dt}{\sqrt{t_1 - t} },
        \end{split}
    \end{equation}
    where
         $$f'_3(x,t) =  \int_M \left( 1 + \frac{d(x,y) }{ \sqrt{t_1 - t} } \right)\frac{H(x,y,t_1 - t) }{ \rho(y) } |h(y,t)|^q \, dV_y$$
     and
     $$f''_3(x,t) = \frac{1}{\sqrt{t_1}} \int_M \left( 1 + \frac{d(x,y) }{ \sqrt{t_1 - t} } \right) H(x,y,t_1 - t)  |h(y,t)|^q \, dV_y.$$
    For $f_3',$ we apply Lemma \ref{lemma:youngsconvolution}$b$ with $\alpha = \frac12, \beta = 1,$ $b = c = \frac{p}{q},$ and $a = 1,$ to get
    $$\|f_3'(\cdot, t) \|_{L^{\frac{p}{q}}(M \setminus B_R)} \leq \frac{C \|h(t) \|_{L^p(M)}^q }{ \left( R^2 + t_1 -t \right)^{ \frac{1 - \frac{q}{p}}{2} } } \leq  \frac{C \|h(t) \|_{L^p(M)}^q }{ R^{1 - \frac{q}{p} } } \qquad {\text{for } R \ge R_0}.$$
    For $f_3'',$ we apply Lemma \ref{lemma:youngsconvolution}$b$, $\alpha = \frac12, \beta = 0,$ $b = c = \frac{p}{q},$ and $a = 1,$ to get
    $$\|f_3''(\cdot, t) \|_{L^{\frac{p}{q}}(M \setminus B_R)} \leq \frac{C \|h(t) \|_{L^p(M)}^q }{ \sqrt{t_1} } \qquad {\text{for } R \ge R_0}.$$
    Returning to (\ref{u3Lpfirstest}), we get
   \begin{equation}\label{u3Lpest}
       \begin{split}
           \| u_3 \|_{L^{p}(M \setminus B_R)} \leq C \eps \sup_{t_1 - 1 \leq t \leq t_1} \|h(t) \|_{L^p(M)} \left( \frac{1}{R^{\frac{1}{q} - \frac{1}{p} }} + \frac{1}{t_1^{\frac{1}{2q}} } \right).
       \end{split}
   \end{equation}

       \vspace{2mm}
       
     \noindent {\bf Fourth term.} \emph{Pointwise estimate.} Applying the Peter-Paul inequality on $u_4,$ we have
       \begin{equation}\label{u4firstpointwiseexpression}
       \begin{split}
           u^q_4(x) & \leq C \int_0^{t_1 - 1} \!\!\!\! \int_M \left( \frac{1}{t_1 - t} + \frac{d(x,y)^2  }{(t_1 - t)^{2} } \right) H(x,y,t_1 - t) |h(y,t)|^{q + 2}  \, dV_y dt \\
           & = C \int_0^{t_1 - 1} \!\!\!\! f_4(x,t) \frac{dt}{t_1 - t},
           \end{split}
       \end{equation}
       where
       $$f_4(x,t) = \int_M \left( 1 + \frac{d(x,y)^2  }{t_1 - t } \right) H(x,y,t_1 - t) |h(y,t)|^{q + 2}  \, dV_y.$$
       For the pointwise estimate, we write
       \begin{equation*}
       \begin{split}
           f_4(x,t) & \leq \eps^{q + 1} \int_M \left( 1 + \frac{d(x,y)^2  }{t_1 - t } \right) H(x,y,t_1 - t) \left( \frac{1}{\rho(y)} + \frac{1}{\sqrt{1 + t}} \right)^{\ell(q + 1)} |h(y,t)| \, dV_y \\
           & \leq C \eps^{q + 1} \int_M \left( 1 + \frac{d(x,y)^2  }{t_1 - t } \right) H(x,y,t_1 - t) \left( \frac{1}{\rho(y)^{\ell(q + 1)}} + \frac{1}{(1 + t)^{\frac{\ell}{2} \left( q + 1 \right)}} \right) |h(y,t)| \, dV_y \\
           & =: I + II.
           \end{split}
           \end{equation*}
    To estimate $I,$ choose $p'$ large-enough that $q\ell + 1 < \frac{n(p' - 1)}{p'}.$ We may apply Lemma \ref{lemma:youngsconvolution}$a$ with $a = \frac{p'}{p' - 1},$ $b = p',$ and $\beta = q\ell + \min\{\ell, 1\},$ to get
    \begin{equation*}
        I \leq \frac{C \eps^{q + 2} }{(t_1 - t)^{\frac{n}{2p'}}(\rho(x) + t_1 - t)^{q \ell + \min \{ \ell, 1 \} } }.
        \end{equation*}
    For $II,$ applying Lemma \ref{lemma:heatkernelLpests} with $a = \frac{p}{p - 1},$ $b = p$ and $\beta = 0,$ we have
    \begin{equation*}
       \begin{split}
           II & \leq \frac{C \eps^{q + 2} }{(t_1 - t)^{\frac{n}{2p}} (1 + t)^{\frac{\ell}{2} \left( q + 1 \right)}}.
           \end{split}
           \end{equation*}
Inserting into (\ref{u4firstpointwiseexpression}) and integrating, we get
    \begin{equation}\label{u4secondpointwiseexpression}
        u_4(x)^q \leq C \eps^{q + 2}  \left( \frac{1}{\rho(x)^{q\ell + \min \{ \ell, 1 \} } } + \int_0^{t_1 - 1}\frac{dt}{(t_1 - t)^{1 + \frac{n}{2p} }(1 + t)^{\frac{\ell}{2} \left( q + 1 \right)}}  \right).
    \end{equation}
    We have
    \begin{equation}\label{u4integraldumbcalc}
    \begin{split}
    \int_0^{t_1 - 1}\frac{dt}{(t_1 - t)^{1 + \frac{n}{2p} }(1 + t)^{\frac{\ell}{2} \left( q + 1 \right)}} & \leq \frac{2^{1 + \frac{n}{2p} }}{(1 + t_1)^{1 + \frac{n}{2p} }}\int_0^{\frac{t_1-1}{2}}\frac{dt}{(1 + t)^{\frac{\ell}{2} \left( q + 1 \right)}} \\
    & \qquad \qquad + \frac{2^{\frac{\ell}{2} \left( q + 1 \right)} }{(1 + t_1)^{\frac{\ell}{2} \left( q + 1 \right)}} \int_{\frac{t_1 - 1}{2}}^{t_1 - 1}\frac{dt}{(t_1 - t)^{1 + \frac{n}{2p} }} \\
    & \leq \frac{C}{t_1^{\min\{1 + \frac{n}{2p} ,\frac{\ell}{2} \left( q + 1 \right)\}} }.
    \end{split}
    \end{equation}
    Applying this in (\ref{u4secondpointwiseexpression}) and taking $q$-th roots, we get
        \begin{equation}\label{u4xest}
        u_4(x) \leq C \eps^{1 + \frac{2}{q} } \left( \frac{1}{\rho(x)} + \frac{1}{\sqrt{t_1}} \right)^{\min\left\{ \ell \left( 1 + \frac{1}{q} \right), \frac{2}{q} + \frac{n}{pq} \right\} }.
    \end{equation}
    
   \noindent \emph{$L^p$-estimate.} To estimate the $L^p$-norm, we instead write
    \begin{equation*}
       \begin{split}
           f_4(x,t) & \leq \eps^{2 - r} \int_M \left( 1 + \frac{d(x,y)^2  }{t_1 - t } \right) H(x,y,t_1 - t) \left( \frac{1}{\rho(y)} + \frac{1}{\sqrt{1 + t}} \right)^{(2 - r)\ell} |h(y,t)|^{q + r} \, dV_y \\
           & \leq C \eps^{2 - r} \int_M \left( 1 + \frac{d(x,y)^2  }{t_1 - t } \right) H(x,y,t_1 - t) \left( \frac{1}{\rho(y)^{(2 - r) \ell} } + \frac{1}{(1 + t)^{\frac{(2 - r)\ell}{2}}} \right) |h(y,t)|^{q + r} \, dV_y.
           \end{split}
           \end{equation*}
    Applying Minkowski's inequality, we have
     \begin{equation}\label{u4Lpfirstest}
     \begin{split}
        \| u_4\|^q_{L^{p}(M \setminus B_R)} = \| u_4^q(x) \|_{L^{\frac{p}{q}}(M \setminus B_R)} & \leq \left( \int_{M \setminus B_R} \left( \int_{0}^{t_1 - 1} \frac{f_4(x,t)}{t_1 - t} dt \right)^{\frac{p}{q}} \, dV_x \right)^{\frac{q}{p}} \\
        & \leq \int_{0}^{t_1 - 1} \|f_4(\cdot, t) \|_{L^{\frac{p}{q}}(M \setminus B_R)} \, \frac{dt}{t_1 - t} \\
        & \leq C \eps^{2 - r} \int_{0}^{t_1 - 1} \left( \|f_4'(\cdot, t) \|_{L^{\frac{p}{q}}(M \setminus B_R)} + \|f_4''(\cdot, t) \|_{L^{\frac{p}{q}}(M \setminus B_R)} \right) \, \frac{dt}{t_1 - t },
        \end{split}
    \end{equation}
    where
         $$f'_4(x,t) =  \int_M \left( 1 + \frac{d(x,y)^2  }{t_1 - t } \right) \frac{H(x,y,t_1 - t) }{ \rho(y)^{(2 - r)\ell } } |h(y,t)|^{q + r} \, dV_y$$
     and
     $$f''_4(x,t) = \frac{1}{(1 + t)^{\frac{\ell}{2}} } \int_M \left( 1 + \frac{d(x,y)^2  }{t_1 - t } \right) H(x,y,t_1 - t)  |h(y,t)|^{q + r} \, dV_y.$$
          For $f_4',$ we apply Lemma \ref{lemma:youngsconvolution}$b$ with $\alpha = 1, \beta = \ell,$ $c = \frac{p}{q},$ $b = \frac{p}{q + r},$ and $a = \frac{p}{p-r},$ to get
    $$\|f_4'(\cdot, t) \|_{L^{\frac{p}{q}}(M \setminus B_R)} \leq \frac{C \|h(t) \|_{L^p(M)}^{q + r} }{ (t_1 - t)^{\frac{n r}{2p} } \left( R^2 + t_1 -t \right)^{ \frac{ \left( p - q - r \right) \ell}{2(p-r)} } }.$$
    Here (\ref{2-relllessthan...}) guarantees $\beta < \frac{n}{a},$ which allows us to apply Lemma \ref{lemma:youngsconvolution}$b$.
    
    For $f_4'',$ we apply Lemma \ref{lemma:youngsconvolution}$b$ with $\alpha = 1, \beta = 0,$ and the same $a,b,c,$ to get
    $$\|f_3''(\cdot, t) \|_{L^{\frac{p}{q}}(M \setminus B_R)} \leq \frac{C \|h(t) \|_{L^p(M)}^{q + r} }{ (1 + t)^{\frac{(2 - r)\ell}{2}} (t_1 - t)^{\frac{n r}{2p} }  } .$$
    We can estimate
    $$\int_{0}^{t_1 - 1} \|f_4'(\cdot, t) \|_{L^{\frac{p}{q}}(M \setminus B_R)}   \, \frac{dt}{t_1 - t } \leq \frac{C \sup_{0 \leq t \leq t_1 - 1} \|h(t) \|_{L^p(M)}^{q + r} }{ R^{ \frac{ \left( p - q - r \right) \ell}{(p-r)} } }$$
    and, as in (\ref{u4secondpointwiseexpression}-\ref{u4integraldumbcalc}),
    \begin{equation*}
        \begin{split}\int_{0}^{t_1 - 1} \|f_4''(\cdot, t) \|_{L^{\frac{p}{q}}(M \setminus B_R)} \, \frac{dt}{t_1 - t } & \leq C \sup_{0 \leq t \leq t_1 - 1} \|h(t) \|_{L^p(M)}^{q + r} \int_0^{t_1 - 1} \frac{dt}{ (1 + t)^{\frac{(2 - r)\ell}{2}} (t_1 - t)^{1 + \frac{n r}{2p} }  } \\
        & \leq \frac{C \sup_{0 \leq t \leq t_1 - 1} \|h(t) \|_{L^p(M)}^{q + r} }{t_1^{\min\{1 + \frac{n r}{2p} , \frac{(2 - r)\ell}{2} \}}}.
        \end{split}
        \end{equation*}
        Returning to (\ref{u4Lpfirstest}) and taking $q$-th roots, we get
        \begin{equation}\label{u4Lpest}
     \begin{split}
        \| u_4\|_{L^{p}(M \setminus B_R)} & \leq C \eps^{1 + \frac{2}{q}} \left( \frac{1}{R^{ \frac{ \left( p - q - r \right) \ell}{q(p-r)} }} + \frac{1}{t_1^{\frac{\min\{1 + \frac{nr}{2p} , \frac{(2 - r)\ell}{2} \} }{q} } } \right).
        \end{split}
    \end{equation}
Finally, by (\ref{hlessthanu1tou4}), the pointwise estimate (\ref{Lpbootstrap:Linftybound}) follows by combining (\ref{u1xest}), (\ref{u2xest}), (\ref{u3xest}), and (\ref{u4xest}). The $L^p$-estimate (\ref{Lpbootstrap:Lpbound}) follows by combining (\ref{u1Lpest}), (\ref{u2Lpest}), (\ref{u3Lpest}), and (\ref{u4Lpest}).
\end{proof}

\vspace{2mm}

\subsection{Estimate with $p > 1$ for a convergent solution}\label{ss:smallpproofofLpprop}

We need the following variant of Lemma \ref{lemma:Duhamel}.

\begin{lemma}\label{lemma:smallpDuhamel} 
{Let $(M^n, g_0)$ be a Ricci-flat ALE manifold} and let $g_\infty \in \mathcal{F}$ be a metric with $\calM(g_\infty) = 0$ and 
\begin{equation}\label{gaugedricciflatclose}
\|g_\infty - g_0 \|_{C^{1,0}_{-n}} < \eps.
\end{equation}
Fix $q > 1.$ Let $g(t)$ solve Ricci-DeTurck flow with respect to $g_0,$ and write
$$\tilde{h}(t) = g(t) - g_\infty.$$
Also write
$|\tilde{h}| = |\tilde{h}|_{g_\infty}$ and $\nabla = \nabla^{g_\infty}.$ Let $\tilde{H}(x,y,t)$ be the heat kernel with respect to $g_\infty.$ Suppose that (\ref{gg0lambdaassumption}) is satisfied, where $\lambda$ is sufficiently close to one {(depending only on $n, q$)}.
Then
        \begin{equation*}
        \begin{split}
        |\tilde{h}|^q(x,t) & \leq \int_M \tilde{H}(x,y,t) |\tilde{h}|^q(y,0) \, dV_y \\
        & + C_{\ref{lemma:smallpDuhamel}} \int_0^t \!\!\!\! \int_M \tilde{H}(x,y,t - s) \left( \rho(x)^{-2 - \tau} |\tilde{h}|^q + \left( \frac{1}{\sqrt{t - s}} + \frac{d(x,y)}{t-s} \right) |\tilde{h}|^{\frac{q}{2} + 1} |\nabla |\tilde{h}|^{\frac{q}{2}} | - c_{\ref{lemma:smallpDuhamel}} |\nabla |\tilde{h}|^{\frac{q}{2}} |^2 \right) \, dV_y ds.
        \end{split}
    \end{equation*}
    Here $C_{\ref{lemma:smallpDuhamel}}, c_{\ref{lemma:smallpDuhamel}}$ are constants that can be taken to depend only on $M, g_0, q$.
\end{lemma}
\begin{proof}
    Observe that
    \begin{equation}
    \begin{split}
    V(g,g_\infty) - V(g,g_0) & = g^{-1} \left( \Gamma(g_0) - \Gamma(g_\infty) \right) \\
    & = \cancelto{0}{g_\infty^{-1} \left( \Gamma(g_0) - \Gamma(g_\infty) \right)} + \tilde{h} \# \left( \Gamma(g_0) - \Gamma(g_\infty) \right) \\
    & = \eps O(\rho(x)^{-\tau-1} ) \# \tilde{h}.
    \end{split}
    \end{equation}
    Here we have used the fact that $V(g_0,g_\infty) = 0.$
    We therefore have
    \begin{equation*}
    \begin{split}
   \frac{\p \tilde{h}}{\p t} = \frac{\p g}{\p t} & = \mathcal{M}(g) \\
   & = \calM_{g_\infty}(g) + \scrL_{V(g,g_0) - V(g,g_\infty)} g \\
   & = \calM_{g_\infty}(g) + \eps \left( O(\rho(x)^{-\tau-2} ) \# \tilde{h} + O(\rho(x)^{-\tau-1} ) \# \nabla \tilde{h} \right).
   \end{split}
    \end{equation*}
    Arguing as in (\ref{evolutionof|h|^2}), we have
         \begin{equation}\label{firstevolutionoftildeh2}
       \begin{split}
           \left( \frac{\p}{\p t} - \Delta_{g, g_\infty} \right) |\tilde{h}|_{g_\infty}^2 & \leq - 2 |\nabla \tilde{h}|_{g,g_\infty}^2 + 4 \lambda^4 \left( (1 + \eps ) O(\rho(x)^{-2 -\tau}) \right) |\tilde{h}|^2 \\
           & \qquad + \frac{9}{2} \lambda^8 |\tilde{h}| |\nabla \tilde{h}|_{g,g_infty}^2 + \eps O(\rho(x)^{-1 - \tau}) |\tilde{h} | | \nabla \tilde{h} |.
       \end{split}
   \end{equation}
   Applying Young's inequality, we have
            \begin{equation}
       \begin{split}
          \eps O(\rho(x)^{-1 - \tau}) |\tilde{h} | | \nabla \tilde{h} | \leq O(\rho(x)^{-2 - \tau}) + \eps^2 |\nabla \tilde{h} |^2.
       \end{split}
   \end{equation}
   Returning to (\ref{firstevolutionoftildeh2}), we have
            \begin{equation}
       \begin{split}
           \left( \frac{\p}{\p t} - \Delta_{g, g_\infty} \right) |\tilde{h}|_{g_\infty}^2 & \leq 
            \left( - 2 +  \frac{9}{2} \lambda^8 |\tilde{h}| + \eps^2 \right) |\nabla \tilde{h}|_{g,g_0}^2 + 4 \lambda^4 O(\rho(x)^{-2-\tau}) |\tilde{h}|^2,
       \end{split}
   \end{equation}
    The formula follows by redoing the argument of Lemma \ref{lemma:hpevolution} and using Duhamel's principle as in Lemma \ref{Duhamel}, with $g_\infty$ in place of $g_0.$
\end{proof}

\begin{prop}\label{prop:smallpLpbootstrap} 
Let $(M, g_0)$ be a Ricci-flat ALE space of dimension and order $\tau = n \geq 4.$ Let $1 < q < p$ and $\frac{2}{q} < m < n - 2.$ Let $t_1 \geq 1$ and $0 < \eps < \lambda - 1,$ where $\lambda$ is sufficiently close to one as in Lemma \ref{lemma:hpevolution}. 

Assume that $g(t)$ solves the Ricci-DeTurck flow with respect to $g_0.$ Suppose further that there exists a gauged Ricci-flat metric $g_\infty$ satisfying (\ref{gaugedricciflatclose}) such that, writing
$$\tilde{h}(t) = g(t) - g_\infty,$$
we have
      \begin{equation}\label{smallpLpbootstrap:hLpestimate}
        \sup_{0 \leq t < t_1} \| \tilde{h}(t) \|_{L^p(M)} \leq \eps
    \end{equation}
     and
\begin{equation}\label{smallpLpbootstrap:hLinftyassn}
        | \tilde{h}(t) | \leq \frac{ \eps }{ (1 + t)^{\frac{m}{2}} }.
    \end{equation}
For any $2 < \beta < n,$ we then have
\begin{equation}\label{smallpLpbootstrap:Linftybound}
\begin{split}
    |\tilde{h}|(x,t_1) & \leq C \left( t_1^{-\frac{n}{2p}} \| \tilde{h}(0) \|_{L^p} + \eps \left( \frac{1}{ \left( \rho(x)^2 + t_1 \right)^{\frac{\beta}{2q}} }  + \frac{1}{t_1^{\frac{m}{2}} \rho(x)^{\frac{\beta - 2}{q}}} \right) \right) \\
    & \qquad + C \left( \frac{\eps^{1 + \frac{1}{q}} }{t_1^{\frac{(q + 1)m + 1}{2q}} }  + \frac{ \eps^{1 + \frac{2}{q} }}{t_1^{\min\{ \frac{m}{q}, \frac{n}{2p} + \frac{1}{q} \}}} \right),
    \end{split}
\end{equation}
and, for $R \geq 1,$
        \begin{equation}\label{smallpLpbootstrap:Lpbound}
        \begin{split}
        \|\tilde{h}(t_1) \|_{L^p(M \setminus B_R)} & \leq \|\tilde{h}(0) \|_{L^p ( M \setminus B_R )} + \min \left\{ 1, \left( \frac{C R^2}{t_1 } \right)^{\frac{n}{2}\left( \frac{1}{q} - \frac{1}{p} \right)} \right\} \|\tilde{h}(0)\|_{L^p(B_R)} \\
        & \qquad + C \eps \left( \frac{1}{R^{\left( \frac{1}{q} - \frac{1}{p} \right) \beta } } + \frac{\eps^{\frac{1}{q}}}{ t_1^{\frac{m + 1}{2q}} } + \frac{\eps^{\frac{2}{q}} }{t_1} \right).
        \end{split}
    \end{equation}
\end{prop}
\begin{proof}

In the formula of Lemma \ref{lemma:Duhamel}, we break up the right-hand side as follows:
\begin{equation}\label{smallpLpbootstrap:RHSsplitting}
    \begin{split}
        u_1^q(x) & := \int_M \tilde{H}(x,y,t_1) |\tilde{h}|^q(y,0) \, dV_y \\
        u_2^q(x) & := \int_0^{t_1} \!\!\!\! \int_M \tilde{H}(x,y,t_1 - t) |\Rm(g_0)| |\tilde{h}(y,t)|^q \, dV_y dt \\
        u_3^q(x) & := \int_{t_1 - 1}^{t_1} \!\! \int_M \tilde{H}(x,y,t_1 - t) \left( \frac{1}{\sqrt{t_1-t}} + \frac{d(x,y)}{t_1-t} \right) |\tilde{h}|^{\frac{q}{2} + 1} |\nabla |\tilde{h}|^{\frac{q}{2}} | \, dV_y dt \\
        u_4^q(x) & := \int_0^{t_1 - 1} \!\!\!\! \int_M \tilde{H}(x,y,t_1 - t) \left( \left( \frac{1}{\sqrt{t_1-t}} + \frac{d(x,y)}{t_1 - t}  \right) |\tilde{h}|^{\frac{q}{2} + 1} |\nabla |\tilde{h}|^{\frac{q}{2}} | - c_q |\nabla |\tilde{h}|^{\frac{q}{2}} |^2 \right) \, dV_y dt. 
    \end{split}
\end{equation}
According to Lemma \ref{lemma:Duhamel}, we have
\begin{equation}\label{smallphlessthanu1tou4}
    |\tilde{h}|(x,t_1) \leq u_1(x) + C \left( u_2(x) + u_3(x) + u_4(x) \right).
    \end{equation}
    We now use Lemma \ref{lemma:youngsconvolution} to estimate $u_i$ pointwise and in $L^{p}(M \setminus B_R),$ for $i = 1, \ldots, 4.$

       \vspace{2mm}
       
       \noindent {\bf First term.} Same as Proposition \ref{prop:Lpbootstrap}.

       \vspace{2mm}

     \noindent {\bf Second term.}   \emph{Pointwise estimate.}  Note that since $(M,g_0)$ is Ricci-flat ALE of order $\tau > 0,$ Lemma \ref{lemma: Rm est for ALE} implies
    \begin{equation*}
        u^q_2(x) \leq C \int_0^{t_1} \!\!\!\! \int_M \frac{\tilde{H}(x,y,t_1 - t)}{ \rho(y)^{2 + n} } |\tilde{h}(y,t)|^q \, dV_y dt.
    \end{equation*}
     Let
     $$f_2(x,t) = \int_M \frac{\tilde{H}(x,y,t_1 - t) }{ \rho(y)^{2 + n} } |\tilde{h}(y,t)|^q \, dV_y.$$
Applying (\ref{smallpLpbootstrap:hLinftyassn}) and Lemma \ref{lemma:heatkernelLpests}, we have
$$|f_2(x,t)| \leq \frac{C \eps^q }{(1 + t)^{\frac{qm}{2}} \left( \rho(x)^2 + t_1 - t \right)^{\frac{\beta}{2}} }.$$
Since $\frac{qm}{2} > 1,$ this gives
\begin{equation}
\begin{split}
    u_2^q 
    & \leq C\eps^q \left( \int_0^{t_1/2} + \int_{t_1/2}^{t_1}\right)\frac{dt}{(1 + t)^{\frac{qm}{2}} \left( \rho(x)^2 + t_1 - t \right)^{\frac{\beta}{2}} } \\
    & \leq \frac{C\eps^q }{ \left( \rho(x)^2 + t_1 \right)^{\beta/2}} \int_0^{t_1/2} \frac{dt}{(1 + t)^{\frac{qm}{2}}} + \frac{C \eps^q }{t_1^{\frac{qm}{2}}} \int_{0}^{t_1/2} \frac{d\tau}{\left( \rho(x)^2 + \tau \right)^{\frac{\beta}{2}} } \\
    & \leq \frac{C\eps^q}{ \left( \rho(x)^2 + t_1 \right)^{\beta/2}}  + \frac{C \eps^q}{t_1^{\frac{qm}{2}}} \frac{\min\{ t_1, \rho(x)^2 \}}{\rho(x)^\beta} \\
    & \leq \frac{C\eps^q}{ \left( \rho(x)^2 + t_1 \right)^{\beta/2}}  + \frac{C \eps^q}{t_1^{\frac{qm}{2}} \rho(x)^{\beta - 2}} \frac{t_1}{\rho(x)^2 + t_1}.
    \end{split}
\end{equation}

\noindent \emph{$L^p$-estimate.} We first apply Minkowski's integral inequality:
     \begin{equation}\label{smallpu2Lpfirstest}
     \begin{split}
        \| u_2\|^q_{L^{p}(M \setminus B_R)} = \| u_2^q(x) \|_{L^{\frac{p}{q}}(M \setminus B_R)} & \leq C \left( \int_{M \setminus B_R} \left( \int_0^{t_1} f_2(x,t) dt \right)^{\frac{p}{q}} \, dV_x \right)^{\frac{q}{p}} \\
        & \leq C \int_0^{t_1} \|f_2(\cdot, t) \|_{L^{\frac{p}{q}}(M \setminus B_R)} \, dt.
        \end{split}
    \end{equation}
    Applying Lemma \ref{lemma:youngsconvolution}$b$, with $b = c = \frac{p}{q},$ $a = 1,$ and $\beta < n = \frac{n}{a}$ as given, 
    we get
    \begin{equation}
    \begin{split}
    \|f_2(\cdot, t) \|_{L^{\frac{p}{q}}(M \setminus B_R)} & \leq \frac{C \|\tilde{h}(t) \|_{L^p(M)}^q }{ \left( R^2 + t_1 -t \right)^{ \frac{ \left( 1 - \frac{q}{p} \right) \beta}{2} } } \\
    & \leq \frac{C \eps^q }{ (1 + t)^{\frac{qm}{2}} \left( R^2 + t_1 -t \right)^{ \frac{ \left( 1 - \frac{q}{p} \right) \beta}{2} } } \\
    & \leq \frac{C \eps^q }{ (1 + t)^{\frac{qm}{2}} R^{\left( 1 - \frac{q}{p} \right) \beta } }.
    \end{split}
    \end{equation}
    Integrating in (\ref{smallpu2Lpfirstest}) and taking $q$-th roots, we get
   \begin{equation}\label{smallpu2Lpest}
       \begin{split}
           \| u_2 \|_{L^{p}(M \setminus B_R)} \leq \frac{C \eps }{ R^{ \left( \frac{1}{q} - \frac{1}{p} \right)\beta } }.
       \end{split}
   \end{equation} 

       \vspace{2mm}
       
     \noindent {\bf Third term.} \emph{Pointwise estimate.} Note that (\ref{smallpLpbootstrap:hLinftyassn}) implies
     \begin{equation}\label{smallpLpbootstrap:gradassumption}
     |\nabla \tilde{h}(t)| \leq \frac{C \eps}{t_1^{\frac{m + 1}{2}} }
     \end{equation}
    for $t_1 - 1 \leq t \leq t_1.$ We have
    $$|\tilde{h}|^{\frac{q}{2} + 1} |\nabla |\tilde{h}|^{\frac{q}{2}} | = |\tilde{h}|^{q}  | \nabla \tilde{h}| \leq \frac{C \eps^{q + 1} }{t_1^{\frac{(q + 1)m + 1}{2}} }.$$
    This gives
    \begin{equation}\label{smallpu3firstbound}
    \begin{split}
       u^q_3(x) & \leq \frac{C \eps^{q + 1} }{t_1^{\frac{(q + 1)m + 1}{2}} } \int_{t_1 - 1}^{t_1} \!\! \int_M \left( \frac{1}{\sqrt{t_1 - t}} + \frac{d(x,y)}{t_1 - t} \right) \tilde{H}(x,y,t_1 - t) dV_y dt \leq \frac{C \eps^{q + 1} }{t_1^{\frac{(q + 1)m + 1}{2}} }.
       \end{split}
       \end{equation}
Taking $q$-th roots, we get
    \begin{equation}\label{smallpu3xest}
        u_3(x) \leq \frac{C \eps^{1 + \frac{1}{q} } }{t_1^{\frac{(q + 1)m + 1}{2q}} }.
    \end{equation}
    
    \noindent \emph{$L^p$-estimate.} Arguing as in Proposition \ref{prop:Lpbootstrap}, we have only the term $f_3''(x,t).$ In view of (\ref{smallpLpbootstrap:gradassumption}), we have
     $$f''_3(x,t) = \frac{C \eps}{t_1^{\frac{m + 1}{2} } } \int_M \left( 1 + \frac{d(x,y)  }{ \sqrt{t_1 - t}  } \right) \tilde{H}(x,y,t_1 - t) |\tilde{h}(y,t)|^q \, dV_y$$
    for $t_1 -1 \leq t \leq t_1.$ We apply Lemma \ref{lemma:youngsconvolution}$b$, $\alpha = \frac12, \beta = 0,$ $b = c = \frac{p}{q},$ and $a = 1,$ to get
    $$\|f_3''(\cdot, t) \|_{L^{\frac{p}{q}}(M)} \leq \frac{C \eps^{q + 1} }{ t_1^{\frac{m + 1}{2} } }.$$
   Integrating in time, we get
   \begin{equation}\label{smallpu3Lpest}
       \begin{split}
           \| u_3 \|_{L^{p}(M \setminus B_R)} \leq C \frac{\eps^{1 + \frac{1}{q}}}{t_1^{\frac{m + 1}{2q} }}.
       \end{split}
   \end{equation}

       \vspace{2mm}
       
     \noindent {\bf Fourth term.} \emph{Pointwise estimate.} Applying the Peter-Paul inequality on $u_4,$ we have
       \begin{equation}\label{smallpu4firstpointwiseexpression}
       \begin{split}
           u^q_4(x) & \leq C \int_0^{t_1 - 1} \!\!\!\! \int_M \left( \frac{1}{t_1 - t} + \frac{d(x,y)^2  }{(t_1 - t)^{2} } \right) \tilde{H}(x,y,t_1 - t) |\tilde{h}(y,t)|^{q + 2}  \, dV_y dt \\
           & = C \int_0^{t_1 - 1} \!\!\!\! f_4(x,t) \frac{dt}{t_1 - t},
           \end{split}
       \end{equation}
       where
       \begin{equation*}
       \begin{split}
       f_4(x,t) & = \int_M \left( 1 + \frac{d(x,y)^2  }{t_1 - t } \right) \tilde{H}(x,y,t_1 - t) |\tilde{h}(y,t)|^{q + 2}  \, dV_y \\
           & \leq \frac{\eps^{2}}{(1 + t)^m} \int_M \left( 1 + \frac{d(x,y)^2 }{t_1 - t } \right) \tilde{H}(x,y,t_1 - t)  |\tilde{h}(y,t)|^q \, dV_y.
           \end{split}
           \end{equation*}
    Applying Lemma \ref{lemma:heatkernelLpests} with $a = \frac{p}{p - q}$ and $b = \frac{p}{q},$ we get
    \begin{equation*}
        \|f_4\|_{L^{\frac{p}{q}}} \leq \frac{C \eps^{2 + q}}{(1 + t)^m(t_1 - t)^{\frac{nq}{2p}} }.
        \end{equation*}
        Returning to (\ref{smallpu4firstpointwiseexpression}), we get
        \begin{equation}
            |u_4^q| \leq C \eps^{2 + q} \int_0^{t_1 - 1} \frac{dt}{(1 + t)^m (t_1 - t)^{\frac{nq}{2p} + 1} } \leq \frac{C \eps^{2 + q}}{t_1^{\min\{ m, \frac{nq}{2p} + 1 \} } }.
        \end{equation}
    
   \noindent \emph{$L^p$-estimate.} Applying Lemma \ref{lemma:youngsconvolution}$b$ with $a = 1$ and $b = c = \frac{p}{q},$ we have
   \begin{equation*}
       \|f_4 \|_{L^{\frac{p}{q}}} \leq \frac{C \eps^{2}}{(1 + t)^m} \|\tilde{h}(t) \|_{L^q}^q \leq \frac{C \eps^{2 + q}}{(1 + t)^m}.
   \end{equation*}
   Integrating, we get
   \begin{equation}
   \begin{split}
            \|u_4 \|_{L^p}^q & \leq C \eps^{2 + q} \int_0^{t_1 - 1} \frac{dt}{t^m (t_1 - t) } \\
            & \leq \frac{C \eps^{2 + q}}{t_1} \int_0^{t_1/2} \frac{dt}{(1 + t)^m} + \frac{C \eps^{2 + q}}{t_1^m} \int_{t_1/2}^{t_1 - t} \frac{dt}{t_1 - 1} \\
            & \leq \frac{C \eps^{2 + q}}{t_1} \left( 1 + \frac{\log t_1}{t_1^{m-1}} \right) \leq  \frac{C \eps^{2 + q}}{t_1}.
            \end{split}
        \end{equation}
Combining all of the above gives the claimed estimates.
\end{proof}

\bibliographystyle{alpha}
\bibliography{bibfile}

\end{document}